\documentclass[twoside,11pt]{article}

\usepackage{blindtext}
\usepackage{graphicx} 
\usepackage{booktabs} 
\usepackage{mathrsfs}
\usepackage{amssymb}
\usepackage{amsmath}
\usepackage{mathtools}
\usepackage{bbding}
\usepackage{pifont}
\usepackage{bm}
\usepackage{amsbsy}
\usepackage{algorithm}
\usepackage{algorithmicx}
\usepackage{algpseudocode}
\usepackage{float}
\usepackage{lipsum}
\usepackage{cases}
\usepackage{bbm}
\usepackage{color}
\usepackage{tabularx}
\usepackage{CJK}
\usepackage{multirow}
\usepackage{setspace}
\usepackage{geometry}
\usepackage{color}
\usepackage{multicol}
\usepackage{extarrows}
\usepackage{subfigure}
\usepackage{tikz}
\usepackage{multicol}
\usepackage{comment}

\definecolor{blue}{rgb}{0,0,0.9}
\definecolor{red}{rgb}{0.9,0,0}
\definecolor{green}{rgb}{0,0.9,0}

\usepackage{jmlr2e}

\newcommand{\floor}[1]{\left\lfloor#1\right\rfloor}
\newcommand{\ceil}[1]{\left\lceil#1\right\rceil}
\newcommand{\be}{\begin{equation}}
\newcommand{\ee}{\end{equation}}
\newcommand{\ben}{\begin{enumerate}}
\newcommand{\een}{\end{enumerate}}
\newcommand{\bfr}{\begin{frame}}
\newcommand{\efr}{\end{frame}}
\newcommand{\bit}{\begin{itemize}}
\newcommand{\eit}{\end{itemize}}

\newcommand{\R}{\mathbb R}
\newcommand{\E}{\mathbb{E}}

\newcommand{\cA}{\mathcal A}
\newcommand{\cB}{\mathcal B}
\newcommand{\cC}{\mathcal C}
\newcommand{\cD}{\mathcal D}
\newcommand{\cE}{\mathcal E}
\newcommand{\cF}{\mathcal F}
\newcommand{\cG}{\mathcal G}

\newcommand{\cL}{\mathcal L}
\newcommand{\cM}{\mathcal M}

\newcommand{\cS}{\mathcal S}

\newcommand{\cU}{\mathcal U}

\newcommand{\wtnabla}{\widetilde\nabla}

\DeclareMathOperator*{\argmin}{argmin}

\def\norm#1{\left\|#1\right\|}
\def\inprod#1#2{\langle#1,\,#2\rangle}

\def\eps{\epsilon}

\def\Mat{{\rm Mat}}

 \newtheorem{assumption}{Assumption}

\def\va{{\bm{a}}}
\def\vb{{\bm{b}}}
\def\vc{{\bm{c}}}
\def\vd{{\bm{d}}}

\def\vg{{\bm{g}}}
\def\vh{{\bm{h}}}

\def\vp{{\bm{p}}}

\def\vu{{\bm{u}}}
\def\vv{{\bm{v}}}
\def\vw{{\bm{w}}}
\def\vx{{\bm{x}}}
\def\vy{{\bm{y}}}
\def\vz{{\bm{z}}}

\def\mA{{\bm{A}}}

\def\mM{{\bm{M}}}

\def\mW{{\bm{W}}}

\usepackage{lastpage}

\ShortHeadings{Nonconvex Stochastic Bregman Proximal Gradient Method with Application to Deep Learning}{Ding, Li, and Toh}
\firstpageno{1}

\begin{document}

\title{Nonconvex Stochastic Bregman Proximal Gradient Method with Application to Deep Learning}

\author{\name Kuangyu Ding \email  kuangyud@u.nus.edu \\
       \addr Department of Mathematics\\
       National University of Singapore\\
       10 Lower Kent Ridge Road, Singapore 119076
       \AND
       \name Jingyang Li \email  li\_jingyang@u.nus.edu \\
       \addr Department of Mathematics\\
       National University of Singapore\\
       10 Lower Kent Ridge Road, Singapore 119076
       \AND
       \name Kim-Chuan Toh \email  mattohkc@nus.edu.sg \\
       \addr Department of Mathematics \\ Institute of Operations Research and Analytics\\
       National University of Singapore\\
       10 Lower Kent Ridge Road, Singapore 119076}

\editor{}

\maketitle

\begin{abstract}
Stochastic gradient methods for minimizing nonconvex composite objective functions typically rely on the Lipschitz smoothness of the differentiable part, but this assumption fails in many important problem classes like quadratic inverse problems and neural network training, leading to instability of the algorithms in both theory and practice. To address this, we propose a family of stochastic Bregman proximal gradient (SBPG) methods that only require smooth adaptivity. SBPG replaces the quadratic approximation in SGD with a Bregman proximity measure, offering a better approximation model that handles non-Lipschitz gradients in nonconvex objectives. We establish the convergence properties of vanilla SBPG and show it achieves optimal sample complexity in the nonconvex setting. Experimental results on quadratic inverse problems demonstrate SBPG's robustness in terms of stepsize selection and sensitivity to the initial point. Furthermore, we introduce a momentum-based variant, MSBPG, which enhances convergence by relaxing the mini-batch size requirement while preserving the optimal oracle complexity. We apply MSBPG to the training of deep neural networks, utilizing a polynomial kernel function to ensure smooth adaptivity of the loss function. Experimental results on benchmark datasets confirm the effectiveness and robustness of MSBPG in training neural networks. Given its negligible additional computational cost compared to SGD in large-scale optimization, MSBPG shows promise as a universal open-source optimizer for future applications.

\end{abstract}


\section{Introduction}\label{sec-introduction}
In this paper, we present and analyze a family of nonconvex stochastic Bregman proximal gradient methods (SBPG) for solving the following generic stochastic minimization problem:
\begin{equation}
\min_{\vx\in \overline{C}}\;\E_{{\bm\xi}}[f(\vx,{\bm\xi})]+R(\vx),
\label{min-prob}
\end{equation}
where $f(\cdot,{\bm\xi})$ is a nonconvex differentiable function on $\overline{C}$, $R$ is a proper lower-semicontinuous convex function, ${\bm\xi}$ is a random variable, and $\overline{C}$ is the closure of $C$, which is a nonempty convex open subset of $\R^d$. We denote $F(\vx)\coloneqq\E_{{\bm\xi}}[f(\vx,{\bm\xi})]$, and $\Phi(\vx)\coloneqq F(\vx)+R(\vx)$. This type of stochastic minimization problem is common in machine learning and statistics \citep{hastie2009elements, shapiro2021lectures, zhang2004solving}, where the optimizer has limited access to the distribution of $\bm\xi$ and can only draw samples from it. In many instances, the smooth part of the objective function $F(\vx)$ can be formulated as a finite-sum structure $F(\vx)=\frac{1}{n}\sum_{i=1}^nf_i(\vx)$. Although the distribution of $\bm\xi$ is known in such cases, when $n$ is extremely large, calculating the true gradient for the smooth part of the objective function becomes extremely expensive. As a result, stochastic first-order methods, originating from the work of \citet{robbins1951stochastic}, have emerged as the prevailing approach for solving these large-scale optimization problems. In particular, stochastic (proximal) gradient descent and its numerous variants \citep{duchi2011adaptive,duchi2009efficient,gu2020unified,kingma2014adam,JMLR:v18:16-410,wang2022riemannian} have been widely used in large-scale stochastic optimization for machine learning \citep{lecun2015deep,shapiro2021lectures,zhang2004solving}. From a modeling perspective, stochastic (proximal) gradient descent can be viewed as minimizing a sequence of upper quadratic approximations of the nonconvex objective $\Phi(\vx)$:
\begin{equation}
{\vx^{k+1}}=\argmin_{\vx\in \overline{C}}\left\{\underbrace{F({\vx^k},{\bm\Xi}_{k})+\inprod{\widetilde\nabla_k}{\vx-{\vx^k}}+\frac{1}{2\alpha_k}\|\vx-{\vx^k}\|^2}_{F_{{\vx^k}}(\vx):\;model\;of\;F\;at\;{\vx^k}}+R(\vx)\right\},
\label{convention-model} 
\end{equation}
where $F(\vx^k,{\bm\Xi}_k)\coloneqq\frac{1}{|{\bm\Xi}_k|}\sum_{{\bm\xi}\in{\bm\Xi}_k}f(\vx^k,{\bm\xi})$, ${\bm\Xi_k}$ is the set of samples of ${\bm\xi}$ at the $k$-th iteration, and $\widetilde{\nabla}_k$ is an estimator of the exact gradient $\nabla F({\vx^k})$. This modeling perspective is well-established in deterministic optimization and has been used in methods such as Newton method, Gauss-Newton method, bundle method, and trust-region method, as discussed in various sources such as \citet{hiriart1993convex,nesterov2003introductory,lin2007trust,paren2022stochastic}.

Despite their widespread use, stochastic gradient methods \eqref{convention-model} face several well-known challenges both in theory and practice. One of the key assumptions underlying the analysis of these methods is the Lipschitz continuity of the gradient of the differentiable part, which is critical for ensuring convergence and establishing complexity results. However, this assumption is not always valid. For instance, even a simple function like $F(x)=x^4$ does not admit a globally Lipschitz gradient over $\R$, illustrating the limitations in analyzing stochastic gradient methods in more general settings. In addition, choosing the appropriate stepsize is another challenge in both the theoretical analysis and the practical usage of stochastic gradient methods. {In practice, the stepsize plays a decisive role—arguably the most important—in determining the convergence behavior of the algorithm. Finding the optimal stepsize can be time-consuming, as engineers often need to conduct numerous experiments to fine-tune it. From a theoretical perspective, choosing a cautious stepsize is necessary to ensure the descent property at each iteration, which is typically proportional to the inverse of the Lipschitz smoothness parameter. As a result, the absence of Lipschitz smoothness can lead to instability in classical stochastic first-order methods, complicating both their theoretical analysis and practical implementation.}

To address these issues, classical approaches often resort to either line search or more complicated inner loops, but these methods can negatively impact the efficiency of the algorithm or even become intractable in a stochastic setting. For example, stochastic proximal point algorithm (PPA) models the approximation of $F(\vx)$ in \eqref{convention-model} as $F_{{\vx^k}}(\vx)=F(\vx,{\bm\xi}_k)+\frac{1}{2\alpha_k}\|\vx-{\vx^k}\|^2$ \citep{bertsekas2011incremental,bianchi2016ergodic,patrascu2017nonasymptotic,rockafellar1976monotone}, making the selection of stepsize $\alpha_k$ more robust than the original model \eqref{convention-model}. However, the application of stochastic PPA is limited due to the difficulty of solving the subproblems, especially for complicated objectives, such as training deep neural networks. In such cases, solving the subproblem is almost as difficult as solving the original problem, rendering the method impractical. Recently, \citet{bauschke2017descent,lu2018relatively} have proposed using Bregman proximity measures to relax the assumption of gradient Lipschitz continuity to smooth adaptivity. The Bregman gradient method was first introduced as the mirror descent scheme by \citet{nemirovskij1983problem} for minimizing convex nonsmooth functions. From the modeling perspective, Bregman methods consider the following subproblem at each iteration:
\begin{equation}
{\vx^{k+1}}=\argmin_{\vx\in \overline{C}}\left\{\underbrace{F({\vx^k},{\bm\Xi}_{k})+\inprod{\widetilde\nabla_k}{\vx-{\vx^k}}+\frac{1}{\alpha_k}\cD_\phi(\vx,{\vx^k})}_{F_{{\vx^k}}(\vx):\;model\;of\;F\;at\;{\vx^k}}+R(\vx)\right\},
\label{model-Bregman}
\end{equation}
where $\cD_\phi$ is the Bregman distance induced by the kernel function $\phi$. To illustrate the advantage of the Bregman proximity model, we present a toy example. Consider the objective function $F(x)=x^4$, which does not admit a globally Lipschitz continuous gradient. We compare the performance of the upper quadratic approximation model \eqref{convention-model} and Bregman proximity model \eqref{model-Bregman}. As shown in Figure \eqref{model-base-figure}(a), the Bregman proximity model \eqref{model-Bregman} ($F_2(x)$) with the kernel function $\phi(x)=\frac{1}{2}x^2+\frac{1}{4}x^4$ can provide a closer approximation to $F(x)$ than the upper quadratic approximation model \ref{convention-model} ($F_1(x)$), as the yellow curve stays closer to the curve of the objective function $F(x)=x^4$. This improved approximation enables the Bregman gradient method to generate ${x^{k+1}}$ that makes more substantial progress toward the optimal solution ($x^*=0$) compared to the gradient descent method, as shown in Figure \ref{model-base-figure}(b).
\begin{figure}
\centering
\subfigure[Two models' approximations of $F(x)$]{
\includegraphics[width=5.5cm,height=4.5cm]{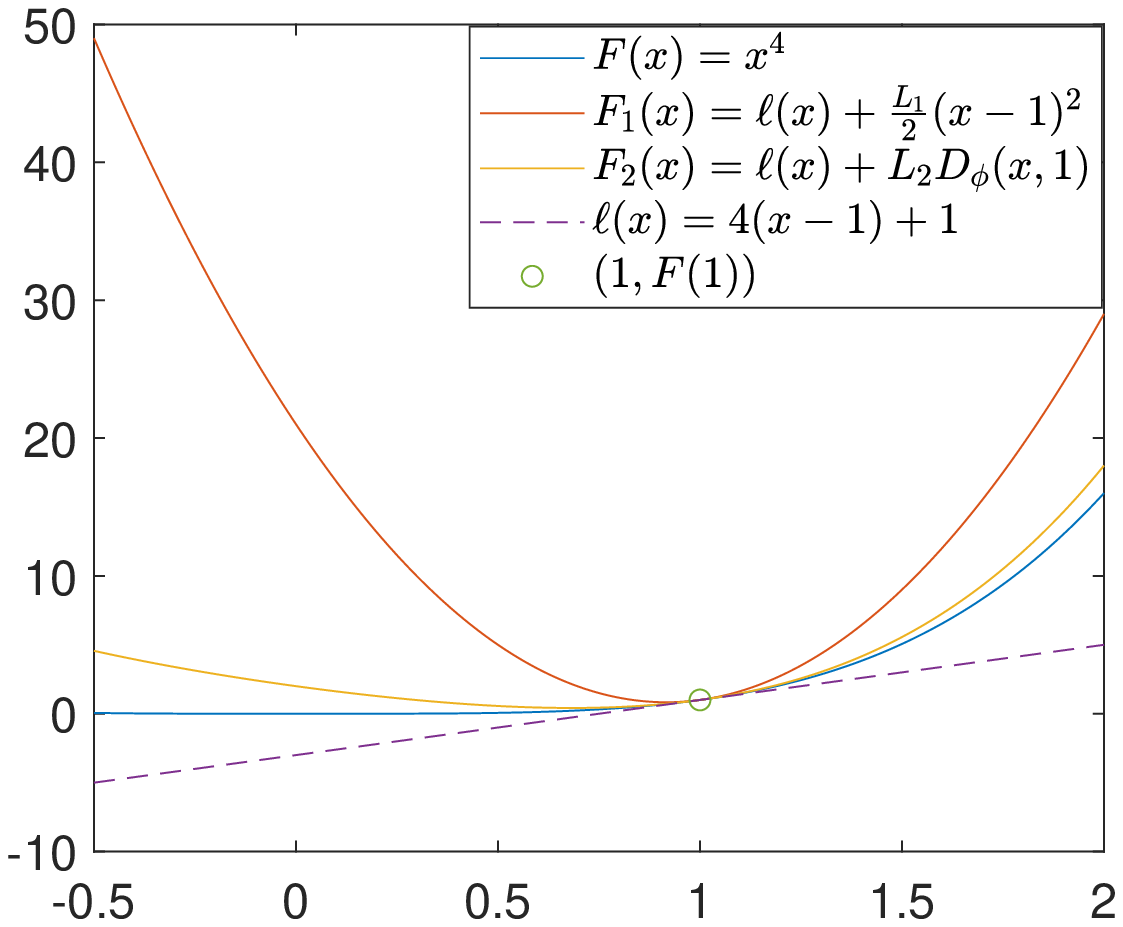}
}
\subfigure[Zoomed-in of version plot (a)]{
\includegraphics[width=5.5cm,height=4.5cm]{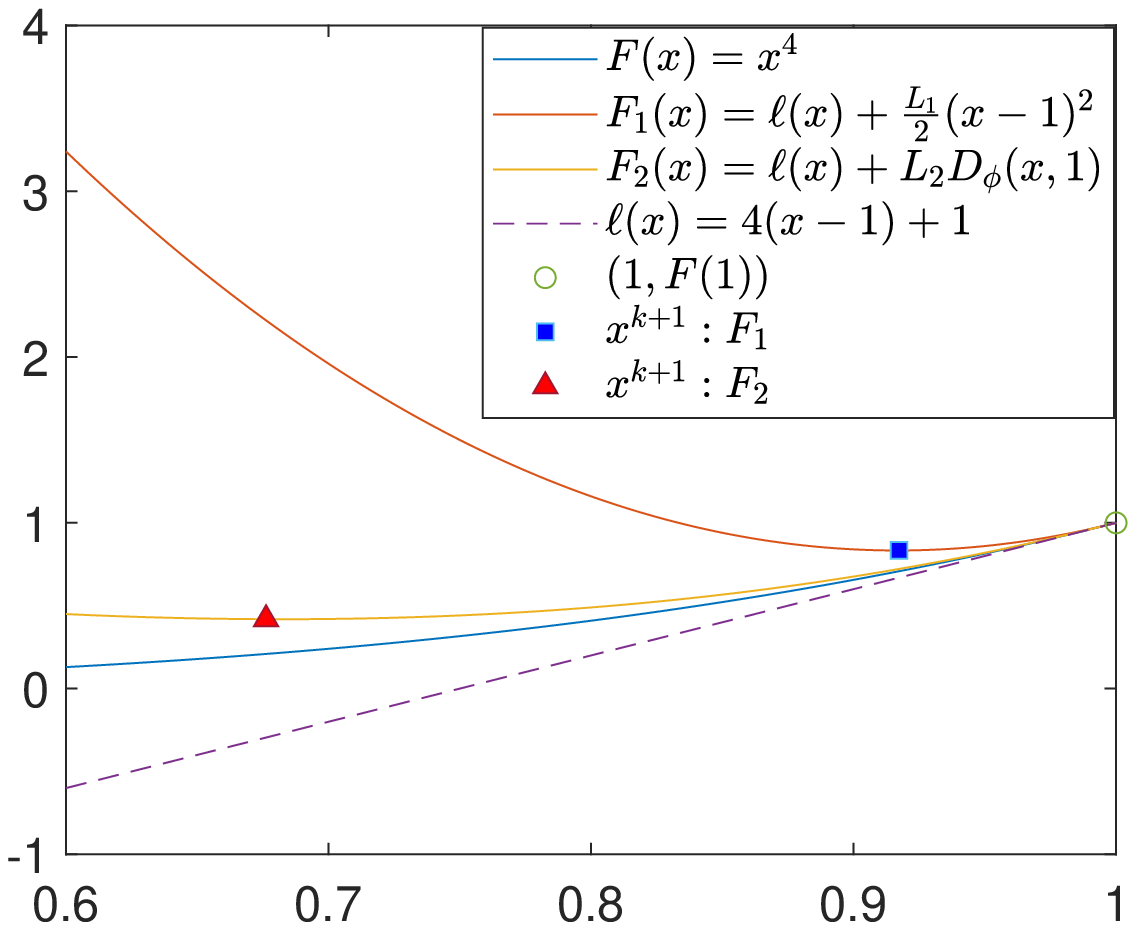}
}
\caption{For function $F(x)=x^4$, which does not admit a globally Lipschitz continuous gradient. We restrict the feasible set to $[-0.5,2]$. Consider the models \eqref{convention-model} and \eqref{model-Bregman} of $F$ at ${x^k}=1$. The Lipschitz smooth constant of $F$ with respect to the kernel $\phi(x)=\frac{1}{2}x^2$ is 48. The smooth adaptivity constant of $F$ with respect to the kernel $\phi(x)=\frac{1}{2}x^2+\frac{1}{4}x^4$ is 4. The figure in (b) is a zoomed-in version of the plot in (a) for the range $[0.6,1]$. The unique minimum of $F(x)$ is at $x=0$.}
\label{model-base-figure}
\end{figure}

While several stochastic extensions of Bregman methods based on smooth adaptivity assumptions have been developed recently, the existing literature primarily focuses on stochastic convex problems \citep{dragomir2021fast, hanzely2021fastest, lu2019relative}. Convergence analyses for Bregman methods in nonconvex settings \citep{latafat2022bregman, wang2023bregman} rely heavily on variance reduction techniques and finite sum structure. However, these approaches are either memory-intensive or require periodic computation of the full gradient, making them impractical for large-scale problems such as deep neural networks as demonstrated in \citep{defazio2019ineffectiveness}.

{Beyond the optimization literature, recent works in learning theory \citep{azizan2019stochastic, li2021implicit, sun2022mirror, sun2023unified} have focused on specific scenarios, such as highly overparameterized models or linear models for classification and regression tasks, examining the implicit bias of Bregman-type methods. For instance, \citep{azizan2019stochastic} demonstrates that in overparameterized models, mirror descent iterates converge to the global minimum closest  to the initial point in terms of Bregman divergence, provided the initial point is near the minimum set. Other works \citep{li2021implicit, sun2022mirror, sun2023unified} explore convergence towards the direction that maximizes the margin in classification problems. Our work, however, focuses on the optimization properties of stochastic Bregman proximal gradient (SBPG) in general nonconvex optimization problems, which presents a different perspective from the implicit bias literature. As a result, we will not review the broader literature on implicit bias in this context.}

As we can see, stochastic Bregman methods have not been fully explored in the general large-scale nonconvex problems such as training neural networks, and rigorous numerical evaluations of their performance are limited. Additionally, the current literature gives insufficient attention to the robustness of stochastic Bregman methods, particularly regarding the selection of stepsizes and initial points—factors that can have a substantial impact on their effectiveness in large-scale applications.

In this paper, we consider stochastic Bregman proximal gradient methods (SBPG) for nonconvex problems with application to the training of deep neural networks. We establish the convergence result of a vanilla SBPG without Lipschitz smoothness assumption for nonconvex problems. {Moreover, we propose a momentum-based variant, denoted as MSBPG, and prove that it offers improved convergence properties compared to vanilla SBPG, particularly through a relaxed requirement on the mini-batch size. Both methods exhibit sample complexity of $\mathcal{O}(\epsilon^{-4})$, which matches the optimal bound for stochastic first order methods \citep{arjevani2023lower}.} We apply MSBPG to the training of deep neural networks with a polynomial kernel function, which ensures the smooth adaptivity of the loss function. {Through an implicit reformulation, we observe that MSBPG enhances the robustness of neural network training by mitigating the gradient explosion phenomenon. For numerical illustrations, we conduct experiments on quadratic inverse problems (QIP) and testify vanilla SBPG's robustness to stepsize selection and initial point scaling. We also conduct extensive experiments on training deep neural networks for image classification and LSTMs for language modeling by employing MSBPG, which is well-suited for solving large-scale problems. Experimental results on representative benchmarks show that MSBPG converges to stationary points and, in some cases, nearly achieves the global minimum (i.e., training loss approaches zero). Moreover, MSBPG achieves comparable generalization performance to widely-used optimization algorithms, such as SGD~\citep{robbins1951stochastic}, Adam~\citep{kingma2014adam}, and AdamW~\citep{loshchilov2017decoupled}, and in some instances, even outperforms these methods. The polynomial kernel function employed contributes to improved algorithmic stability compared to standard SGD, which may partly explain the good generalization performance of MSBPG observed in our experiments (see Appendix B for details). As the primary focus of this paper is on the optimization properties of Bregman-type methods, we only provide a preliminary analysis of algorithmic stability 
in the appendix.} Furthermore, compared with standard SGD, SBPG/MSBPG is more robust to large stepsize and initial point scaling, which are the common reasons behind gradient explosion. 

 To summarize, our contributions are as follows:
\begin{itemize}
    \item[1.] {\textbf{Development of SBPG for General Nonconvex Composite Problems:} We investigate Stochastic Bregman Proximal Gradient (SBPG) method to solve nonconvex problems without finite-sum structure, which employs Bregman distance to handle the non-Lipschitz gradient continuity. we establish its convergence properties along with optimal sample complexity of $\mathcal{O}(\epsilon^{-4})$. Furthermore, we propose a momentum-based variant, MSBPG, which improves the convergence property by relaxing the mini-batch size requirements, which is more suitable for large-scale problems. To our knowledge, this is the first integration of momentum techniques into a stochastic Bregman proximal gradient framework for nonconvex problems.}
    \item[2.] \textbf{Tailored MSBPG for Deep Neural Networks with Polynomial Kernel:} We apply MSBPG to training deep neural networks (DNN), which leverages on a suitable polynomial kernel function to ensure that the DNN's loss function is smooth adaptable with respect to the designed kernel function.  Through an implicit reformulation, we observe that MSBPG is more robust than the traditional SGD in terms of stepsize selection and initialization. We highlight that MSBPG is a theoretically derived method that is able to ease the difficulty of selecting stepsize, mitigate gradient explosion, and maintain good generalization performance simultaneously. This distinguishes MSBPG from many existing techniques that rely on intuition and empirical observations. 
    \item[3.] \textbf{Empirical Evaluation of SBPG/MSBPG:} We demonstrate the efficiency and robustness of SBPG/MSBPG across various applications, including sparse quadratic inverse problems and large-scale deep neural networks. In the quadratic inverse problem, SBPG shows greater robustness to both stepsize and initial point selection. When training deep neural networks, MSBPG achieves comparable generalization performance to commonly used optimizers such as SGD, Adam, and AdamW, and in many cases, even outperforms them. Additionally, our method demonstrates robustness to stepsize selection and initialization. These results highlight the potential of MSBPG as a powerful tool for optimizing complex and large-scale deep neural networks, thus offering a promising direction for future research in this area.  
\end{itemize}
The remainder of this paper is organized as follows. In Section \ref{section-notation-preliminaries}, we present notation, some related preliminaries, and our problem setting. In Section \ref{section-vanilla-SBPG}, we first describe SBPG and establish its convergence results. Then, we propose a momentum-based SBPG (MSBPG) and prove its improved convergence properties. In Section \ref{DNN implementation}, we adapt SBPG/MSBPG to the training of deep neural networks and analyze its capacity in mitigating gradient explosion. In Section \ref{section-experiments}, we present numerical experiments that demonstrate the efficiency and robustness of vanilla SBPG on quadratic inverse problems and {MSBPG on} 
training deep neural networks. Finally, we give some concluding remarks in Section \ref{section-conclusion}. Additional supplementary materials are provided in the Appendix.

\section{Preliminaries and Problem setting}\label{section-notation-preliminaries}
In this paper, vectors are represented using boldface letters like $\vv$, while scalars are represented using normal font. Given a proper, lower-semicontinuous function $F:\R^d\rightarrow\bar\R := [-\infty,\infty]$, $dom\,F=\{\vx:F(\vx)<\infty\}$. The Fenchel conjugate function of $F$ is defined as $F^*(\vy)=\sup\{\inprod{\vx}{\vy}-F(\vx):\vx\in\R^d\}$. Given a set $\cS\subset\R^d$, $\bar\cS$ denotes its closure, $int\,\cS$ denotes the set of interior points. A function is of class $\cC^k(\cS)$ if it is $k$ times differentiable  and the $k$-th derivative is continuous on $\cS$. We say that $F$ is level bounded if the 
set $\{x:F(\vx)<\alpha\}$ is bounded for any real number $\alpha$. Given a matrix $A$, ${\rm Vec}(A)$ denotes the vectorization of $A$ by column order. ${\rm Mat}(\cdot)$ is the inverse operation of ${\rm Vec}(\cdot)$, which reshapes a vector back into its original matrix form. Define the operator $\text{Diag}(\cdot)$ to {map}  a vector into a diagonal matrix with diagonal elements equal to the corresponding entries of the vector. The Hadamard product is represented by the symbol $\circ$. If we use the notation $\|\cdot\|$ without any additional explanation, we assume that it refers to the Euclidean vector norm for vectors and the Frobenius matrix norm for matrices. 

Let $(\Omega,\cF,\mathbb P)$ be a  probability space. Given a random variable ${\bm\xi}$ and a $\sigma$-algebra $\cF$, we write ${\bm\xi}\vartriangleleft\cF$ if ${\bm\xi}$ is measurable over $\cF$. {Let $\{{\bm\xi}_k\}_{k\geq0}$ be a stochastic process, and $\{\cF_k\}_{k\geq0}$ be a filtration, where $\cF_k$ defined by a $\sigma$-algebra $\cF_k:=\sigma(\bm\xi_0, \ldots, \bm\xi_{k-1}) $ on $\Omega$. The conditional expectation is denoted by $\E[\cdot|\cF_k]$.  For simplicity, we use the notation  $\E[\cdot]$ to denote $\E[\cdot|\cF_{\infty}]$. The sequence $\{\vx^k\}_{k\geq0}$ generated by our proposed method is adapted to the filtration $\{\cF_k\}_{k\geq0}$, i.e. $\vx^k\vartriangleleft\cF_k$, for all $k\geq0$. The notation $\wtnabla_k$ represents an estimator of the exact gradient $\nabla F(\vx^k)$, which satisties $\widetilde\nabla_k\vartriangleleft\cF_{k+1}$. This estimator is applicable to both the vanilla and momentum cases. The stochastic error is denoted by ${\bm\varepsilon}_k=\nabla F(\vx^k)-\wtnabla_k$. The unbiasedness of the stochastic error {$\bm{\varepsilon}_k$} is assumed throughout this paper, i.e., $\E[{\bm\varepsilon}_k|\cF_{k}]=0$.}
The following supermartingale convergence lemma is a fundamental tool in the analysis of stochastic algorithms.
\begin{lemma} \citep{robbins1951stochastic}
Let $\left\{y_k\right\},\left\{u_k\right\},\left\{a_k\right\}$ and $\left\{b_k\right\}$ be non-negative adapted processes with respect to the filtration $\left\{\mathcal{F}_k\right\}$ such that $\sum_{k=0}^\infty a_k<\infty, \sum_{k=0}^\infty b_k<\infty$, and for all $k$, $\mathbb{E}\left[y_{k+1} \mid \mathcal{F}_k\right] \leq\left(1+a_k\right) y_k-u_k+b_k$ almost surely. Then, $\left\{y_k\right\}$ converges almost surely to a non-negative finite random variable and $\sum_{k=0}^\infty u_k<\infty$ almost surely.
\label{martingale-convergence}
\end{lemma}

\subsection{Smooth adaptable functions}
In this subsection, we introduce the concept of smooth adaptivity, initially proposed by \citet{bolte2018first}. This concept extends the idea of relative smoothness for convex functions, first introduced in \citet{bauschke2017descent, lu2018relatively}, and has since been applied in various contexts \citep{hanzely2021accelerated, yang2021inexact, yang2022bregman,yang2024inexact}. We first give the definitions of kernel function and Bregman distance.
\begin{definition}\label{Bregman-distance-def}
(Kernel function and Bregman distance). Let $\cS$ be a nonempty, convex and open subset of $\R^d$. A function $\phi:\R^d\rightarrow(-\infty,+\infty]$ is called a kernel function associated with $\cS$ if it satisfies the following two conditions:
\begin{enumerate}
    \item $\phi$ is proper, lower-semicontinuous and convex, with $dom\,\phi\subset\bar\cS$, $dom\,\partial\phi=\cS$.
    \item $\phi\in\cC^1(\cS)$ and $int\,dom\,\phi=\cS$.
\end{enumerate}
Denote the class of kernel function {associated with $\cS$}  by $\cM(\cS)$. Given $\phi\in\cM(\cS)$, the Bregman distance \citep{bregman1967relaxation} generated by $\phi$ is defined as $\cD_\phi(\vx,\vy):dom\,\phi\times int\,dom\,\phi\rightarrow[0,+\infty)$, where
\[
\cD_\phi(\vx,\vy)=\phi(\vx)-\phi(\vy)-\inprod{\nabla\phi(\vy)}{\vx-\vy}.
\]
\end{definition}
Bregman distance measures the difference between the value of $\phi$ at $\vx$ and its linear approximation at $\vy$ based on the gradient of $\phi$ at $\vy$. Some basic properties of Bregman distance can be found in \citep{chen1993convergence,teboulle2018simplified}. Some  kernel functions commonly used in optimization are $\frac{1}{2}\|\vx\|^2$, $\frac{1}{2}\|\vx\|^2+\frac{\alpha}{4}\|\vx\|^4$, $-\sum_{i=1}^d\log\vx_i$ and $\sum_{i=1}^d\vx_i\log\vx_i$, where $\frac{1}{2}\|\vx\|^2\in\cM(\R^d)$ recovers the classical half squared Euclidean distance. The kernel function $\frac{1}{2}\|\vx\|^2+\frac{\alpha}{4}\|\vx\|^4\in\cM(\R^d)$ has found applications in various problems, such as quadratic inverse problems, non-negative matrix factorization, and low-rank minimization \citep{bolte2018first,dragomir2021quartic}. The entropy function $\sum_{i=1}^d\vx_i \log \vx_i\in\cM(\R^d_{++})$ is commonly used in applications that involve probability constraints, where the resulting Bregman distance is known as the Kullback–Leibler (KL) divergence. Throughout the paper we will focus on the following pair of functions $(f,\phi)$ satisfying smooth adaptivity condition. We introduce this concept in the following definition: 
\begin{definition}\label{relative-smooth-def}
 (Smooth adaptivity). Given a kernel function $\phi\in\cM(\cS)$, 
 a proper lower-semicontinuous function $f:\R^d\rightarrow(-\infty,+\infty]$ with $dom\,f\supset dom\,\phi $ that  is $\cC^1$ on $\cS$. $f$ is $L$-smooth adaptable with respect to $\phi$ if there exists $L>0$, such that $L\phi+f$ and $L\phi-f$ are convex on $\cS$.
\end{definition}
Alternative definition of smooth adaptivity is the two-side descent lemma \cite[Lemma 2.1]{bolte2018first}. When both $f$ and $\phi$ belong to $\cC^2(\cS)$, we can verify their smooth adaptivity by comparing the Hessians of $f$ and $\phi$.
\begin{lemma}\label{two-side-descent}
$f$ is $L$-smooth adaptable with respect to $\phi\in\cM(\cS)$, if and only if 
\[
|f(\vx)-f(\vy)-\inprod{\nabla f(\vy)}{\vx-\vy}|\leq L\cD_\phi(\vx,\vy),\;\forall\,\vx,\vy\in int\,dom\,\phi.
\]
Moreover, when both $f$ and $\phi$ {belong to} $\cC^2(int\,dom\,\phi)$, then the above is equivalent to
\[
\exists L>0,\;L \nabla^2\phi(\vx){\pm}\nabla^2f(\vx) \succeq 0, \text { for all } \vx \in int\,dom\phi.
\]
\end{lemma}
The following four-point identity is frequently employed in our proofs, and can be easily verified.
\begin{lemma}(Four points identity) 
Given points $\va,\vb,\vc,\vd$ and any convex function $\phi$ which is differentiable at $\va$ and $\vb$, then
\[
\inprod{\nabla\phi(\va)-\nabla\phi(\vb)}{\vc-\vd}=\cD_\phi(\vc,\vb)+\cD_\phi(\vd,\va)-\cD_\phi(\vc,\va)-\cD_\phi(\vd,\vb).
\]
\label{four-point-identity}
\end{lemma}

\subsection{Bregman Proximal Mapping}\label{subsec-Bregman-Proximal-mapping}
Throughout this paper, we make the following basic assumptions.
\begin{assumption}\label{relative-smooth-assumption}
(Basic requirements). In problem \eqref{min-prob}:
\begin{itemize}
 \item[A1.]  $F$ is a proper lower-semicontinuous function with $dom\,\phi\subset dom\,F$, {and it} is $\cC^1$ on $int\,C$.
 \item[A2.] The Legendre kernel (Definition \ref{Legendre}) $\phi\in\cM(C)$ is $\mu$-strongly convex {for some} $\mu>0$.  $F(\cdot)$ is $L_F$-smooth adaptable with respect to $\phi$.
 \item[A3.] $R$ is is a proper, lower-semicontinuous and convex function with $dom\,R\cap int\,C\neq\emptyset$.
 \item[A4.] $\inf_{\vx\in\overline{C}}\{\Phi(\vx)\}>-\infty$.  
\end{itemize}
\end{assumption}
 Assumption \ref{relative-smooth-assumption} is a standard requirement for Bregman-type methods and is usually satisfied in practice. 
 It ensures the well-definedness of Bregman-type methods, as shown in \citep{bolte2018first,latafat2022bregman}. We also recall the definition of the Legendre function in \citep{latafat2022bregman}, which makes additional supercoercive conditions on the concept in \citep{rockafellar1997convex}. 

\begin{definition}\label{Legendre}
(Legendre kernel). Let $\phi:{\overline{C}}\rightarrow(-\infty,\infty]$ be a proper
lower-semicontinuous convex function. 
It is called essentially smooth if $\text{int\;dom}\,\phi$ is nonempty and $\phi$ is differentiable on $\text{int\;dom}\,\phi$, moreover $\lim_{k\rightarrow\infty}\|\nabla\phi({\vx^k})\|=\infty$ whenever $\{{\vx^k}\}_{k\in\mathbb N}$ converges to a boundary point of $\text{dom}\,\phi$. The function $\phi$ is called Legendre function if it is essentially smooth, {strictly} convex on $\text{int\,dom}\,\phi$ and supercoercive, i.e. $\lim_{\|\vx\|\rightarrow\infty}\frac{\phi(\vx)}{\|\vx\|}=\infty$.
\end{definition}

\begin{definition}\label{Bregman-prox-maping}
Given a nonempty convex open set $C$, a proper lower-semicontinuous convex function $R$ and a Legendre kernel function $\phi\in\cM(C)$, $\vx\in int\,dom\,\phi$, we denote the Bregman proximal mapping by ${\rm Prox}^\phi_R:=(\nabla\phi+\partial R)^{-1}\nabla\phi$, which is equivalent to 
\begin{equation}
{\rm Prox}^\phi_R(\vx)=\argmin_{\vu\in\overline{C}}\;\{R(\vu)+\cD_\phi(\vu,\vx)\}.
\label{Breg-prox-def}
\end{equation}  
\end{definition}

Note that the objective function of \eqref{Breg-prox-def} is strictly convex on $\text{dom}\,\phi\cap \text{dom}R$, therefore \eqref{Breg-prox-def} has at most one solution. To ensure that \eqref{Breg-prox-def} is well-defined, the following result claims that ${\rm Prox}^\phi_{\alpha R}(\vx)$ is well-defined for any $\alpha>0$, and moreover ${\rm Prox}^\phi_{\alpha R}(\vx)\in\text{int\;dom}\,\phi$ under standard assumptions. The proof can be found in Appendix \ref{appendix-preliminary}.
\begin{lemma}
Suppose Assumption \ref{relative-smooth-assumption} holds. Then \eqref{Breg-prox-def} has a unique solution. Moreover, the solution ${\rm Prox}^\phi_{\alpha R}(\vx)\in C$. 
\label{well-def-Breg-prox}
\end{lemma}
The following proposition for Bregman proximal mapping generalizes the nonexpansive property of the classical proximal mapping (in the case $\phi(\vx)=\frac{1}{2}\|\vx\|^2$). This property is commonly used in convergence proofs. The proof of the following proposition can be found in Appendix \ref{appendix-preliminary}.
\begin{proposition}
 Suppose Assumption \ref{relative-smooth-assumption} holds. Let $\vx_i^+:={\rm Prox}^\phi_R(\nabla\phi^*(\vx_i))$, $i=1,2$. Then $\|\vx^+_1-\vx^+_2\|\leq\frac{1}{\mu}\|\vx_1-\vx_2\|$.
\label{nonexpansive}
\end{proposition}

In this paper, we make the assumption that $R$ and $\phi$ are simple enough so that \eqref{Breg-prox-def}  either has a closed-form solution or {admits} an efficient subroutine to solve it. Using the definition of the Bregman proximal mapping, we can then define the Bregman gradient mapping associated with \eqref{min-prob}. This mapping measures the 
{solution accuracy of}
the methods we propose. Note that $\phi$ is a Legendre kernel, which implies that $\phi^*\in\cC^{1}(\R^d)$ is strictly convex and $(\nabla\phi)^{-1}=\nabla\phi^*$ \citep[Corollary 13.3.1, Theorem 26.5]{rockafellar1997convex}. Therefore, the following concept is well-defined. 

\begin{definition}[Bregman Gradient Mapping]
Given $\alpha>0$, a nonempty convex open set $C$ and a Legendre kernel function $\phi\in\cM(C)$, the Bregman gradient mapping associated with \eqref{min-prob} is defined as follows
\[
\mathcal G_\alpha(\vx)=\frac{\vx-{\rm Prox}^\phi_{\alpha R}\left(\nabla\phi^*(\nabla \phi(\vx)-\alpha\nabla F(\vx))\right)}{\alpha}.
\]
To simplify notation, we use $\cG(\vx)$ to denote  $\cG_1(\vx)$ when $\alpha=1$.
\label{Breg-prox-gradient-mapping}
\end{definition}
When the kernel function $\phi(\vx) = \frac{1}{2}\|\vx\|^2$, the resulting Bregman Gradient Mapping becomes equivalent to the classical Gradient Mapping \citep{nesterov2003introductory,nesterov2005smooth}, which measures the solution's accuracy for proximal gradient methods. 
\begin{definition}
(Limiting subdifferential \citep[Definition 8.3]{RockWets98}) Consider a function $f:\R^d\rightarrow\bar\R$ and a point ${\vx}$, the regular subdifferential is defines as
\[
\hat\partial f(\vx)=\{\vv:f(\vy)\geq f(\vx)+\inprod{\vv}{\vy-\vx}+o(\|\vy-\vx\|)\}.
\]
The limiting subdifferential is defined as 
\[
\partial f(\vx)=\{\vv:\vx_n\rightarrow \vx,f(\vx_n)\rightarrow f(\vx),\vv_n\in\hat\partial f(\vx_n),\;and\;\vv_n\rightarrow \vv\}.
\]
\end{definition}
{ Now, we restrict our attention on the case $C=\R^d$.} By Fermat's rule \cite[Theorem 10.1]{RockWets98}, the set of critical point of $\Phi$ is given by
\[
{\rm crit}\,\Phi=\left\{\vx\in\R^d:\;0\in\partial\Phi(\vx)\equiv\nabla F(\vx)+\partial R(\vx)\right\}.
\]
The Bregman Gradient Mapping can also be used to evaluate the 
{solution accuracy} for Bregman methods. Let $\vx^+ = \text{Prox}^\phi_{\alpha R}(\nabla\phi^*(\nabla \phi(\vx) - \alpha\nabla F(\vx)))$. From Definition \ref{Breg-prox-gradient-mapping} and equation \eqref{Breg-prox-def}, it can be easily verified by definition that $0 \in \partial\Phi(\vx) \Leftrightarrow 0 = \mathcal{G}_{\alpha}(\vx).$ 
{Hence,} $0 \in \partial\Phi(\vx^+)$ for any $\alpha > 0$. 
The proof of this result is omitted for brevity. Furthermore, if $\nabla\phi$ is $L_\phi$-Lipschitz continuous, then the following proposition holds, implying that $\|\cG_\alpha(\vx)\|$ can be used as a reasonable criterion {to measure the accuracy of $\vx$.}
\begin{proposition}
\label{prop-Bregman-mapping-subdifferential}
Suppose Assumption \ref{relative-smooth-assumption} holds and that $\nabla\phi$ is $L_\phi$ Lipschitz continuous. Then, we have the following inequality:
\[
{\rm dist}\left(0,\partial\Phi({\vx^+})\right)\leq(1+\alpha L_F)L_\phi\|\cG_\alpha(\vx)\|.
\]
\end{proposition}
 We also define the stochastic counterpart of Definition \ref{Breg-prox-gradient-mapping}, which is commonly utilized to evaluate the accuracy of solutions for nonconvex stochastic proximal gradient methods, as discussed in \citep{ghadimi2016mini}.
\begin{definition}\label{stoc-Bregman-grad-map}
(Stochastic Bregman Gradient Mapping).
Given $\alpha>0$ a nonempty convex open set $C$ and a Legendre kernel function $\phi\in\cM(C)$, the stochastic Bregman gradient mapping associated with \eqref{min-prob} is defined as follows
 \[
 \widetilde{\mathcal G}_{\alpha}({\vx}):=\frac{{\vx}-{\rm Prox}^\phi_{\alpha R}\left(\nabla\phi^*\left(\nabla\phi({\vx})-\alpha\widetilde\nabla\right)\right)}{\alpha},\;\text{where }\widetilde{\nabla}\;\text{is an estimator of }\nabla F(\vx).
 \]   
\end{definition}

\section{Stochastic Bregman Proximal Gradient Method}\label{section-vanilla-SBPG}
In this section, we will study the Stochastic Bregman Proximal Gradient method (SBPG) with the following update scheme:
\begin{equation}
{\vx^{k+1}}=\argmin_{\vx\in\overline{C}}\;R(\vx)+\inprod{\widetilde\nabla_k}{\vx-{\vx^k}}+\frac{1}{\alpha_k}\cD_\phi(\vx,{\vx^k}).
\label{vanilla-SBPG}
\tag{SBPG}
\end{equation}
We call {the above method as} "vanilla" SBPG in this section, meaning that the method we study is a basic version without any additional techniques such as variance reduction, momentum, etc., except for the use of mini-batches. In this case, we suppose the following assumptions.
\begin{assumption}\label{variance-assumption}
(Noise requirement).
The estimator satisfies the following two conditions:
$$\E[\widetilde{\nabla}_k|\cF_k]=\nabla F({\vx_k}) \quad \text{and}\quad 
\mathbb E[\|\widetilde\nabla_k-\nabla F({\vx^k})\|^2|\cF_k]\leq\frac{\sigma^2}{m_k},$$ where $m_k$ is the size of the mini-batch in the $k$-th iteration.    
\end{assumption}
 {Note} that we do not assume a finite-sum structure for $F(\vx)$ in this section. The solution of \eqref{vanilla-SBPG} can be written in the form of the Bregman proximal mapping. This is stated in the following proposition. 
 \begin{proposition} 
  \label{subprob-closed-form-sol} 
 Suppose Assumption \ref{relative-smooth-assumption} holds. Then the solution of \eqref{vanilla-SBPG} can be written as the following Bregman proximal mapping:
 \[
 {\vx^{k+1}}={\rm Prox}^\phi_{\alpha_k R}\left(\nabla\phi^*\left(\nabla\phi({\vx^k})-\alpha_k\widetilde\nabla_k\right)\right).
 \]
 \end{proposition}
 \begin{proof}
From the optimality condition of the main subproblem \eqref{vanilla-SBPG}, we have
 \[
 0\in\partial R({\vx^{k+1}})+\wtnabla_k+\frac{1}{\alpha_k}\left(\nabla\phi({\vx^{k+1}})-\nabla\phi({\vx^k})\right).
 \]
Let $\vu^{k+1}:={\rm Prox}^\phi_{\alpha_k R}\left(\nabla\phi^*\left(\nabla\phi({\vx^k})-\alpha_k\widetilde\nabla_k\right)\right)$. From the definition of Bregman proximal mapping, we have
\[
\vu^{k+1}=\arg\min_\vu\;\left\{\alpha_kR(\vu)+{\cD_\phi}\left(\vu,\nabla\phi^*\left(\nabla\phi({\vx^k})-\alpha_k\widetilde\nabla_k\right)\right)\right\},
\]
which is equivalent to
\[
0\in\alpha_k\partial R(\vu^{k+1})+\nabla\phi(\vu^{k+1})-\nabla\phi\big(\nabla\phi^*(\nabla\phi({\vx^k})-\alpha_k\wtnabla_k)\big).
\]
Note that the function $\phi^*$ is the Fenchel conjugate of the Legendre kernel $\phi$, which implies that $\nabla\phi(\nabla\phi^*(\vw))=\vw$ for all $\vw\in\R^d$, as stated in \cite[Corollary 13.3.1, Theorem 26.5]{rockafellar1997convex}. Furthermore, since the objective function in \eqref{vanilla-SBPG} is strictly convex, there exists a unique solution to the inclusion above. By comparing the two inclusions, we can conclude that $\vu^{k+1}={\vx^{k+1}}$.
 \end{proof}
Based on Proposition \ref{subprob-closed-form-sol} and definition of the definition of $\widetilde{\cG}_\alpha(\vx)$, we can easily observe that $\vx^{k+1}=\vx^k-\alpha_k\widetilde{\cG}_{\alpha_k}(\vx^k)$. We can derive the following proposition, which bounds the difference between $\mathcal{G}_{\alpha}(\vx)$ and $\widetilde{\mathcal{G}}_{\alpha}(\vx)$ directly from Proposition \ref{nonexpansive}. The proof is omitted for brevity.
\begin{proposition}
\label{gradient-map-gap}
Suppose Assumption \ref{relative-smooth-assumption} holds. At the $k$-th step, we have the estimation:
    \[
    \|\mathcal G_{\alpha_k}({\vx^k})-\widetilde{\mathcal G}_{\alpha_k}({\vx^k})\|\leq\frac{1}{\mu}\|\nabla F({\vx^k})-\widetilde{\nabla}_k\|=\frac{\|{\bm\varepsilon}_k\|}{\mu},
    \]
    where 
    {$\bm{\varepsilon}_k=\nabla F({\vx^k})-\widetilde{\nabla}_k.$}
\end{proposition}
Before presenting the main convergence result, we state the following one-step descent lemma below.
\begin{lemma}
Suppose Assumption \ref{relative-smooth-assumption} holds. The sequence generated by SBPG satisfies the following condition: 
\[
\Phi({\vx^{k+1}})\leq\Phi({\vx^k})-\frac{1}{\alpha_k}\cD_\phi({\vx^k},{\vx^{k+1}})-\left(\frac{1}{\alpha_k}-L_F\right)\cD_\phi({\vx^{k+1}},{\vx^k})+\inprod{{\bm\varepsilon}_k}{{\vx^{k+1}}-\vx^{k}}.
\]
\label{basic-descent-lemma-2}
\end{lemma}
\begin{proof}
By the optimality condition of \eqref{vanilla-SBPG}, we obtain that 
\[
0\in\partial R({\vx^{k+1}})+\wtnabla_k+\frac{1}{\alpha_k}\left(\nabla\phi({\vx^{k+1}})-\nabla\phi({\vx^k})\right).
\]
Appealing to the convexity of $R$, we have
\[
R(\vx)-R({\vx^{k+1}})\geq\big\langle{-\wtnabla_k-\frac{1}{\alpha_k}\left(\nabla\phi({\vx^{k+1}})-\nabla\phi({\vx^k})\right)},\,{\vx-{\vx^{k+1}}}\big\rangle.
\]
By the four points identity and the definition of ${\bm\varepsilon}_k$, we get 
\[
 R(\vx)-R({\vx^{k+1}})\geq\frac{1}{\alpha_k}\left[{\cD_\phi}({\vx^{k+1}},{\vx^k})+{\cD_\phi}(\vx,{\vx^{k+1}})-{\cD_\phi}(\vx,{\vx^k})\right]-\inprod{\nabla F({\vx^k})}{\vx-{\vx^{k+1}}}+\inprod{{\bm\varepsilon}_k}{\vx-{\vx^{k+1}}}.
\]
Set $\vx={\vx^k}$ in the above inequality, we have the following inequality:
\[
R({\vx^k})-R({\vx^{k+1}})\geq\frac{1}{\alpha_k}\left[{\cD_\phi}({\vx^{k+1}},{\vx^k})+{\cD_\phi}({\vx^k},{\vx^{k+1}})\right]-\inprod{\nabla F({\vx^k})}{{\vx^k}-{\vx^{k+1}}}+\inprod{{\bm\varepsilon}_k}{{\vx^k}-{\vx^{k+1}}}.
\]
By the smooth adaptivity of $F$, we have 
\[
F({\vx^{k+1}})\leq F({\vx^k})+\inprod{\nabla F({\vx^k})}{{\vx^{k+1}}-{\vx^k}}+L_F{\cD_\phi}({\vx^{k+1}},{\vx^k}).
\]
Combining the above two inequalities above, we complete the proof.
\end{proof}

\subsection{Convergence analysis of SBPG}

 In this subsection, we establish the convergence results for SBPG, which is an extension of the convergence result in \citep{ghadimi2016mini}, in which the lassical Lipschitz gradient assumption is required. In many literature, the bounded sequence assumption is often required in the convergence analysis of stochastic algorithms. However, in this section, we relax this assumption and prove that under a certain condition, the sequence generated by \eqref{vanilla-SBPG} is bounded almost surely. We need the following result to bound the stochastic error term $\inprod{{\bm\varepsilon}_k}{{\vx^{k+1}}-{\vx^k}}$ in Lemma \ref{basic-descent-lemma-2}.
\begin{lemma}
Suppose Assumption \ref{relative-smooth-assumption}, \ref{variance-assumption} hold. We have the following estimation of the error term:
\[
\mathbb E\left[\inprod{{\bm\varepsilon}_k}{{\vx^{k+1}}-{\vx^k}}\right]\leq\frac{\alpha_k}{\mu}\E[\|{\bm\varepsilon}_k\|^2]\leq\frac{\alpha_k\sigma^2}{\mu m_k}.
\]
\label{inprod-bound}
\end{lemma}
\begin{proof}
Define ${\bar\vx}^{k+1}:= {{\rm Prox}_{\alpha_kR}^\phi} (\nabla\phi^*(\nabla\phi({\vx^k})-\alpha_k\nabla F({\vx^k})))$. By Proposition \ref{subprob-closed-form-sol} and 
the optimality condition for {${\bar\vx}^{k+1}$}, we have
\[
0\in\partial R({\bar\vx}^{k+1})+\nabla F({\vx^k})+\frac{1}{\alpha_k}(\nabla\phi({\bar\vx}^{k+1})-\nabla\phi({\vx^k})).
\]
Similarly,
\[
0\in\partial R({\vx^{k+1}})+\wtnabla_k+\frac{1}{\alpha_k}(\nabla\phi({\vx^{k+1}})-\nabla\phi({\vx^k})).
\]
By the monotonicity of $\partial R$ and Lemma \ref{four-point-identity}, we have
\[
\big\langle{{\bar\vx}^{k+1}-{\vx^{k+1}}},\,
{-{\bm\varepsilon}_k-\frac{1}{\alpha_k}(\nabla\phi({\bar\vx}^{k+1})-\nabla\phi({\vx^{k+1}}))}
\big\rangle
\geq0.
\]
Therefore, 
\[
\inprod{{\vx^{k+1}}-{\bar\vx}^{k+1}}{{\bm\varepsilon}_k}\geq
\inprod{{\bar\vx}^{k+1}-{\vx^{k+1}}}{\frac{1}{\alpha_k}(\nabla\phi({\bar\vx}^{k+1})-\nabla\phi({\vx^{k+1}}))}\geq\frac{\mu}{\alpha_k}\|{\bar\vx}^{k+1}-{\vx^{k+1}}\|^2.
\]
By Cauchy-Schwarz inequality, we get $\|{\bar\vx}^{k+1}-{\vx^{k+1}}\|\leq\frac{\alpha_k}{\mu}\|{\bm\varepsilon}_k\|$.  

Now, we are ready to prove Lemma \ref{inprod-bound}. From the definition, we know that ${\bar\vx}^{k+1}\vartriangleleft\cF_k$. Therefore, $\E[\inprod{{\bm\varepsilon}_k}{{\vx^k}-{\bar\vx}^{k+1}}]=\E[\E[\inprod{{\bm\varepsilon}_k}{{\vx^k}-{\bar\vx}^{k+1}}|\cF_k]]=\E[\inprod{\E[{\bm\varepsilon}_k|\cF_k]}{{\vx^k}-{\bar\vx}^{k+1}}]=0$, {where} the first equality is from the tower rule of conditional expectation, the second comes from {the fact} that ${\vx^k}-{\bar\vx}^{k+1}\vartriangleleft\cF_k$. Hence,
{\small
\[
\mathbb E\left[\inprod{{\bm\varepsilon}_k}{{\vx^{k+1}}-\vx^{k}}\right]=\E\left[\inprod{{\bm\varepsilon}_k}{{\vx^{k+1}}-{\bar\vx}^{k+1}}\right]-\E\left[\inprod{{\bm\varepsilon}_k}{{\vx^k}-{\bar\vx}^{k+1}}\right]\leq\frac{\alpha_k}{\mu}\E[\|{\bm\varepsilon}_k\|^2]\leq\frac{\alpha_k\sigma^2}{\mu m_k},
\]}
which completes the proof.
\end{proof}
 
\begin{lemma}[Bounded sequence]
Suppose Assumption \ref{relative-smooth-assumption}, \ref{variance-assumption} hold. If $\sum_k\frac{\alpha_k}{m_k}<\infty$, $\sup_k\alpha_k\leq\bar\alpha<\frac{1}{L_F}$, then,
 \begin{itemize}
     \item[1.] $\sum_{k=0}^\infty \E[\cD_\phi({\vx^{k+1}},{\vx^k})]<\infty$.
     \item[2.] If $\Phi$ is level bounded, then $\{{\vx^k}\}_{k\geq0}$ is bounded almost surely.
   \end{itemize}
   \label{bounded sequence-vanilla}
\end{lemma}
\begin{proof}
By Cauchy-Young inequality, we have 
\[
|\inprod{{\bm\varepsilon}_k}{{\vx^k}-{\vx^{k+1}}}|\leq\frac{\mu}{2\alpha_k}\|{\vx^k}-{\vx^{k+1}}\|^2+\frac{\alpha_k}{2\mu}\|{\bm\varepsilon}_k\|^2\leq \frac{1}{\alpha_k}{\cD_\phi}({\vx^k},{\vx^{k+1}})+\frac{\alpha_k}{2\mu}\|{\bm\varepsilon}_k\|^2.
 \]
 By Lemma \ref{basic-descent-lemma-2}, we have
 \begin{equation}
 \left(\frac{1}{\alpha_k}-L_F\right){\cD_\phi}({\vx^{k+1}},{\vx^k})\leq\Phi({\vx^k})-\Phi({\vx^{k+1}})+\frac{\alpha_k}{2\mu}\|{\bm\varepsilon}_k\|^2.
 \label{bound-sequence-recursion}
 \end{equation}
Taking conditional expectation for both sides of \eqref{bound-sequence-recursion}, we get
\[
\E\left[\left(\frac{1}{\alpha_k}-L_F\right){\cD_\phi}({\vx^{k+1}},{\vx^k})|\cF_k\right]\leq\Phi({\vx^k})-\E[\Phi({\vx^{k+1}})|\cF_k]+\frac{\alpha_k}{2\mu}\E[\|{\bm\varepsilon}_k\|^2|\cF_k].
\]
Since $\sum_{k\geq0}\frac{\alpha_k}{2\mu}\E[\|{\bm\varepsilon}_k\|^2|\cF_k]\leq\sum_{k\geq0}\frac{\alpha_k\sigma^2}{2\mu m_k}<\infty$, applying Theorem \ref{martingale-convergence}, we have that $\Phi({\vx^k})$ converges and $\sum_{k\geq0}\E\left[\left(\frac{1}{\alpha_k}-L_F\right){\cD_\phi}({\vx^{k+1}},{\vx^k})|\cF_k\right]<\infty$  almost surely.  By the tower rule of conditional expectation, we have $\sum_{k=0}^\infty \E[{\cD_\phi}({\vx^{k+1}},{\vx^k})]<\infty$. Since $\Phi({\vx^k})$ converges {almost surely}, thus $\{\Phi({\vx^k})\}_{k\geq0}$ is bounded almost surely. By the level boundness of $\Phi$, we deduce that $\{{\vx^k}\}_{k\geq0}$ is bounded almost surely.
\end{proof}

Now, we present our main convergence result for the vanilla SBPG, which is in the sense of expectation.
\begin{theorem}[Convergence result in expectation]
  Suppose Assumption \ref{relative-smooth-assumption}, \ref{variance-assumption} hold, $\alpha_k<\frac{1}{L_F}\min\{1,\frac{1}{\mu}\}$. Define a random variable $r$ with the distribution $\mathbb{P}\{r=k\}=\frac{\alpha_k}{\sum_{k=0}^{N-1}\alpha_k}$ for $k=0,...,N-1$. Then,
\be\label{vSBPG-convergence-thm}
 \E[\|\widetilde{\mathcal G}_{\alpha_r}(\vx^{r})\|^2]\leq\frac{2\Delta_0+2\sum_{k=0}^{N-1}\frac{\alpha_k\sigma^2}{\mu m_k}}{\mu\sum_{k=0}^{N-1}\alpha_k},
\ee
  where $\Delta_0:=\Phi(\vx^0)-\Phi^*$. If $\sum_{k}\frac{\alpha_k}{m_k}<+\infty$ and $\sum_{k}\alpha_k=+\infty$, then the right hand side of \eqref{vSBPG-convergence-thm} converges to zero. Moreover, if $\Phi$ is level bounded, then the sequence $\{{\vx^k}\}_{k\geq0}$ is bounded almost surely.
 \label{converge-thm-vanilla-expt}
\end{theorem}
\begin{proof}
    Note that ${\vx^{k+1}}={\vx^k}-\alpha_k\widetilde{\cG}_{\alpha_k}({\vx^k})$ and by the strongly convexity of $\phi$, Lemma \ref{basic-descent-lemma-2} yields 
\[
\begin{aligned}
    \mu(\alpha_k-\frac{L_F\alpha_k^2}{2})\|\widetilde{\cG}_{\alpha_k}({\vx^k})\|^2&\leq\frac{1}{\alpha_k}\cD_\phi({\vx^k},{\vx^{k+1}})+\left(\frac{1}{\alpha_k}-L_F\right)\cD_\phi({\vx^{k+1}},{\vx^k})\\
    &\leq\Phi({\vx^k})-\Phi({\vx^{k+1}})+\inprod{{\bm\varepsilon}_k}{{\vx^{k+1}}-{\vx^k}}.
\end{aligned}
\]
Taking expectations, telescoping from $k=0...N-1$, and using Lemma \ref{inprod-bound}, we obtain
\be
 \sum_{k=0}^{N-1}\mu(\alpha_k-\frac{\mu L_F\alpha_k^2}{2})\E[\|\widetilde{\mathcal G}_{\alpha_k}({\vx^k})\|^2]\leq\Phi(\vx^0)-\Phi(\vx^N)+\sum_{k=0}^{N-1}\frac{\alpha_k\sigma^2}{\mu m_k}.
 \label{proof-vsbpg-thm}
\ee
By utilizing the inequality $\alpha_k - \frac{\mu L_F \alpha_k^2}{2} \geq \frac{\alpha_k}{2}$, the condition $\Phi(\vx^N) \geq \Phi^*$, and considering the definition of the random variable $r$, we can derive \eqref{vSBPG-convergence-thm} from \eqref{proof-vsbpg-thm}. 
\end{proof}

\begin{remark}
We give some remarks for Theorem \ref{converge-thm-vanilla-expt}.
\begin{itemize} 
\item[1.] The mini-batch size plays a crucial role in ensuring convergence, as it allows us to control the stochastic error term in Lemma \ref{basic-descent-lemma-2} and provide a bound for $\E[\|\widetilde{\cG}_{\alpha_k}({\vx^k})\|^2]$ that converges to $0$ as $k$ tends to infinity. If $m_k=1$ for all $k$, then the upper bound for $\E[\|\widetilde{\cG}_{\alpha_k}({\vx^k})\|^2]$ will not converge to $0$, no matter how $\{\alpha_k\}$ is selected. 
\item[2.] In \citep{ghadimi2016mini}, a similar convergence result is established for mini-batch stochastic proximal gradient methods, but our analysis differs in a crucial aspect in that we do not assume the Lipschitz continuity of $F(\vx)$. Instead, we rely on the smooth adaptivity of $F(\vx)$, which is a more relaxed assumption. Additionally, we provide specific conditions on the stepsizes $\{\alpha_k\}$ and mini-batch sizes $\{m_k\}$ that guarantee the convergence of $\E[\|\widetilde{\cG}_{\alpha_k}({\vx^k})\|^2]$ to $0$, as well as the almost sure boundedness of the sequence $\{\vx^k\}$.
\item[3.] {We now provide a specific choice of $\{\alpha_k\}$ and $\{m_k\}$ to establish a convergence rate in terms of expected stationarity. For given positive constants $c_1,c_2,\gamma\in(0,\infty)$ and $\delta\in[0,1)$, we choose
\[
\alpha_k=\frac{c_1}{(k+1)^\delta},\;m_k=\ceil{c_2(k+1)^\gamma}.
\]
Under these choices, we have $\sum_{i=0}^k\alpha_i=\mathcal{O}(k^{1-\delta})$ and $\sum_{i=0}^k\frac{\alpha_k}{m_k}=\mathcal{O}(k^{1-(\delta+\gamma)})$. As a result, the RHS of \eqref{vSBPG-convergence-thm} has the bound of $\mathcal{O}(k^{-(1-\delta)}+k^{-\gamma})$. When $\delta=0$ and $\gamma\geq1$, the order is $\mathcal{O}(k^{-1})$ which recovers the optimal rate of deterministic first order method. In this case, the mini-batch size $m_k$ increases rapidly and the variance in the stochastic gradient reduces to zero quickly, so the algorithm behaves like a deterministic algorithm. However, such a favorable convergence rate comes at the expense of a heavy computational burden and high memory cost in each iterative step.}
\end{itemize}
\label{Vanilla-convergence-expt-remark}
\end{remark}

{Next, we present the sample complexity of the SBPG method. We define each computation of $\nabla f(\vx,\bm\xi)$ for a fix $\bm\xi$ as an oracle evaluation.
\begin{corollary}
\label{cor:complexity-SBPG}
Given an accuracy level $\epsilon>0$ and a positive constant $\gamma > 0$, choose $\alpha_k=\alpha<\frac{1}{L_F}\min\{1,\frac{1}{\mu}\}$. Then to achieve an $\epsilon$-stationary point in expectation, at most $\bar N:=\ceil{\max\left\{\gamma^2,\,\frac{4\sigma^4}{\mu^4\epsilon^4}\left(\frac{2\Delta_0\mu\gamma}{\alpha\sigma^2}+\frac{1}{\gamma}\right)^2\right\}}$ oracle evaluations are required. 
In this case, we need to set $m_k = m := \lceil\min\{\gamma\sqrt{\bar N},\, \bar{N} \}\rceil$ for all $k.$
\end{corollary}
\begin{proof}
Let the total number of oracle evaluations be $\bar N$. 
Then $N=\floor{\frac{\bar N}{m}}$ and $N\geq\frac{\bar N}{2m}$.  We have that
\[
\begin{aligned}
 \text{RHS of }\eqref{vSBPG-convergence-thm}  \leq &\frac{4\Delta_0m}{\alpha\mu\bar N}+\frac{2\sigma^2}{\mu^2m}
 \leq \frac{4\Delta_0}{\alpha\mu\bar N}\left(\gamma\sqrt{\bar N}\right)+\frac{2\sigma^2}{\mu^2}\max\left\{\frac{1}{\gamma\sqrt{\bar N}},\frac{1}{\bar N}\right\}
 \\
 \leq&
 \frac{2\sigma^2}{\mu^2\sqrt{\bar N}}\left(\frac{2\Delta_0\mu\gamma}{\alpha\sigma^2}+\frac{1}{\gamma}\right),
\end{aligned}
\]
where $\Delta_0:=\Phi(\vx^0)-\Phi(\vx^*)$. By the choice of $\bar N$, we have $\text{RHS of }\eqref{vSBPG-convergence-thm}\leq\epsilon^2$. This completes the proof.
\end{proof} 
This oracle complexity matches the lower bound provided by \citep{arjevani2023lower}, thus it cannot be improved in general.}

\subsection{Momentum based Stochastic Bregman Gradient Descent Method}
{Remark \ref{Vanilla-convergence-expt-remark} and Corollary \ref{cor:complexity-SBPG} suggest that a large mini-batch size $m_k$ is necessary  to achieve a small stationarity error. However, employing a large mini-batch size in each iteration can be computationally expensive and memory inefficient, especially in modern large-scale problems such as deep neural network training.} In this part, we resort to the momentum technique to address this issue. Specifically, we consider using a stochastic moving average estimator (SMAE) for the true gradient given by:
\be
\vv^k=(1-\beta_k)\vv^{k-1}+\beta_k\widetilde\nabla_k,\quad \mbox{where}\quad \E[\widetilde\nabla_k|\cF_k]=\nabla F({\vx^k}),
\label{SMAE}
\ee
where $\vv^{k-1}$ can be viewed as the momentum which contains the information of all historical stochastic gradients, and $\E[\|\widetilde\nabla_k\|^2|\cF_k]\leq\frac{\sigma^2}{m_k}$. We expect that incorporating the SMAE technique can achieve a certain level of variance reduction without requiring an increase in the mini-batch size or the computation of the full gradient. In our approach, we utilize the gradient estimator $\vv^k$ within SBPG, and we refer to the resulting method as MSBPG. Specifically, we consider the following update scheme:
\begin{equation}
{\vx^{k+1}}=\argmin_{\vx\in\overline{C}}\;R(\vx)+\inprod{\vv^k}{\vx-{\vx^k}}+\frac{1}{\alpha_k}\cD_\phi(\vx,{\vx^k}).
\label{momentum-SBPG}
\end{equation}
We need the following assumption that the difference of gradients of $F$ can be bounded by the Bregman distance. 
\begin{assumption}\label{momentum-assumption}
 There exists $\kappa>0$, such that $\|\nabla F(\vx)-\nabla F(\vy)\|^2\leq\kappa \cD_\phi(\vx,\vy)$ for all $\vx\in dom\,\phi$, $\vy\in int\,dom\,\phi$.   
\end{assumption}
\begin{remark}
This assumption generalizes the case of Lipschitz kernel function. If $F$ is $L_F$-smooth adaptable to $\phi$, it can be easily shown that if $\phi$ has $L_\phi$-Lipschitz gradient, this assumption holds for $\kappa\geq\frac{2L_F^2L_\phi^2}{\mu}$. In this paper, we are particularly interested in polynomial kernel functions. For functions with polynomially bounded growth rates, this assumption is not restrictive. For example, consider the one-dimensional objective function $F(x)=\frac{1}{4}x^4$ and the kernel function $\phi(x)=\frac{1}{2}x^2+\frac{1}{8}x^8$. Then, by \cite[Proposition 2.1]{lu2018relatively}, we know that $F$ is smooth adaptable with respect to $\phi$. Simple algebra shows that $\cD_\phi(x,y)=\frac{1}{8}(x-y)^2(x^6 + 2 x^5 y + 3 x^4 y^2 + 4 x^3 y^3 + 5 x^2 y^4 + 6 x y^5 + 7 y^6 + 4)$ and $(F'(x)-F'(y))^2=(x-y)^2(x^2+xy+y^2)^2$. Numerical computation shows that $(x^6 + 2 x^5 y + 3 x^4 y^2 + 4 x^3 y^3 + 5 x^2 y^4 + 6 x y^5 + 7 y^6 + 4)-(x^2+xy+y^2)^2\geq3.71$. Therefore, $(F'(x)-F'(y))^2\leq8\cD_\phi(x,y)$, which holds globally for any $x$ and $y$ in $\R^d$. 
\end{remark}

Next, we present a recursion lemma that allows us to estimate the accuracy of the SMAE. While similar lemmas have been proposed in the literature, such as in \citep{wang2017stochastic}, their bounds are not directly applicable in the Bregman setting. As a result, we have developed a version of the recursion lemma that is tailored to our specific context. 
\begin{lemma}
\label{SMAE-recursion-lemma}
The following recursion holds
{\small
\[
\E[\|\vv^k-\nabla F({\vx^k})\|^2|\cF_k]\leq (1-\beta_k)\|\vv^{k-1}-\nabla F(\vx^{k-1})\|^2+\beta_k^2\E[\|\widetilde\nabla_k-\nabla F({\vx^k})\|^2|\cF_k]+\frac{\|\nabla F(\vx^{k-1})-\nabla F({\vx^k})\|^2}{\beta_k}.
\]
}
\end{lemma}
\begin{proof}
Note that $\vv^k-\nabla F(\vx^k)=(1-\beta_k)(\vv^{k-1}-\nabla F(\vx^{k-1}))+(1-\beta_k)(\nabla F(\vx^{k-1})-\nabla F(\vx^k))+\beta_k(\widetilde\nabla_k-\nabla F(\vx^k))$, and $\E[\widetilde\nabla_k-\nabla F(\vx^k)|\cF_k]=0$. Then we have
\[
\begin{aligned}
&\E[\|\vv^k-\nabla F(\vx^k)\|^2|\cF_k]\\
={}&(1-\beta_k)^2\|\vv^{k-1}-\nabla F(\vx^{k-1})\|^2+(1-\beta_k)^2\|\nabla F(\vx^{k-1})-\nabla F(\vx^k)\|^2+\\&\beta_k^2\E[\|\widetilde\nabla_k-\nabla F(\vx^k)\|^2|\cF_k]+2(1-\beta_k)^2\inprod{\vv^{k-1}-\nabla F(\vx^{k-1})}{\nabla F(\vx^{k-1})-\nabla F(\vx^k)}\\
\leq{}&(1-\beta_k)^2\|\vv^{k-1}-\nabla F(\vx^{k-1})\|^2+(1-\beta_k)^2\|\nabla F(\vx^{k-1})-\nabla F(\vx^k)\|^2+\\
&\beta_k^2\E[\|\widetilde\nabla_k-\nabla F(\vx^k)\|^2|\cF_k]+\beta_k(1-\beta_k)\|\vv^{k-1}-\nabla F(\vx^{k-1})\|^2+\frac{(1-\beta_k)^3}{\beta_k}\|\nabla F(\vx^{k-1})-\nabla F(\vx^k)\|^2\\
={}&(1-\beta_k)\|\vv^{k-1}-\nabla F(\vx^{k-1})\|^2+\beta_k^2\E[\|\widetilde\nabla_k-\nabla F(\vx^k)\|^2|\cF_k]+\frac{(1-\beta_k)^2\|\nabla F(\vx^{k-1})-\nabla F(\vx^k)\|^2}{\beta_k}\\
\leq{}&(1-\beta_k)\|\vv^{k-1}-\nabla F(\vx^{k-1})\|^2+\beta_k^2\E[\|\widetilde\nabla_k-\nabla F(\vx^k)\|^2|\cF_k]+\frac{\|\nabla F(\vx^{k-1})-\nabla F(\vx^k)\|^2}{\beta_k}.
\end{aligned}
\]
This completes the proof.
\end{proof}

Now we are ready to provide the convergence result for our momentum based SBPG. 
\begin{theorem}
Suppose Assumption \ref{relative-smooth-assumption}, \ref{variance-assumption} and \ref{momentum-assumption} hold. Let $\alpha_k=c\mu\beta_{k+1}$ for any $c\in(0,\frac{1}{2\sqrt{\mu\kappa}}]$. Then, it holds that
\begin{equation}
\E[\|\widetilde{\mathcal G}_{\alpha_r}(\vx^{r})\|^2]\leq\frac{\Delta_0+c\|\vv_0-\nabla F(\vx^0)\|^2+\sum_{k=0}^{N-1}\frac{\alpha_k^2\sigma^2}{c\mu^2m_k}}{\sum_{k=0}^{N-1}\frac{\mu\alpha_k}{8}},
\label{eq:rate-MSBPG}
\end{equation}
where $r$ is a random variable with distribution $\mathbb{P}\{r=k\}=\frac{\alpha_k}{\sum_{k=0}^{N-1}\alpha_k}$, for $k=0,...,N-1$.
\label{momentum-convergence-thm}
\end{theorem}
\begin{proof}
 From Lemma \ref{basic-descent-lemma-2} and Cauchy-Young‘s inequality, we have 
\[
\begin{aligned}
\Phi({\vx^{k+1}})&\leq\Phi({\vx^k})-\frac{1}{\alpha_k}{\cD_\phi}({\vx^k},{\vx^{k+1}})+\frac{\mu}{4\alpha_k}\|{\vx^k}-{\vx^{k+1}}\|^2+\frac{\alpha_k}{\mu}\|{\bm\varepsilon}_k\|^2\\
&\leq\Phi({\vx^k})-\frac{1}{2\alpha_k}{\cD_\phi}({\vx^k},{\vx^{k+1}})+\frac{\alpha_k}{\mu}\|{\bm\varepsilon}_k\|^2.
\end{aligned}
\]
where we have defined ${\bm\varepsilon}_k:=\nabla F({\vx^k})-\vv^k$. Summing the above inequality over $k=0,\ldots,N-1$ and rearranging the terms, we get
\[
\sum_{k=0}^{N-1}\frac{1}{2\alpha_k}{\cD_\phi}({\vx^k},{\vx^{k+1}})\leq\Phi(\vx^0)-\Phi^*+\sum_{k=0}^{N-1}\frac{\alpha_k}{\mu}\|{\bm\varepsilon}_k\|^2.
\]
By applying Lemma \ref{SMAE-recursion-lemma}, we can obtain the following inequality:
{
\[
\beta_k\E[\|{\bm\varepsilon}_{k-1}\|^2]
\;\leq\; \E[\|{\bm\varepsilon}_{k-1}\|^2] - \E[\|{\bm\varepsilon}_{k}\|^2]
+\beta_k^2\E[\|\widetilde\nabla_k-\nabla F({\vx^k})\|^2] 
+\E\left[\frac{\|\nabla F(\vx^{k-1})-\nabla F({\vx^k})\|^2}{\beta_k}\right].
\]}
Hence
\small{
\[
\sum_{k=0}^{N-1}\beta_{k+1}
\E[\|{\bm\varepsilon}_{k}\|^2]=\sum_{k=1}^N\beta_k\E[\|{\bm\varepsilon}_{k-1}\|^2]\leq\|{\bm\varepsilon}_0\|^2+\sum_{k=1}^N\beta_k^2\E[\|\widetilde\nabla_k-\nabla F({\vx^k})\|^2]+\sum_{k=1}^N\E\left[\frac{\|\nabla F(\vx^{k-1})-\nabla F({\vx^k})\|^2}{\beta_k}\right].
\]}
Since $\frac{\alpha_k}{\mu}=c\beta_{k+1}$ for some constant $c$, we get the following inequality:
\small{
\[
\sum_{k=0}^{N-1}\frac{1}{2\alpha_k}\E[{\cD_\phi}({\vx^k},{\vx^{k+1}})]\leq\Phi({\vx^0})-\Phi^*+c\left(\|{\bm\varepsilon}_0\|^2+\sum_{k=1}^N\beta_k^2\E[\|\widetilde\nabla_k-\nabla F({\vx^k})\|^2]+\sum_{k=1}^N\E\left[\frac{\|\nabla F(\vx^{k-1})-\nabla F({\vx^k})\|^2}{\beta_k}\right]\right).
\]
}
By using Assumption \ref{momentum-assumption}, we obtain that
\[
\frac{\|\nabla F({\vx^k})-\nabla F({\vx^{k+1}})\|^2}{\beta_{k+1}}\leq\frac{\kappa}{\beta_{k+1}}{\cD_\phi}({\vx^k},{\vx^{k+1}}).
\]
Combining above two inequalities, we get
\[
\sum_{k=0}^{N-1}\frac{1}{2\alpha_k}\E[{\cD_\phi}({\vx^k},{\vx^{k+1}})]\leq\Phi(\vx^0)-\Phi^*+c\left(\|{\bm\varepsilon}_0\|^2+\sum_{k=1}^N\beta_k^2\E[\|\widetilde\nabla_k-\nabla F({\vx^k})\|^2]+\sum_{k=0}^{N-1}\frac{\kappa}{\beta_{k+1}}\E[{\cD_\phi}({\vx^k},{\vx^{k+1}})]\right).
\]
{Since $c\leq\frac{1}{2\sqrt{\mu\kappa}}$ and $\frac{\alpha_k}{\mu}=c\beta_{k+1}$, we can deduce that 
 $\frac{c\kappa}{\beta_{k+1}}\leq \frac{1}{4\alpha_k}$}. Using this condition, we obtain the inequality:
\[
\sum_{k=0}^{N-1}\frac{1}{4\alpha_k}\E[{\cD_\phi}({\vx^k},{\vx^{k+1}})]\leq\Phi(\vx^0)-\Phi^*+c\left(\|{\bm\varepsilon}_0\|^2+\sum_{k=1}^N\beta_k^2\E[\|\widetilde\nabla_k-\nabla F({\vx^k})\|^2]\right). 
\]
Note that ${\cD_\phi}({\vx^k},{\vx^{k+1}})\geq\frac{\mu}{2}\|{\vx^k}-{\vx^{k+1}}\|^2=\frac{\mu\alpha_k^2}{2}\|\widetilde{\mathcal{G}}_{\alpha_k}({\vx^k})\|^2$ and by the definition of the random variable $a$, we get
\[
\E[\|\widetilde{\mathcal G}_{\alpha_a}(\vx^{a})\|^2]\leq\frac{\Phi^0-\Phi^*+ {c\|\bm{\varepsilon}_0\|^2}
+c\sum_{k=1}^N\frac{\beta_k^2\sigma^2}{m_k}}{\sum_{k=0}^{N-1}\frac{\mu\alpha_k}{8}},
\]     
which completes the proof.
\end{proof}

\begin{remark}
Now we give some remarks for Theorem \ref{momentum-convergence-thm}.
\begin{itemize}
\item[1.] When the sequence $\{\vx_k\}$ is bounded, an alternative to Assumption \ref{momentum-assumption} is to assume that $C=\R^d$ and that $\phi$ has a locally Lipschitz gradient, as made in \cite[Theorem 4.1]{bolte2018first} and \cite[Theorem 4.7]{latafat2022bregman}. Under these conditions, we can conclude that there exists a compact set $\cU$ containing $\{\vx_k\}$. Therefore, there exists a constant $L_{\phi,\cU}>0$ such that $\nabla\phi$ is Lipschitz continuous over $\cU$, and we can derive that $\|\nabla F(\vx)-\nabla F(\vy)\|^2\leq L_F^2L_{\phi,\cU}^2\|\vx-\vy\|^2\leq\frac{2L_F^2L_{\phi,\cU}^2}{\mu}\cD_\phi(\vx,\vy)$ holds. 
\item[2.] {The stationarity error for SBPG in Theorem \ref{converge-thm-vanilla-expt} is given by $\mathcal{O}\left(\frac{1}{\sum_{i=0}^k\alpha_i}+\frac{\sum_{i=0}^k{\frac{\alpha_i}{m_i}}}{\sum_{i=0}^k\alpha_i}\right)$, while for MSBPG in Theorem \ref{momentum-convergence-thm}, the error is $\mathcal{O}\left(\frac{1}{\sum_{i=0}^k\alpha_i}+\frac{\sum_{i=0}^k{\frac{\alpha_i^2}{m_i}}}{\sum_{i=0}^k\alpha_i}\right)$. Notably, compared to SBPG, even with a small constant mini-batch size $m_k$, MSBPG can still achieve convergence to a zero error bound by carefully selecting the stepsize sequence $\{\alpha_k\}$. A typical stepsize condition is $\sum_{k=0}^\infty\alpha_k=\infty,\;\sum_{k=0}^\infty\alpha_k^2<\infty$, which coincides with the classical stepsize condition ensuring a moderate decrease of the stepsize, as discussed in \citep{bertsekas2000gradient}. Thus, by incorporating the momentum technique, we can achieve improved convergence properties, particularly in terms of relaxed mini-batch size requirements, with minimal additional computational cost. This favorable convergence property theoretically supports the application of MSBPG for large-scale problems, such as deep neural networks training, without the use of very large mini-batch size. To illustrate, we provide a specific choice of $\{\alpha_k\}$ and $\{m_k\}$ that yields a convergence rate in terms of expected stationarity: Set $m_k=1$, $\alpha_k=\frac{c}{\sqrt{k+1}}$, the convergence rate is $\widetilde{\mathcal{O}}\left(\frac{1}{\sqrt{k}}\right)$ with logarithmic terms hidden.}
\end{itemize}
\end{remark}

Next, we present the sample complexity of MSBPG.
{
\begin{corollary}
\label{cor:complexity-MSBPG}
Given an accuracy level $\epsilon>0$, a constant $\alpha_0>0$, and any integer $m>0$, set $m_k=m$. Then to achieve an $\epsilon$-stationary point in expectation, at most $\bar N:=\ceil{\max\left\{(1+\mu^2)\alpha_0^2L_F^2,\frac{1}{\epsilon^4}\left(\frac{16(\Delta_0+c\norm{\varepsilon_0}^2)}{\mu\alpha_0}+\frac{8\alpha_0\sigma^2}{c\mu m}\right)\right\}}$ oracle evaluations are required. To achieve this complexity, we can choose $\alpha_k=\alpha=\frac{\alpha_0}{\sqrt{\bar N}}$.
\end{corollary}
\begin{proof}
Let the total number of oracle evaluations be $\bar N$. Then $N=\floor{\frac{\bar N}{m}}$ and $N\geq\frac{\bar N}{2m}$. By choosing $\alpha_k=\alpha=\frac{\alpha_0}{\sqrt{\bar N}}$ and using the definition of $\bar N$, we have that $\alpha_k<\frac{1}{L_F}\min\left\{1,\frac{1}{\mu}\right\}$, satisfying the conditions for Theorem \ref{momentum-convergence-thm}. Then we have
\[
\begin{aligned}
 \text{RHS of }\eqref{eq:rate-MSBPG}  \leq \; &\frac{8(\Delta_0+c\norm{\varepsilon_0}^2)m}{\mu\alpha\bar N}+\frac{8\sigma^2\alpha}{c\mu m}.
\end{aligned}
\]
By the choices of $\bar N$ and $\alpha$, we have $\text{RHS of }\eqref{momentum-convergence-thm}\leq\epsilon^2$. This completes the proof.
\end{proof} }
{We observe that MSBPG has the same complexity order as SBPG, with both achieving $\mathcal{O}(\epsilon^{-4})$, which, as shown in \citep{arjevani2023lower}, is the optimal bound and cannot be improved. However, to reach this complexity bound, SBPG requires a large mini-batch size, whereas MSBPG allows for an arbitrary mini-batch size. This relaxed requirement for the mini-batch size makes MSBPG more practical for training deep neural networks, as it is more memory-efficient. Therefore, the convergence properties of MSBPG are improved in terms of reduced mini-batch size requirements.}

{As a final remark for this section, in this work, we have only established convergence to a stationary point. However, as demonstrated by \citep{lee2019first} and \citep{panageas2019first}, the Bregman gradient method almost always converges to a local minimum when the loss function has the strict saddle property (i.e. all the saddle points of $f$ are strict saddle points). In general, finding the global minimum of a nonconvex function is NP-hard. However, in many applications, the objective function is  "benign" nonconvex, meaning that all local minima are global minima. Examples include matrix completion \citep{ge2016matrix}, matrix factorization \citep{chi2019nonconvex}, and phase retrieval \citep{sun2018geometric}, where convergence to a local minimum implies convergence to the global minimum. Additionally, when the objective function satisfies the Polyak–Łojasiewicz (PL) condition \citep{polyak1963gradient}, convergence to a stationary point also guarantees convergence to a global minimum.  In the experimental section, we observe that in some instances of training neural networks, the training loss can even approach zero, indicating near convergence to a global minimum.
}

\section{Application in training deep neural networks}\label{DNN implementation}
In this section, we present a detailed description of MSBPG applied to training deep neural networks. Throughout this section, we assume that the optimization domain $\overline{C}$ is the entire space $\R^d$, so that $\phi\in\cM(\R^d)$ and $F\in\cC^1(\R^d)$. For simplicity, we omit the explicit mention of the feasible set $\R^d$ in this section. In this context, we utilize a polynomial kernel function.

The optimization problem we consider here is given by:
\be 
\min_{\mW}\;\underbrace{\frac{1}{N}\sum_{i=1}^N\cL(\mathcal{DNN}(\mW,\vx_i),y_i)+\lambda_2\norm{\mW}_2^2}_{F(\mW)}+\lambda_1\|\mW\|_1,
\label{DNN-loss-model}
\ee 
where $\mathcal{DNN}({\mW},\vx)$ is the neural network function with training parameters $\mW$ and input data $\vx$, $\mathcal{L}$ is the loss function that measures the difference between the output of the neural network $\mathcal{DNN}({\mW},\vx_i)$ and the label $y_i$, $\lambda_2\norm{\mW}_2^2$ is the weight decay term, and $\lambda_1 \|\mW\|_1$ is the $L_1$ regularization term that is the sparsity induced operator and often used to avoid overfitting in training deep neural networks \citep{ng2004feature}. To illustrate the neural network function $\mathcal{DNN}({\mW},\vx)$, in the $L$-layer fully connected neural network, we have $\mW=[\mW_1,\mW_2,\cdots,\mW_L]$ and
\begin{equation}
\mathcal{DNN}(\mW,\vx)=\sigma_L(\mW_L(\sigma_{L-1}(\mW_{L-1}(...(\sigma_1(\mW_1\vx))...)))),
\label{DNN-def}
\end{equation}
where $\sigma_i$ is the nonlinear activation function. In this paper, we focus on smooth activation functions.

At the $k$-th iteration, MSBPG has the following update scheme:
\begin{eqnarray}
       {\vv^{k}} &=& (1-\beta_k)\vv^{k-1}+\beta_k\widetilde\nabla_k
       \\[5pt]
    {\mW}^{k+1} &=& \argmin_{\mW}\;\inprod{\vv^k}{{\mW}-{\mW}^k}+\frac{1}{\alpha_k}\cD_\phi({\mW},{\mW}^k)+\lambda_1\|\mW\|_1,
    \label{DNN-sub-ori}
\end{eqnarray}
where $\widetilde\nabla_k$ is mini-batch gradient of $F(\mW)$. Omitting all the constants, the subproblem takes the form of:
\be 
{\mW}^{k+1}=\argmin_{\mW}\;\phi({\mW})+\inprod{\vp^k}{{\mW}}+\alpha_k\lambda_1\|{\mW}\|_1,
\label{DNN-sub}
\ee
where $\vp^k=\alpha_k\vv^k-\nabla\phi({\mW}^k)$. Here we adopt the kernel function $\phi({\mW})=\frac{1}{2}\|{\mW}\|^2+\frac{\delta}{r}\|{\mW}\|^{r}$ ($r \ge 2$) for training neural networks, and
then we have an explicit solution for \eqref{DNN-sub} in Proposition \ref{DNN-subpro-solution-prop}. 
\begin{proposition}
Given $\vp^k \in \R^d$, positive constant $\alpha_k$, $\lambda$, and the kernel function $\phi({\mW})=\frac{1}{2}\|{\mW}\|^2+\frac{\delta}{r}\|{\mW}\|^{r}$ $(r \geq 2,\ \delta>0)$. The solution of the subproblem \eqref{DNN-sub} is given by 
\[
{\mW}^{k+1}=-t^*\vp^+,
\]
where $t^*$ is the unique positive real root of the equation
\begin{eqnarray}
(\delta\|\vp^+\|^{r-2})t^{r-1}+t-1=0,
\label{eq-t*}
\end{eqnarray}
and $\vp^+$ is given by 
\[
\vp^+ = \argmin_\vp\Big\{\frac{1}{2}\|\vp-\vp^k\|^2+\alpha_k\lambda\|\vp\|_1\Big\}
\]
which has an explicit expression given by $\vp_j^+ = \operatorname{sign}(\vp_j^k) \max(|\vp_j^k|-\alpha_k\lambda,0)$ for 
the $j$-th coordinate.
\label{DNN-subpro-solution-prop}
\end{proposition}
\begin{proof}
The optimality condition of \eqref{DNN-sub} is given by 
\[
0=\mW^{k+1}(1+\delta\|\mW^{k+1}\|^{r-2})+\vp^k+\alpha_k\lambda{\bm\Gamma}^k,\;\text{where }{\bm\Gamma}^k\in\partial\|\cdot\|_1(\mW^{k+1}).
\]
Let $\vp^+=\vp^k+\alpha_k\lambda{\bm\Gamma}^k$. By the optimality condition, we have $\mW^{k+1}=-t\vp^+$ for some positive scalar $t$, and 
\[
(-t-\delta\|\vp^+\|^{r-2}t^{r-1}+1)\vp^+=0.
\]
If $\vp^+\neq0$, then $\delta\|\vp^+\|^{r-2}t^{r-1}+t-1=0$. If $\vp^+=0$, then $\mW^{k+1}=-t\vp^+=0$. Since $t>0$, then we have $\partial\|\cdot\|_1(\mW^{k+1})=\partial\|\cdot\|_1(-t\vp^+)=-\partial\|\cdot\|_1(\vp^+)$. Recall the definition of $\vp^+$, we have 
\[
\vp^+=\vp^k+\alpha_k\lambda{\bm\Gamma}^k\in\vp^k-\alpha_k\lambda\partial\|\cdot\|_1(\vp^+),
\]
which is sufficient and necessary 
optimality condition of the convex optimization problem:
\[
\vp^+ = \argmin_\vp\left\{\frac{1}{2}\|\vp-\vp^k\|^2+\alpha_k\lambda\|\vp\|_1\right\}.
\]
This completes the proof by noting the the above minimization problem is the well-known soft threshold operator, see for example \citep{friedman2010regularization}. 
\end{proof}

In the absence of $L_1$-regularization, that is, when $\lambda_1=0$, 
then $\bm{p}^+ = \bm{p}^k$ and 
the update formula for MSBPG  at the $k$-th iteration simplifies to ${\mW}^{k+1}=-t^* \vp^k$, where $t^*$ is the positive root of the equation \eqref{eq-t*}. In this case, $\mW^{k+1} = t^*(\nabla\phi(\mW^k)-\alpha_k\vv^k)$. Furthermore, if we choose the kernel function simply as the square of Euclidean distance, i.e. $\delta=0$, then SBPG  reduces to SGD with momentum. Specifically, we have $t^*=1$ and the update ${\mW}^{k+1}={\mW}^k-\alpha_k\vv^k$.

\vspace{1em}
\noindent\textbf{Determining degree of kernel function}\quad We now turn our attention to selecting an appropriate parameter $r$ for the kernel function. Intuitively, in order to bound the Hessian of the loss function in \eqref{DNN-loss-model}, particularly when the number of layers $L$ in \eqref{DNN-def} is large, $r$ should also be chosen larger, so that $\nabla^2F\preceq\frac{1}{\alpha}\nabla^2\phi$ holds globally for some $\alpha>0$. However, this can lead to numerical issues when computing when computing $\|{\mW}\|^{r-2}$. This problem can be avoided if the deep neural network exhibits some special structure such that a moderate $r$ can make $F(\mW)$ smooth adaptable with respect to $\phi(\mW)$. For simplicity of analysis, we assume all the given labels $y_i$ as zero and consider a sum of squares loss function. Then, we have a two-layer model defined as follows:
\be 
\min_{\mW=(\vu,\vv)}\;F(\mW)=\frac{1}{2}\sum_{i=1}^N\left(\left\|\sigma\left(\Mat(\vu)(g_i(\vv))\right)\right\|^2\right)+\frac{\lambda_2}{2}(\norm{\vu}^2+\norm{\vv}^2),
\label{general-two-layer-model}
\ee 
where $\vv\in\R^n$, $\vu\in\R^{km}$, $g_i:\R^n\rightarrow\R^m$, $\sigma:\R\rightarrow\R$, $\Mat(\vu)\in\R^{k\times m}$ and $\sigma(\cdot)$ is a coordinate-wise operator. Notably, any deep neural network can be reformulated as the two-layer model in \eqref{general-two-layer-model}. For instance, if we define $\vv=(\mW_1,...,\mW_{L-1})$, $\vu={\rm Vec}(\mW_L)$, $g_i(\mW_1,...,\mW_{L-1})=\sigma_{L-1}(\mW_{L-1}(...(\sigma_1( \mW_1\vx_i))...))$, then model \eqref{DNN-def} can be transformed into \eqref{general-two-layer-model}. We make the following assumptions in this section, which guarantees that we can find a polynomial kernel function $\phi$ with a moderate degree, such that $F$ in \eqref{general-two-layer-model} is smooth adaptable to $\phi$.
\begin{assumption}\label{act-func}
 $\sigma$ is twice differentiable and $\sigma'$ and $\sigma\cdot\sigma''$ are globally bounded.
\label{bound_assumption_activation}
\end{assumption}
\begin{assumption}
 Each $g_i$ is twice differentiable. All partial derivatives of order zero, one, and two of $g_i$ are globally bounded.
\label{bound_assumption_bn}
\end{assumption}

\begin{remark}
Now we give some remarks on the above assumptions.
\begin{itemize}
\item[1.] Assumption \ref{bound_assumption_activation} is typically valid for various commonly used smooth activation functions.  For example, the sigmoid function $\sigma(x)=\frac{1}{1+e^{-x}}$ satisfies global boundedness for both $\sigma$ and $\sigma''$. Certain activation function may not have bounded function value, such as {GELU} \citep{hendrycks2016gaussian}, which takes the formulation of $\sigma(x)=x\Phi(x)$ where $\Phi$ is the standard Gaussian cumulative distribution function. Nonetheless, the product $\sigma\cdot\sigma''$ is globally bounded. Another type of activation function satisfying Assumption \ref{bound_assumption_activation} is the smoothed ReLU function, for example, the following smoothed ReLU function, which will be considered in our numerical experiments:
\[
\sigma_\epsilon(x) = \left\{
\begin{array}{cc}
     0 & x \le 0 \\
     x^3\left(\frac{1}{\epsilon^2}-\frac{x}{2\epsilon^3}\right) & 0 < x \le \epsilon \\
     x - \frac{\epsilon}{2} & x > \epsilon.
\end{array}\right.
\]
We observe that as $\epsilon$ tends to zero, $\sigma_\epsilon$ converges to the ReLU function. It is straightforward to verify that $\sigma_\epsilon\cdot\sigma_\epsilon''$ is globally bounded.  
Specifically, $\frac{3}{4}$ is a uniform bound on $\sigma_\epsilon\cdot\sigma_\epsilon''$ for $\epsilon\in(0,\frac{1}{2})$. 

\item[2.] In many popular neural network frameworks, batch normalization (BN) layers \citep{ioffe2015batch} are often used before the fully connected layers. For example, in the VGG \citep{simonyan2014very} and ResNet \citep{he2016deep}, BN layers are usually used before the last linear layer. In this case, we can treat all layers except the last one as one layer, which can be modeled as \eqref{general-two-layer-model}. It is expected that the BN layer can make the function $g_i$ sufficiently smooth, thereby satisfying  Assumption \ref{bound_assumption_bn}.
\end{itemize}
\end{remark}

By applying the chain rule, we can compute the Hessian of $F$ and determine a suitable degree parameter $r$ in the kernel function, which will ensure that $\nabla^2F$ is bounded by $\nabla^2\phi$ globally. Consequently, $F$ is smooth adaptable with respect to $\phi$.
In order to compute the Hessian of $F$, two formulas are required, which can be verified directly.\vspace{-0.1cm}
\begin{lemma}
Let $\vu\in\R^{km}$, $\vg\in\R^m$, $\mA \in\R^{n\times m}$ and $\vb\in\R^k$. Consider two linear maps: $\vu\mapsto\Mat(\vu)\vg$ and $\vu\mapsto \mA(\Mat(\vu))^T\vb$, then, the Jacobian of the two maps are given by
\[
J_{\vu}\left[\Mat(\vu)\vg\right]=\vg^T\otimes\mathbb{I}_k,
\]
\[
J_{\vu}[\mA(\Mat(\vu))^T\vb]=\mA \otimes \vb^T.
\]
\label{Jac-Hess-aux-lemma}
\end{lemma}\vspace{-0.5cm}

\begin{proposition}
Suppose Assumptions \ref{bound_assumption_activation} and \ref{bound_assumption_bn} hold. Then, for any given $\delta>0$ and any $r\geq4$, the function $F$ defined in \eqref{general-two-layer-model} is smooth adaptable with respect to $\phi(\mW)=\frac{1}{2}\|\mW\|^2+\frac{\delta}{r}\|\mW\|^r$.
\label{DNN-kernel-degree-prop}
\end{proposition}
\begin{proof}
We denote $\Mat(\vu)$ by $\mM$. The Jacobian of $g$ is denoted by $Jg$, while its transpose is denoted by $J^Tg$. $\mathbb{I}_k$ is $k\times k$ identity matrix. We only need to prove the single sample case, i.e. $N=1$. Using Lemma \ref{Jac-Hess-aux-lemma}, we can compute the Jacobian and Hessian of $F$ as follows:

{\noindent\bf Jacobian of $F$ :}\\
\be 
\begin{aligned}
\frac{\partial F}{\partial\vu}&=\left(g(\vv)\otimes\mathbb I_k\right)\left[\sigma^\prime(\mM g(\vv))\circ\sigma(\mM g(\vv))\right]+\lambda_2\vu,\\
\frac{\partial F}{\partial\vv}&=J^Tg(\vv)\mM^T\left[\sigma^\prime(\mM g(\vv))\circ\sigma(\mM g(\vv))\right]+\lambda_2\vv.
\end{aligned}
\label{Jacobian_F}
\ee 

{\noindent\bf Hessian of $F$ :}\\
\be 
\begin{aligned}
&\frac{\partial^2F}{\partial\vu^2}=(1)+(2)+\lambda_2\mathbb{I}_{km},\\
\text{where }(1)&=\left(g(\vv)\otimes\mathbb I_k\right)\text{Diag}\big(\sigma(\mM g(\vv))\circ\sigma^{\prime\prime}(\mM g(\vv))\big)(\vg^T(v)\otimes\mathbb I_k)\\
(2)&=\left(g(\vv)\otimes\mathbb I_k\right)\text{Diag}\big(\sigma^\prime(\mM g(\vv))\circ\sigma^{\prime}(\mM g(\vv))\big)(\vg^T(v)\otimes\mathbb I_k).
\end{aligned}
\label{Hessian_Fuu}
\ee 

\be 
\begin{aligned}
&\frac{\partial^2F}{\partial\vu\partial\vv}=(1)+(2)+(3),\\
\text{where }(1)&=\left(J^Tg(\vv)\right)\otimes\left[\sigma^\prime(\mM g(\vv))\circ\sigma^{\prime}(\mM g(\vv))\right]^T\\
(2)&=J^Tg(\vv)\mM^T\text{Diag}[\sigma(\mM g(\vv))\circ\sigma^{\prime\prime}(\mM g(\vv))]\left(g^T(\vv)\otimes\mathbb I_k\right)\\
(3)&=J^Tg(\vv)\mM^T\text{Diag}[\sigma^\prime(\mM g(\vv))\circ\sigma^{\prime}(\mM g(\vv))]\left(g^T(\vv)\otimes\mathbb I_k\right).
\end{aligned}
\label{Hessian_F_uv}
\ee 
\be 
\begin{aligned}
&\frac{\partial^2F}{\partial\vv^2}=(1)+(2)+(3)+\lambda_2\mathbb{I}_{n},\\
\text{where }(1)&=D^2g(\vv)\left[\mM^T[\sigma^\prime(\mM g(\vv))\circ\sigma(\mM g(\vv))]\right]=\sum d_i\nabla^2g_i(\vv)\\
(2)&=J^Tg(\vv)\mM^T\text{Diag}[\sigma(\mM g(\vv))\circ\sigma^{\prime\prime}(\mM g(\vv))]\mM Jg(\vv)\\
(3)&=J^Tg(\vv)\mM^T\text{Diag}[\sigma^\prime(\mM g(\vv))\circ\sigma^{\prime}(\mM g(\vv))]\mM Jg(\vv),
\end{aligned}
\label{Hessian_F_vv}
\ee 
where $\vd=\mM^T[\sigma^\prime(\mM g(\vv))\circ\sigma(\mM g(\vv))]$. Now, we are ready to prove this proposition. For any $\vw\in\R^{km+n}$ and $\vh=[\vh^u;\vh^v]\in\R^{km+n}$, it suffices to prove that $\inprod{\nabla^2F(\vw)\vh}{\vh}=\mathcal O(\inprod{\nabla^2\phi(\vw)\vh}{\vh})$. From \eqref{Hessian_Fuu}\eqref{Hessian_F_uv}\eqref{Hessian_F_vv} and Assumption \ref{bound_assumption_activation}, \ref{bound_assumption_bn}, we can easily get $\inprod{\nabla^2F(\vw)\vh}{\vh}=\mathcal O((1+\|\vw\|^2)\|\vh\|^2)$. On the other hand, $\nabla^2\phi(\vw)=I(1+\|\vw\|^{r-2})+(r-2)\|\vw\|^{r-4}\vw\vw^T$. Hence $\inprod{\nabla^2\phi(\vw)\vh}{\vh}\geq(1+\|\vw\|^{r-2})\|\vh\|^2$. So, we only require $r-2\geq2$. This completes the proof.   
\end{proof}

\noindent\textbf{Layerwise kernel function}\quad In Proposition \ref{DNN-subpro-solution-prop}, we use the kernel function $\phi({\mW})=\frac{1}{2}\|{\mW}\|^2+\frac{\delta}{r}\|{\mW}\|^{r}$, which applies the same Bregman distance across all layers of the deep neural network. However, different layers exhibit distinct geometric properties \citep{you2019large}, and computing $\|\mW\|^r$ with $r>2$ may result in numerical instability for neural networks with millions of parameters, such as in VGG \citep{simonyan2014very}. To take advantage of the layerwise structure of neural networks, we design a layerwise kernel function for a $L$-layer neural network as follows:
\be
\phi({\mW})=\sum_{i=1}^L\phi_i({\mW}_i),\quad \phi_i({\mW}_i)=\frac{1}{2}\|{\mW}_i\|^2+\frac{\delta}{r}\|{\mW}_i\|^{r}.
\label{kernel-nn}
\ee
Note that $\delta$ and $r$ can vary from layer to layer, here we take the same $\delta$ and $r$ for all layers for simplicity. With this structure, the Bregman distance takes the form $\cD_\phi=\sum_{i=1}^L\cD_{\phi_i}$. By incorporating this layerwise Bregman distance into the subproblem \eqref{DNN-sub-ori}, our MSBPG algorithm can be implemented in a layerwise manner. Details of the implementation are provided in Algorithm \ref{dnn-alg}.
\begin{algorithm}[ht] 
    \caption{Momentum based Stochastic Bregman Proximal Gradient (MSBPG) for training neural networks}
    \begin{algorithmic}[1]
        \State \textbf{Input: }Total number of iterations $K$, stepsize $\alpha_k>0$, momentum parameter $\beta_k\in(0,1)$, $\delta>0$ and integer $r\geq4$ to determine the kernel function $\phi$, and $\lambda_1.$
        \State \textbf{Initialize: }Set $\mW=\mW^0$, $\vv^0=0$.
        \For {$k=0, \cdots, K-1$}
        \State Compute mini-batch gradient $\widetilde\nabla_k$ of $F$;
        \State Compute SMAE: $\vv^k=(1-\beta_k)\vv^{k-1}+\beta_k\widetilde\nabla_k$;
        \For {$i=1,\ldots,L$}
        \State $\vp_i^k = \alpha_k \vv_i^k - \nabla\phi(\mW^{k}_i)$;
        \State $\vp_i^+ = \argmin_{\vp_i}\{\frac{1}{2}\|\vp_i-\vp_i^k\|^2+\alpha_k\lambda_1\|\vp_i\|_1\}$;
        \State Solve $(\delta \|\vp^+_i\|^{r-2})t_i^{r-1} + t_i -1 =0$ to get $t_i^k$;
        \State $\mW_i^{k+1} = - t_i^k \vp_i^+$;
        \EndFor
        \EndFor
        \State {\bf Output: }$\mW^1, \cdots, \mW^K$.
    \end{algorithmic}\label{dnn-alg}
\end{algorithm}




\vspace{1em}

\noindent\textbf{Mitigating gradient explosion}\quad In the training of deep neural networks, gradient explosion is a common undesired phenomenon, where the gradients of the loss function grow exponentially from layer to layer, leading to numerical instability or even collapse of the training process \citep{hochreiter1991untersuchungen,manchev2020target}. The reasons for gradient explosion include selecting a large stepsize and choosing an improper initialization for the model's parameters \citep{pascanu2013difficulty}. In the following, we will show that MSBPG  provides a novel approach to mitigate gradient explosion. Considering MSBPG  without $L_1$-regularization, the update rule is given by:
\be
{\mW}_i^{k+1}=-t_i^k\vp_i^k=t_i^k\left((1+\delta\|{\mW_i}^k\|^{r-2}){\mW_i}^k-\alpha_k\vv_i^k\right),
\label{Bregman-GD}
\ee
where $t_i^k\in(0,1)$ is the unique positive root of 
\be 
\left(\delta\left\|(1+\delta\|{\mW}_i^k\|^{r-2}){\mW}_i^k-\alpha_k\vv_i^k\right\|^{r-2}\right)t^{r-1}+t-1=0.
\label{Bregman_sub_equation}
\ee
Combining \eqref{Bregman-GD} and \eqref{Bregman_sub_equation}, we have the following equivalent implicit update scheme for the $i$-th layer:
\be
\mW_i^{k+1}=\frac{1+\delta\|\mW_i^k\|^{r-2}}{1+\delta\|\mW_i^{k+1}\|^{r-2}}\mW_i^k-\frac{\alpha_k}{1+\delta\|\mW_i^{k+1}\|^{r-2}}\vv_i^k.
\label{implicit-sbpg}
\ee
It is observed in practice that with large stepsize or large initial point, the gradient $\vv_i^k$ tends to explode if no scaling or clipping is applied, while the norm of the weight $\|\mW_i^{k+1}\|$ also tends to be large. In \eqref{implicit-sbpg}, scaling the gradient by $\frac{1}{1+\delta\|\mW_i^{k+1}\|^{r-2}}$ prevents the weight $\mW_i^{k+1}$ from moving too drastically in the direction of the gradient, thereby controlling the rapid growth of its norm. At the same time, if the norm  $\|\mW_i^{k+1}\|$ does not change significantly, the coefficient of $\mW_i^k$ in \eqref{implicit-sbpg} will remain approximately 1. Thus, the implicit update \eqref{implicit-sbpg} provides automatic scaling of the gradient, effectively mitigating the rapid growth of the weight and preventing subsequent gradient explosion.

Experimental results in Section \ref{DNN-experiment} indeed verify MSBPG 's ability to mitigate gradient explosion for training deep neural networks. An intuitive illustration of SBPG 's ``pull-back" ability is given in Figure \ref{additional-figure} in Appendix \ref{app-fig}, and this ``pull-back" ability originates from the Bregman proximity model and the polynomial kernel function we adopt.

\section{Numerical experiments}\label{section-experiments}
In this section, we conduct numerical experiments to demonstrate the effectiveness and robustness of  MSBPG in comparison to some commonly used optimizers in deep learning. We assess the impact of stepsize and initial point selection on the performance of our method. Our experiments consist of two parts. In the first part, we use a quadratic inverse problem as a toy example to illustrate the capabilities of vanilla SBPG. The second part is the main focus of this section, where we evaluate the performance of MSBPG in training deep neural networks. The experiments for the quadratic inverse problem are conducted using MATLAB R2021b on a Windows workstation equipped with a 12-core Intel Xeon E5-2680 @ 2.50GHz processor and 128GB of RAM. For the deep learning experiments, we conducted the experiments using PyTorch running on a single RTX3090 GPU.

\subsection{Quadratic inverse problem}
The quadratic inverse problem, as formulated in \citep{bolte2018first}, is given by:
\[
\min\left\{\Phi(\vx):=\underbrace{\frac{1}{4}\sum_{i=1}^n(\inprod{A_i\vx}{\vx}-b_i)^2}_{F(\vx)}+\lambda R(\vx): \vx\in\R^d\right\},
\]
which has practical applications \citep{beck2013sparsity}, including the phase retrieval problem as a special case \citep{luke2017phase}. In this experiment, we consider the $L_1$ regularization $R(\vx)=\|\vx\|_1$ with $\lambda=1\times10^{-3}$, and solve the quadratic inverse problem using SBPG and stochastic {(Euclidean)} proximal gradient (SPG) method \citep{bertsekas2011incremental}. Notably, SPG is a special case of SBPG, where $\phi(\vx)=\frac{1}{2}\|\vx\|^2$. Since the smooth term in the objective function $F(\vx)$ does not admit a globally Lipschitz continuous gradient, we employ the kernel function $\phi(\vx)=\frac{1}{2}\|\vx\|^2+\frac{1}{r}\|\vx\|^r$ with $r=4$. It has been shown in \citep{lu2018relatively} that any $r\geq4$ guarantees that $F$ is $\phi$-smooth adaptable globally. Moreover, according to \citep{bolte2018first}, the smooth adaptable constant $L_F$ can be chosen such that $L_F\geq\sum_{i=1}^n(3\|A_i\|^2+\|A_i\||b_i|)$ for $r=4$. In this experiment, we randomly generate the data by the following ${\tt MATLAB}$ commands: 
\[
\begin{aligned}
&{\tt ai = randn(d,1);Ai = ai*ai';}\\
&{\tt x\_true = sprandn(d,1,density\_x); b\_i = x\_true'*(Ai*x\_true);}
\end{aligned}
\]

The true solution for the quadratic inverse problem is chosen as a sparse vector $\vx^*$ that satisfies $\inprod{A_i\vx^*}{\vx^*}=b_i$ for $i=1,\dots,n$. We set the mini-batch size for all algorithms to be $m=1$. To evaluate the effectiveness of each algorithm, we use the following criterion that takes into account the possibility of stationary points being local minimum or saddle points:
\begin{equation}
\epsilon_k=\max\left\{\|\cG({\vx^k})\|,\;\epsilon_{\tt obj}:=\frac{{\tt obj}_k-{\tt obj}_*}{1+{\tt obj}_*}\right\},
\label{eq:exp-qip-stationarity}
\end{equation}
 where ${\tt obj}_k=\Phi({\vx^k})$ and ${\tt obj}_*=\Phi(\vx^*)$. The term $\|\cG(\vx^k)\|$ measures the stationarity of the solution, while a small $\eps_{\tt obj}$ indicates that the solution is a "nearly" global minimum.

We conduct experiments on a problem with data size $d=100$ and ${\tt density\_x=0.05}$. All methods are run until they reach an accuracy of $\epsilon_k\leq0.01$ within a time limit of 30 seconds. To ensure statistical significance, we run each algorithm 10 times and report the median value. The results are presented in Figure \ref{robust-figure}. For Figures \ref{robust-figure}(a), we randomly select initial points within a ball centered at the origin with radius $1\times10^{-2}$. We use the stepsize schedule of $\alpha_k=\max\left\{10^{-4},\frac{\alpha_0}{\sqrt{1+k}}\right\}$, where $\alpha_0$ is the initial stepsize. For Figure \ref{robust-figure}(b), we set constant stepsize schedule $1\times10^{-3}$. For Figures \ref{robust-figure}(c), we randomly select initial points within a ball centered at the origin with radius $1\times10^{-2}$. We use a constant stepsize schedule. To prevent excessively small stepsizes that can slow down all methods, we set a lower bound for the stepsize.

Figure \ref{robust-figure}(a) demonstrates that SBPG has a wider range of convergent stepsizes than SPG, indicating that SBPG is more robust in terms of stepsize selection. The effect of the initial stepsize on the performance of the algorithms is also depicted in this figure. Figure \ref{robust-figure}(b) highlights that SBPG is significantly more robust than SPG with respect to initial point selection, showing high resilience and preventing the training process from collapsing. Additionally, Figure \ref{robust-figure}(c) illustrates that increasing the degree $r$ in the kernel function raises the threshold for safe stepsizes. These observations are explained in Section \ref{DNN implementation}. When a large stepsize or a large initial point radius leads to a potential gradient explosion, the Bregman proximal mapping effectively pulls back the iterate, guiding it toward a more stable region and preventing gradient explosion.

\begin{figure}
\centering
\subfigure[Stepsize robustness]{
\includegraphics[width=4.5cm]{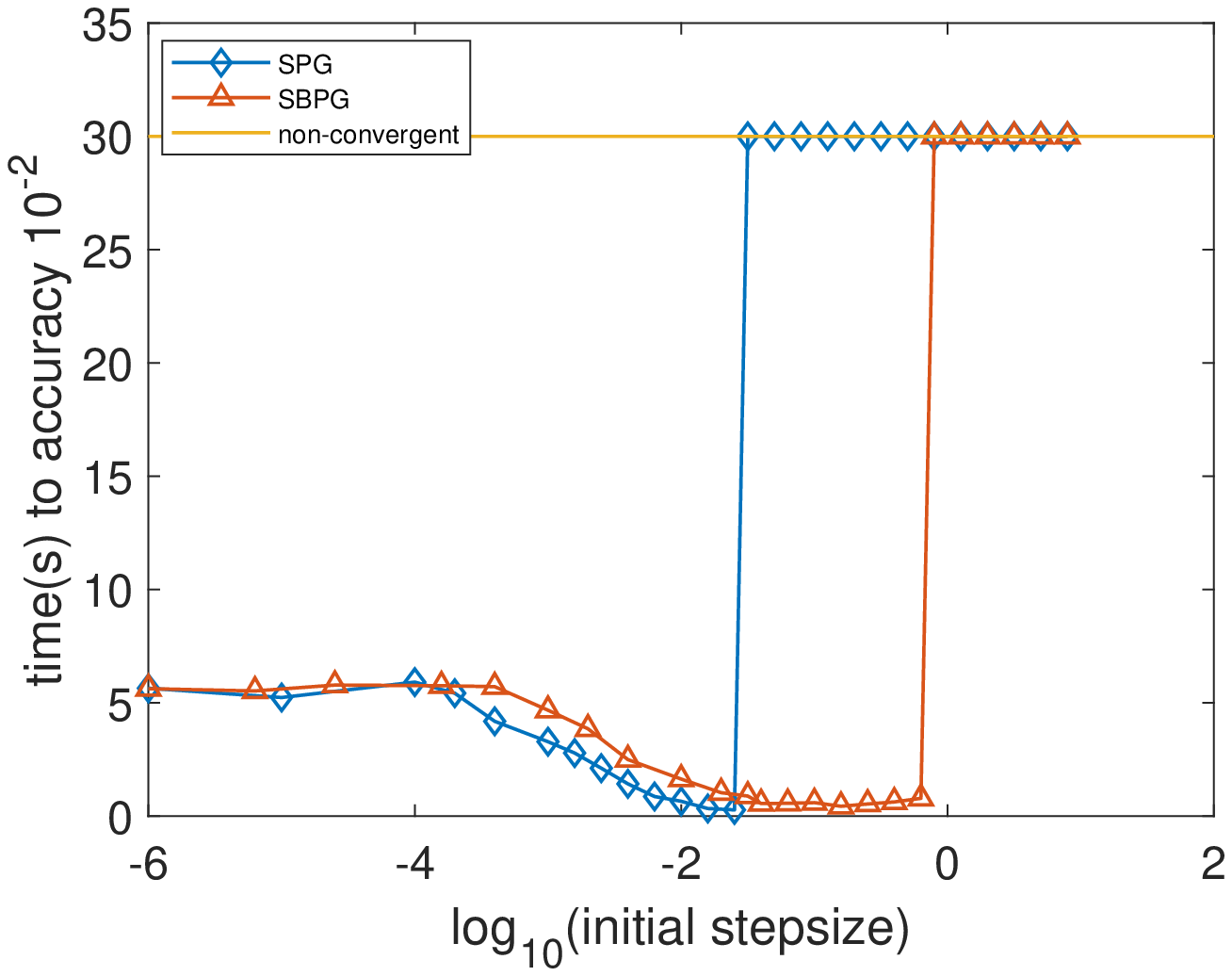}
}
\subfigure[Initial point robustness]{
\includegraphics[width=4.5cm]{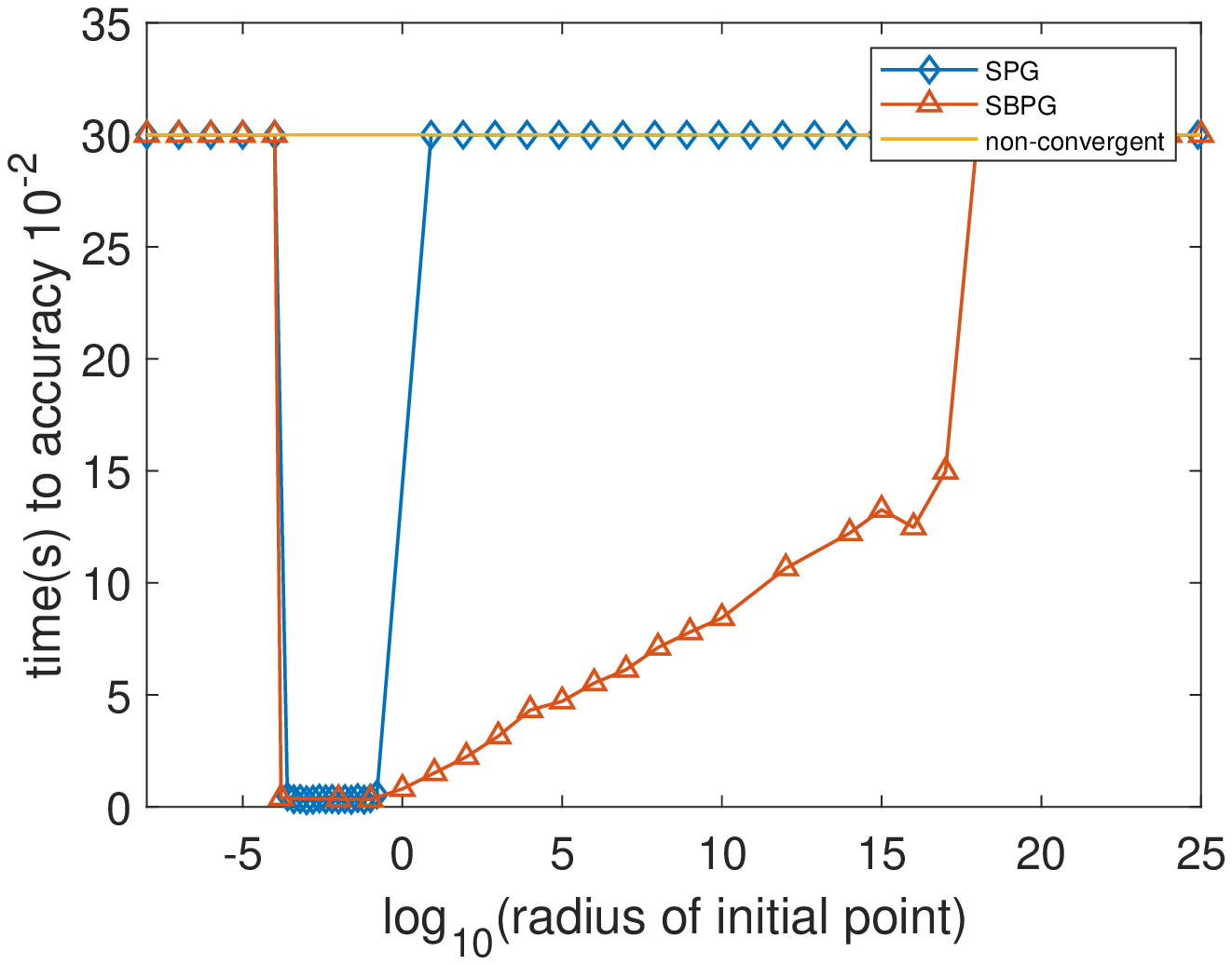}
}
\subfigure[Safe stepsize threshold]{
\includegraphics[width=4.5cm]{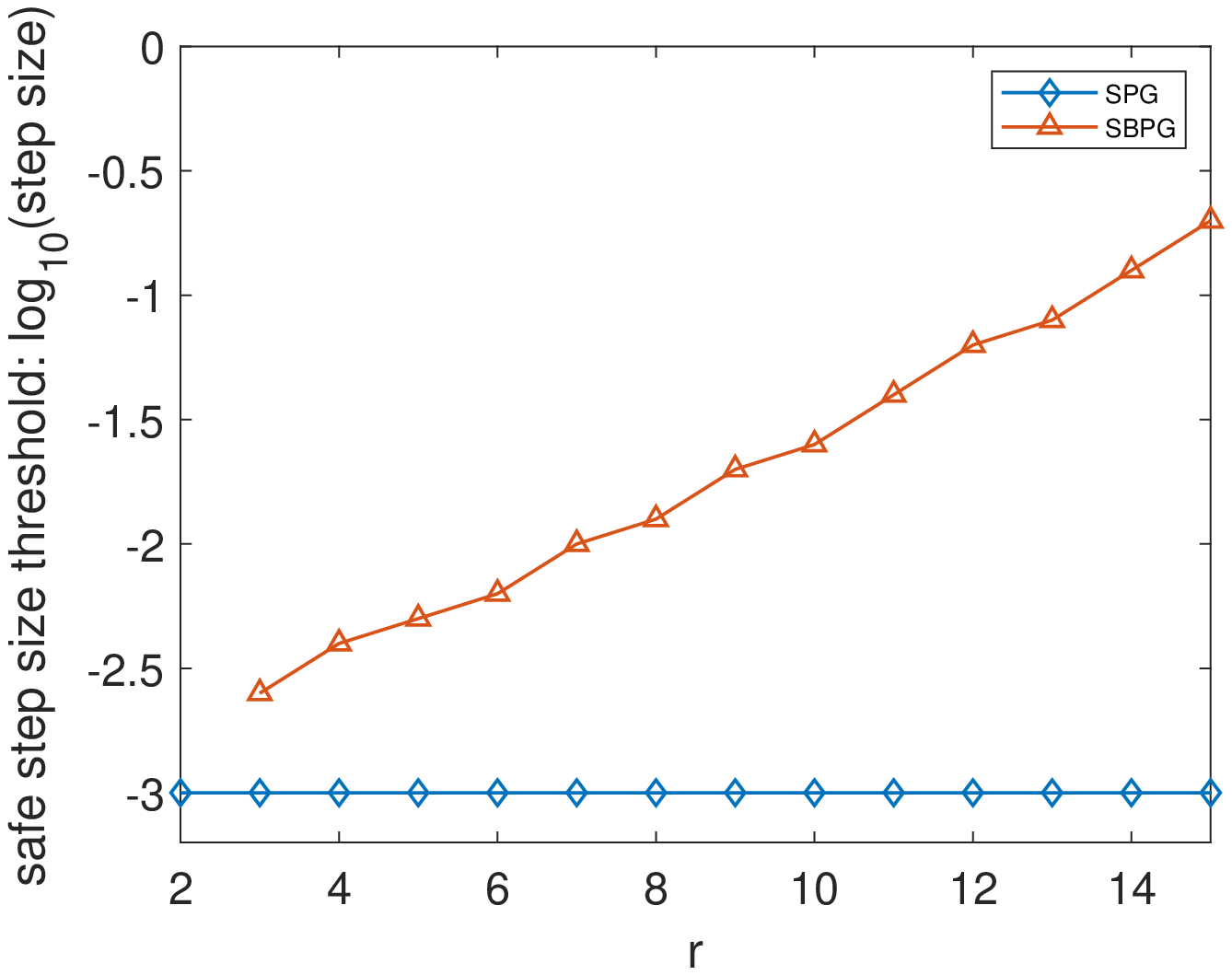}
}
\caption{Comparison of SBPG and SPG in terms of their robustness with respect to stepsize and initial point selction. A method is considered non-convergent if it fails to reach an accuracy of $\epsilon_k<10^{-2}$ within 30 seconds or if it collapses. Generally, choosing large stepsize and large radius for the initial point can cause an algorithm to collapse. The safe stepsize threshold is the maximum stepsize (constant schedule) that a method does not collapse. We run 10 tests for each algorithm and report the median of the results.}
\label{robust-figure}
\end{figure}

\subsection{Deep neural network} \label{DNN-experiment}

To evaluate MSBPG's performance in training deep neural networks, we consider a model with $L_2$ regularization  to enhance generalization:
\be 
\min_{\mW}\;\underbrace{\frac{1}{N}\sum_{i=1}^N\cL(\mathcal{DNN}(\mW,\vx_i),y_i)}_{F(\mW)}+\lambda_2\|\mW\|_2^2.
\label{exp-dnn-model}
\ee 
We employ MSBPG to solve this large-scale problem. {Following AdamW \citep{loshchilov2017decoupled}, we employ the strategy of decoupled weight decay. We use MSBPG to solve this large-scale optimization problem. Following the AdamW approach \citep{loshchilov2017decoupled}, we implement a decoupled weight decay strategy. Specifically, we compute the stochastic gradient only for the loss function $F$, treating $F$ and $L_2$ regularization separately. After performing the Bregman proximal mapping based on the stochastic gradient, we apply a weight decay step. }

The detailed algorithm is summarized in Algorithm \ref{dl-alg-sbpg}. At iteration $k$, MSBPG  first uses automatic differentiation to compute the mini-batch gradient $\widetilde{\nabla}_k$ of $F$. Then, it maintains a bias-corrected gradient estimator $\bar{\vv}^k$ \citep{kingma2014adam} and uses it to calculate the layerwise $\vp_i^k$. With $\vp_i^k$, MSBPG  solves a univariate equation to get $t_i^k$ and updates the weight of the $i$-th layer to $\mW_i^k$. In the end, MSBPG conducts decoupled weight decay.

\begin{algorithm}[ht] 
    \caption{MSBPG with $L_2$ regularization}
    \begin{algorithmic}[1]
        \State \textbf{Input: }Total number of training epochs $K$, momentum coefficient $\beta$, stepsize $\alpha_k$, weight decay coefficient $\lambda_2$, $\delta$ and $r$ to determine the kernel function $\phi$.
        \State \textbf{Initialize: } Set $\mW=\mW^0$, $\vv^0=\bf{0}$.
        \For {$k=1, \cdots, K$}
        \State Compute mini-batch gradient ${\widetilde{\nabla}_k}$;
        \State $\vv^k = \beta\vv^{k-1} + (1-\beta) {\widetilde{\nabla}_k}$, $\bar{\vv}^k = \vv^k / (1 - \beta^k)$;
       \For {$i=1, \cdots, L$} 
        \State $\vp_i^k = \alpha_k \bar{\vv}_i^k - \nabla\phi(\mW^{k-1}_i)$;
        \State Solve $(\delta \|\vp^k_i\|^{r-2})t_i^{r-1} + t_i -1 =0$ to get $t_i^k$;
        \State $\widetilde{\mW}_i^{k} = - t_i^k \vp_i^k$;
        \EndFor
        \State $\mW^k = \widetilde{\mW}^k - \alpha_k \lambda_2 \mW^{k-1}$;
        \EndFor
        \State {\bf Output: $\mW^1, \cdots, \mW^K$}
    \end{algorithmic}\label{dl-alg-sbpg}
\end{algorithm}

{We conducted experiments on several representative benchmarks, including VGG16~\citep{simonyan2014very}, ResNet34~\citep{he2016deep} on CIFAR10 dataset \citep{krizhevsky2009learning}, ResNet34~\citep{he2016deep}, DenseNet121~\citep{huang2017densely} on CIFAR100 dataset \citep{krizhevsky2009learning},
and LSTMs~\citep{schmidhuber1997long} on the Penn Treebank dataset~\citep{marcinkiewicz1994building}. } We compare MSBPG  with the most popular optimization algorithms used for training neural networks, including SGD \citep{sutskever2013importance}, Adam \citep{kingma2014adam}, and AdamW \citep{loshchilov2017decoupled}. Experimental results show that MSBPG has good convergence performance and generalization capacity for both the task that SGD dominates (image classification with CNNs) and the task that Adam dominates (language modeling with LSTMs). We also conducted experiments to compare MSBPG  with SGD on different initial stepsizes and different scales of the initial point. Our experimental results demonstrate the robustness of MSBPG  in training neural networks.

{
To further evaluate the performance of MSBPG on more recent neural network architectures, we conducted additional experiments, with the results presented in Appendix \ref{appendix:additional exp}.}

Before getting into the details of our experiments, we first clarify the activation function. The commonly used ReLU activation function in VGG, ResNet, and DenseNet is defined as $\operatorname{ReLU}(x) = \max(0, x)$, which is not continuously differentiable. To address this, we design a smooth approximation of ReLU with a parameter $\epsilon$, which is twice continuously differentiable and satisfies our Assumption \ref{act-func}:
\[
\sigma_\epsilon(x) = \left\{
\begin{array}{cc}
     0 & x \le 0 \\
     x^3\left(\frac{1}{\epsilon^2}-\frac{x}{2\epsilon^3}\right) & 0 < x \le \epsilon \\
     x - \frac{\epsilon}{2} & x > \epsilon.
\end{array}\right.
\]
The gradient of this activation function is given by:
\[
\sigma_\epsilon'(x) = \left\{
\begin{array}{cc}
     0 & x \le 0 \\
     x^2(\frac{3}{\epsilon^2}-\frac{2x}{\epsilon^3}) & 0 < x \le \epsilon \\
     1 & x > \epsilon.
\end{array}\right.
\]
As $\epsilon$ tends to 0, this smooth activation function converges to the standard ReLU function. We conducted experiments using VGG16 on the CIFAR-10 dataset, replacing all activation functions in VGG16 with $\sigma_\epsilon$. As shown in Figure \ref{act-func}, our algorithm MSBPG 's performance does not degrade as $\epsilon$ tends to 0. Therefore, in the subsequent experiments, we use the original neural network architectures with the ReLU activation function to evaluate our method MSBPG.

\begin{figure*}[th]
              \begin{center}
                     \setlength{\tabcolsep}{0.0pt}  
                     \scalebox{1}{\begin{tabular}{cccc}
                                   \includegraphics[width=0.25\linewidth]{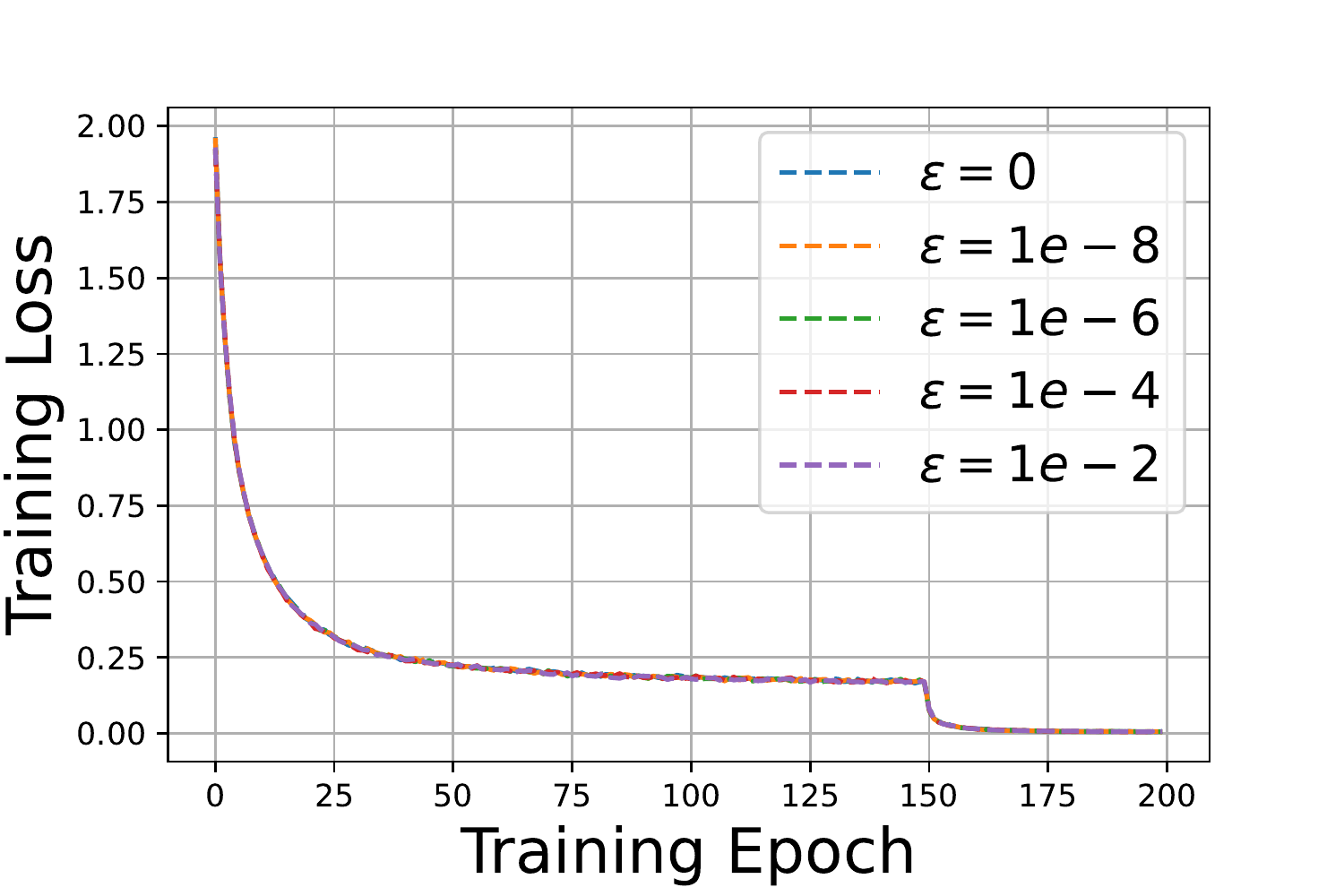}&
                                   \includegraphics[width=0.25\linewidth]{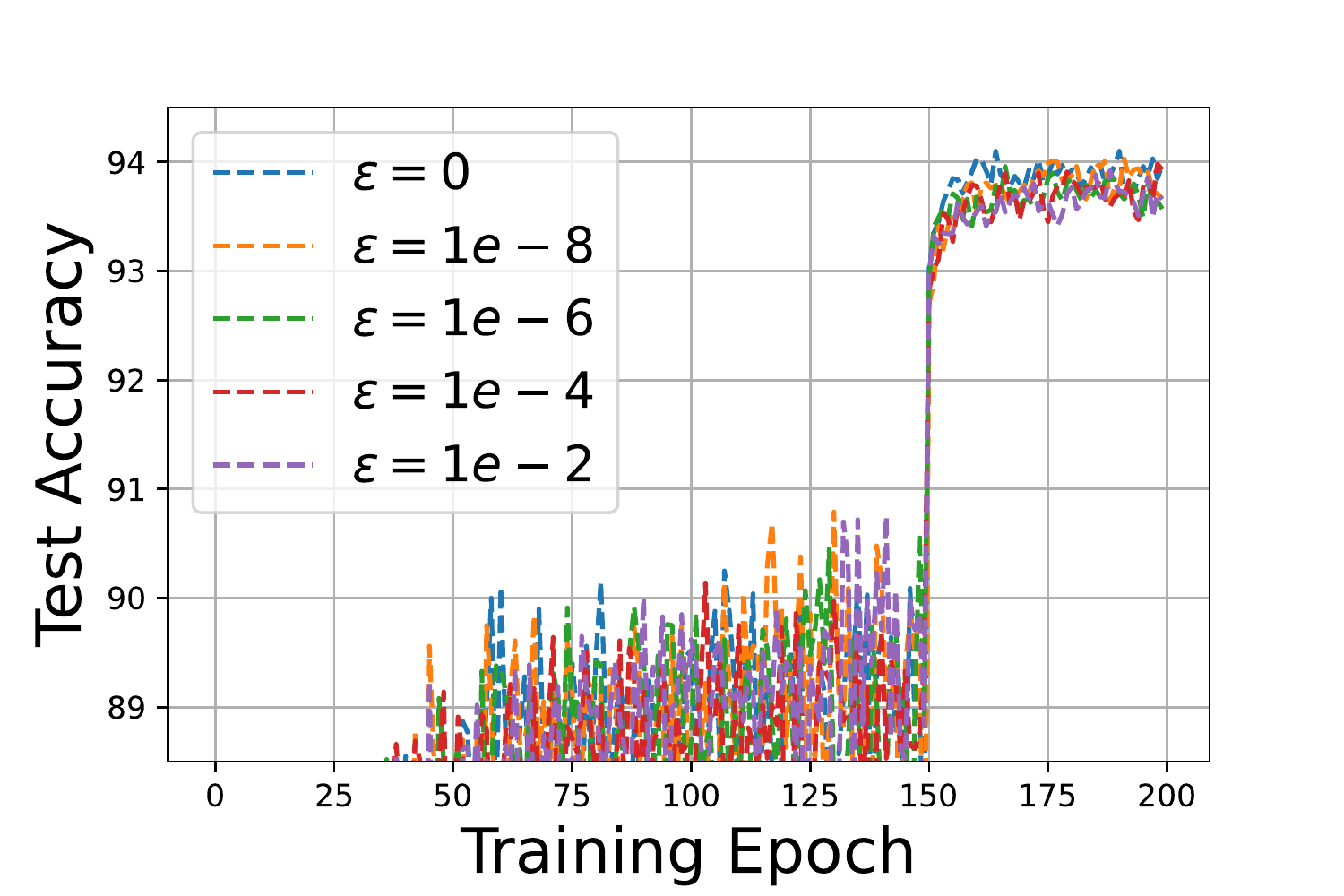}&
                                   \includegraphics[width=0.25\linewidth]{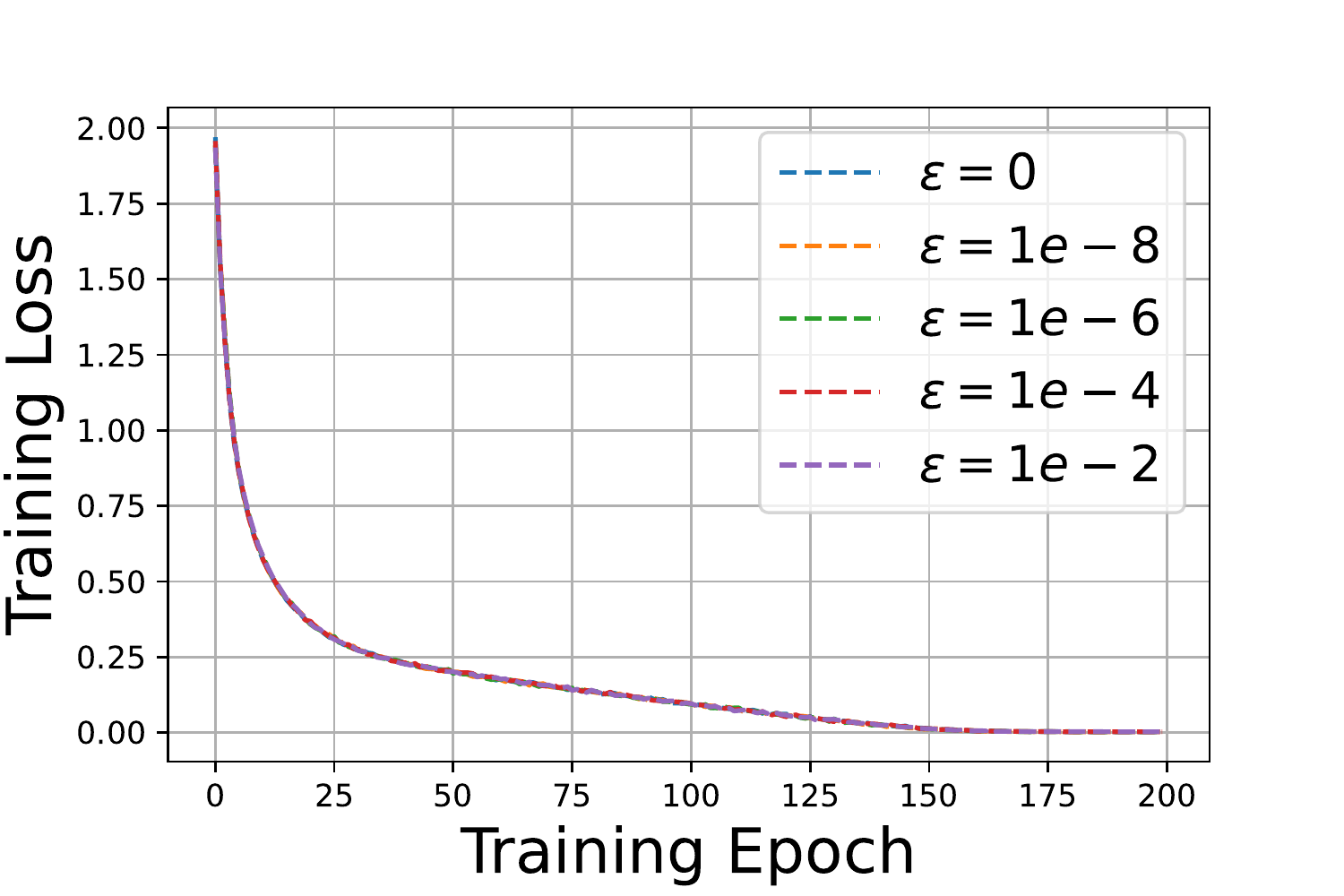} &
                                   \includegraphics[width=0.25\linewidth]{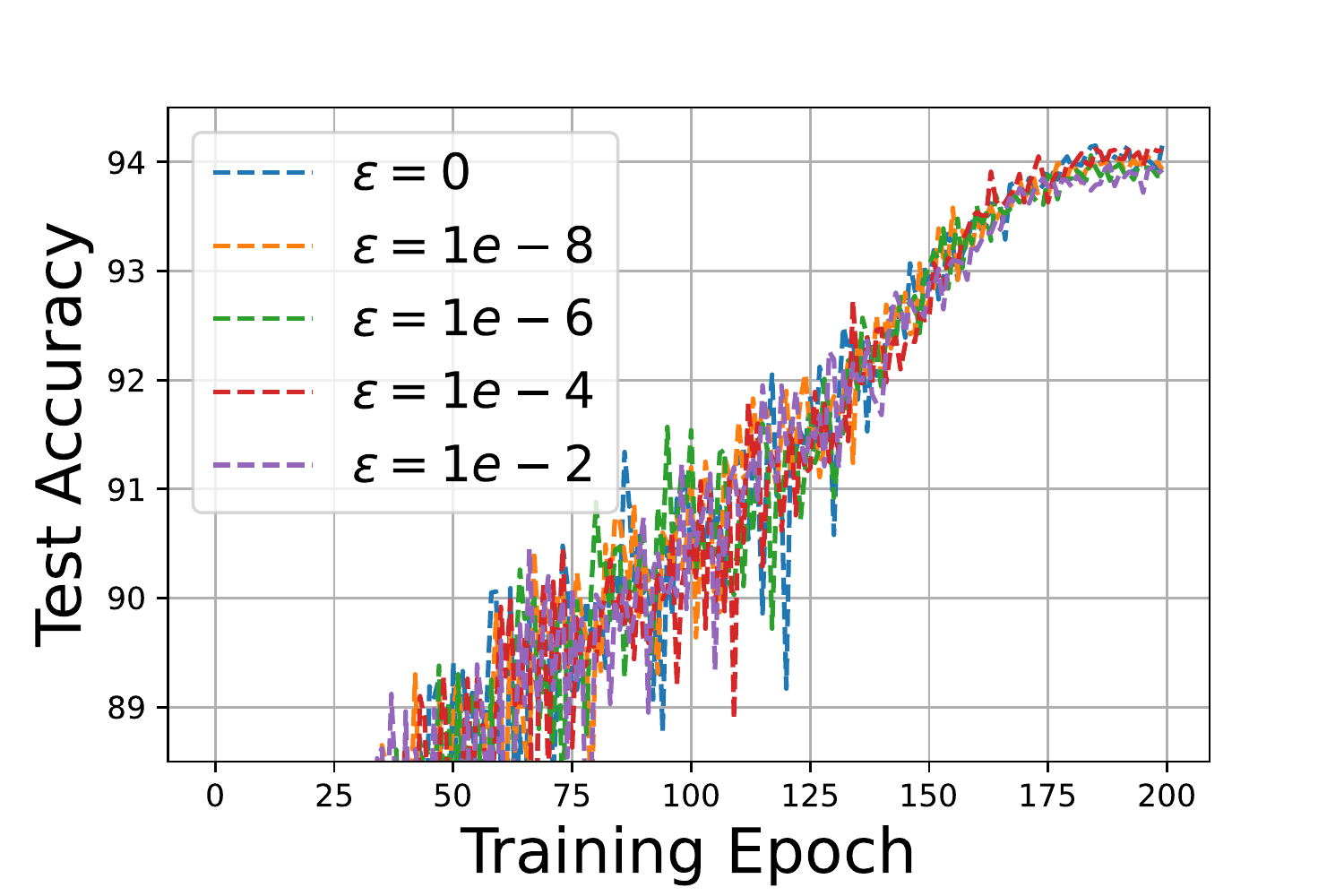}\\
                                   \multicolumn{2}{c}{\footnotesize{(a) Step decay.}} &  \multicolumn{2}{c}{\footnotesize{(b) Cosine annealing.}}\\
                     \end{tabular}}
              \end{center}
       \vspace{-0.8em}
              \caption{Training loss and test accuracy (\%) of VGG16 on CIFAR10 dataset under two frequently used training settings. Here the activation function of VGG16 adopts smoothed ReLU activation function $\sigma_\epsilon$ with different choices of $\epsilon$ ($\epsilon=0$ denotes adopting the original ReLU activation function).} \label{fig:act-func}
\end{figure*}

\paragraph{CNNs on image classification} We conducted experiments with VGG16 and ResNet34 on the CIFAR-10 dataset, and ResNet34 and DenseNet121 on the CIFAR-100 dataset. SGD usually has better generalization performance than adaptive gradient algorithms such as Adam and AdamW when training CNNs on image classification tasks. For our experiments, we utilized two common training strategies: reducing the stepsize to 10\% of its original value near the end of training \citep{zhuang2020adabelief, chen2021closing, luo2019adaptive}, and using a cosine annealing schedule for stepsizes \citep{loshchilov2016sgdr, loshchilov2017decoupled}. These strategies are designed to accelerate convergence and allow for a fair comparison of the generalization capacities of different optimizers. We use the default training hyperparameters of SGD, Adam, and AdamW in these settings \citep{he2016deep,zhuang2020adabelief,chen2021closing}, and set MSBPG 's learning rate (initial stepsize) as 0.1, momentum coefficient $\beta$ as 0.9, weight decay coefficient $\lambda_2$ as $1\times 10^{-3}$. For the layerwise kernel function $\phi_i({\mW}_i)=\frac{1}{2}\|{\mW}_i\|^2+\frac{\delta}{r}\|{\mW}_i\|^{r}$, we set $r=4,\ \delta=1\times 10^{-2}$ for VGG16 and $r=6,\ \delta=1 \times 10^{-3}$ for ResNet34 on CIFAR10 dataset, and $r=4,\ \delta=1\times 10^{-2}$ for ResNet34 and $r=4,\ \delta=1 \times 10^{-3}$ for DenseNet121 on CIFAR100 dataset. 

{Our convergence analysis demonstrates that MSBPG converges to a stationary point. As shown by \citep{lee2019first, panageas2019first}, the Bregman gradient method almost always converges to a local minimum for loss functions with the strict saddle property. When the neural network is highly overparameterized and exhibits a benign nonconvexity in the search region, stochastic first-order methods tend to find the global minimum, where the loss function value is 0.}

{From the experimental results in Figure \ref{fig:cnn1}, \ref{fig:cnn3}, \ref{fig:cnn2}, \ref{fig:cnn4}, we can observe that MSBPG attains almost zero training loss in all the training settings. This implies that our method can find the global minimum in these instances. Furthermore, MSBPG  consistently achieves the highest test accuracy for all experimental settings and attains at least 0.5\% test accuracy improvement compared with the second-best optimization algorithm. This generalization advantage of MSBPG can be attributed to the Bregman proximity model we adopt.  In Appendix B, we further discuss SBPG's performance from the perspective of algorithmic stability, which can influence the generalization gap. We demonstrate that the high-order polynomial kernel used helps reduce the generalization gap bound in high-dimensional scenarios, which partially explains the good generalization performance of our proposed methods.}

\begin{figure*}[th]
              \begin{center}
                     \setlength{\tabcolsep}{0.0pt}  
                     \scalebox{1}{\begin{tabular}{cccc}
                                   \includegraphics[width=0.25\linewidth]{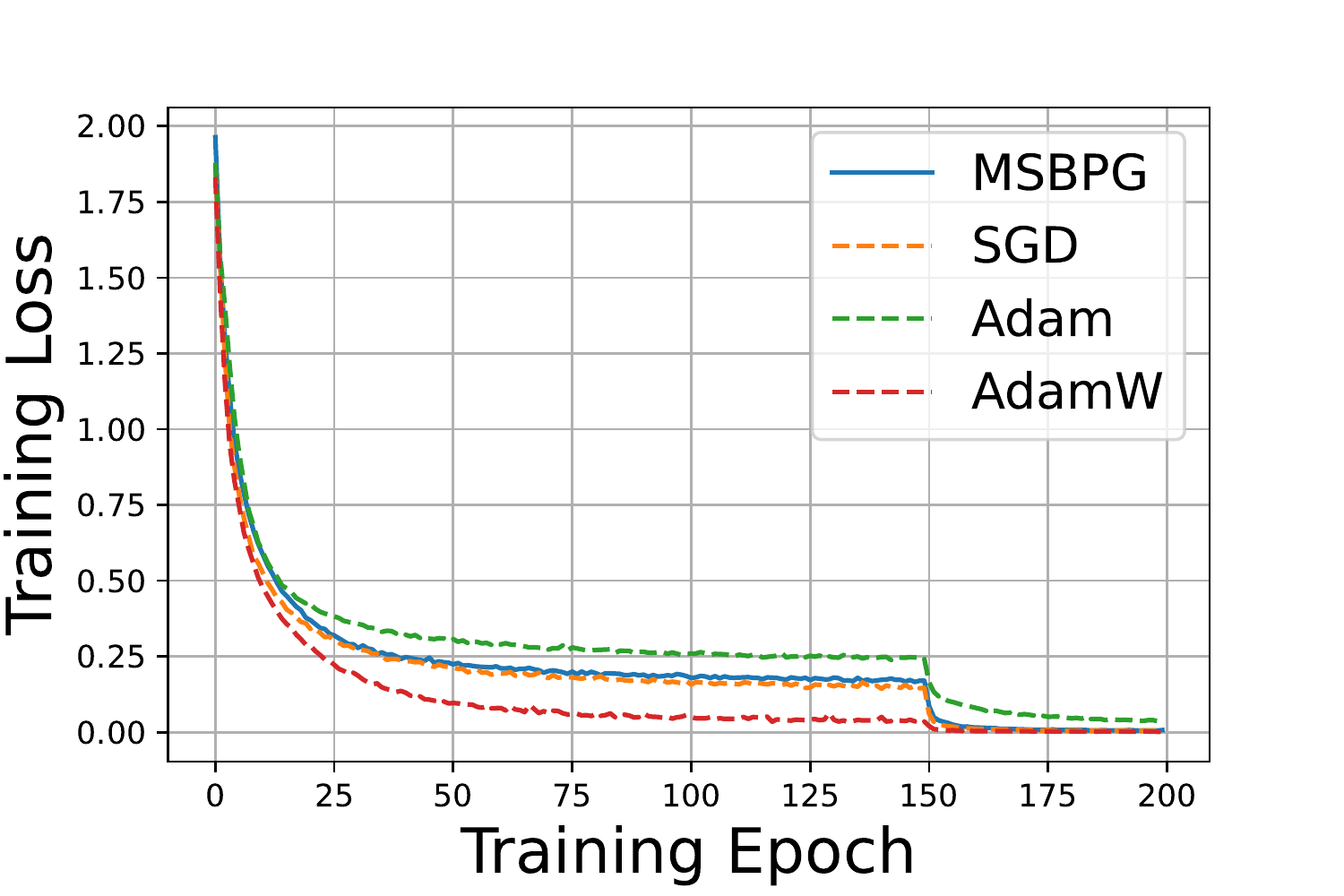}&
                                   \includegraphics[width=0.25\linewidth]{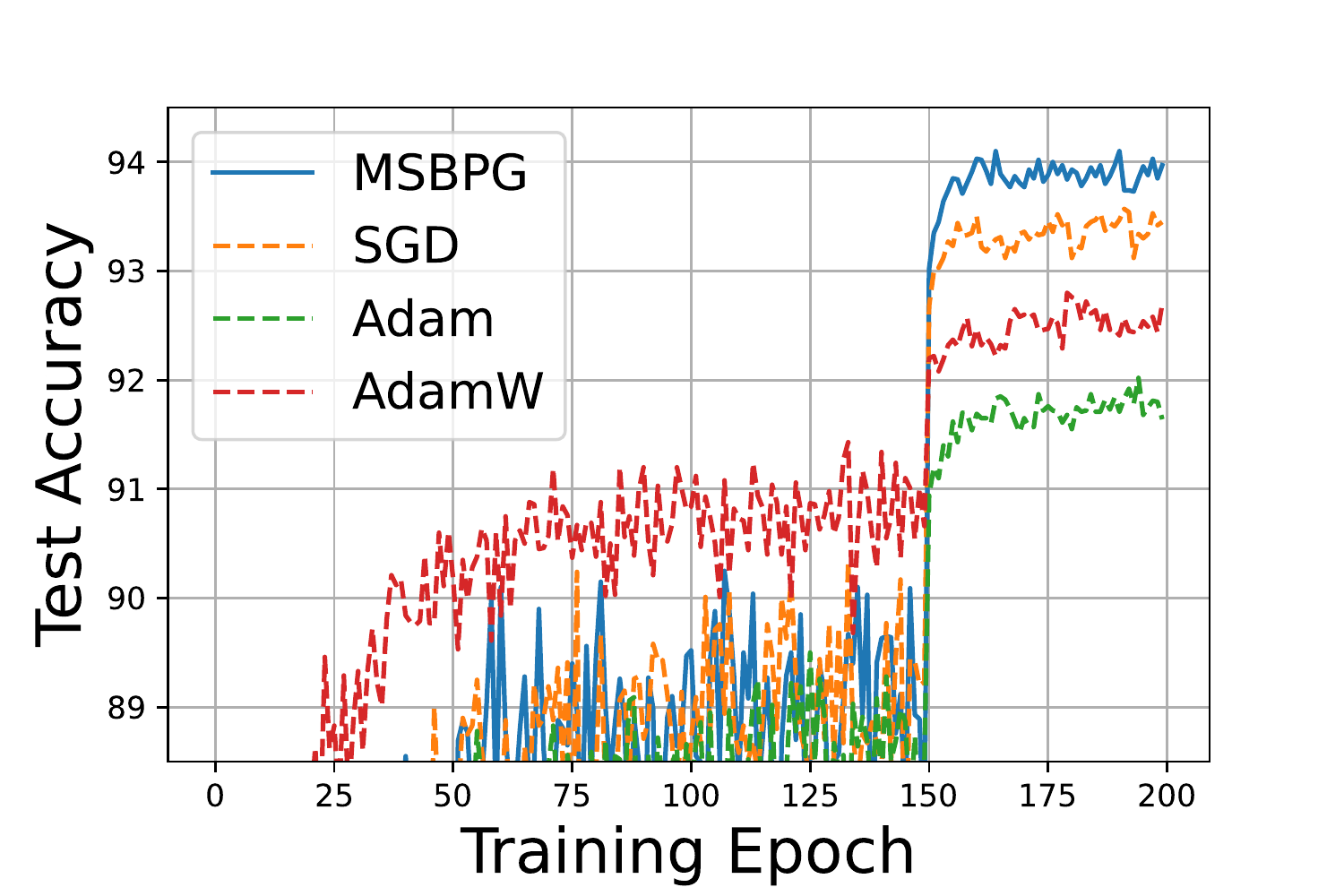}&
                                   \includegraphics[width=0.25\linewidth]{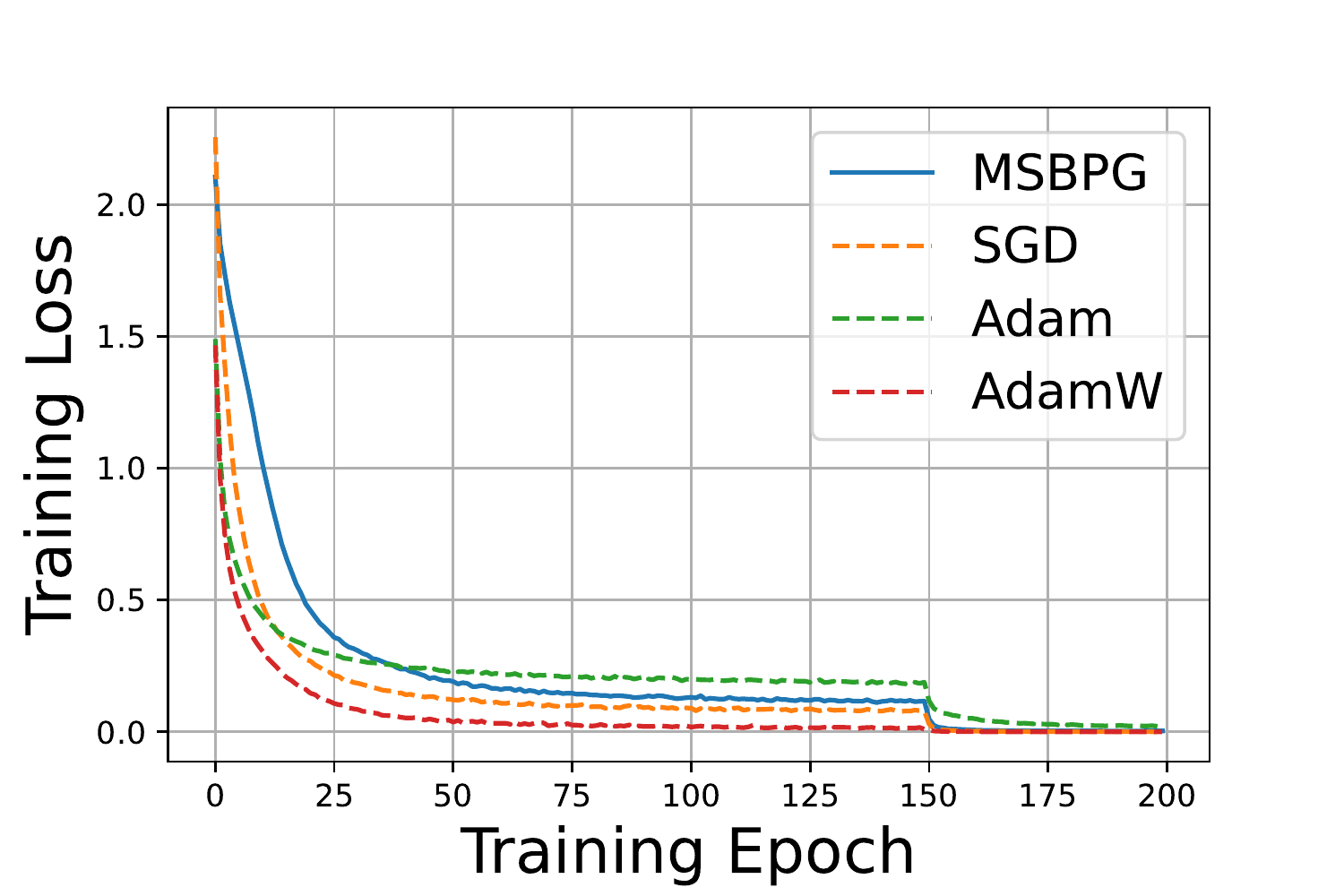} &
                                   \includegraphics[width=0.25\linewidth]{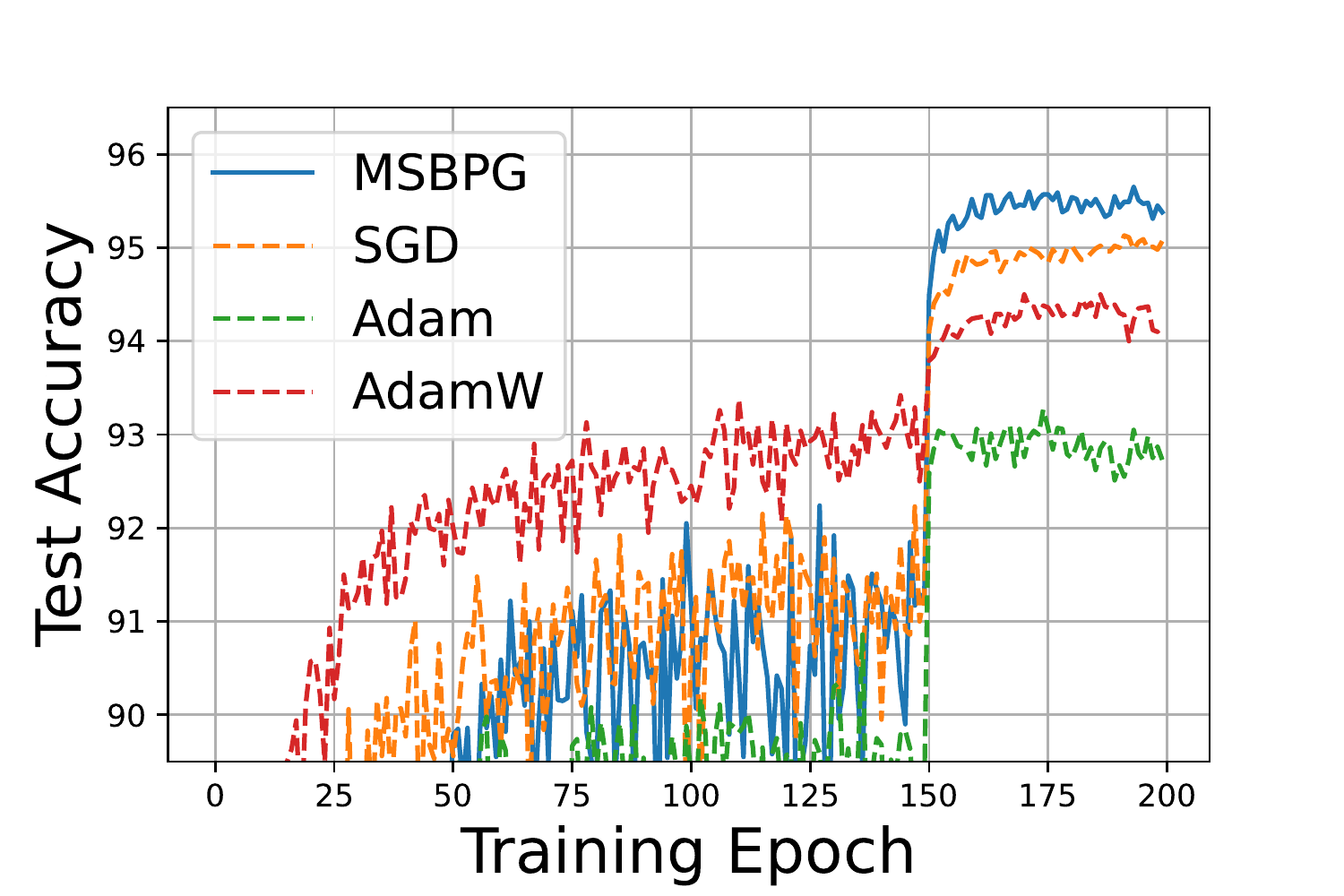}\\
                                          \multicolumn{2}{c}{\footnotesize{(a) VGG16 on CIFAR10}} &  \multicolumn{2}{c}{\footnotesize{(b) ResNet34 on CIFAR10}}\\
                     \end{tabular}}
              \end{center}
       \vspace{-0.8em}
              \caption{Training loss and test accuracy (\%) of CNNs on CIFAR10 dataset with learning rate reduced to 0.1 times of the original value at the 150th epoch.} \label{fig:cnn1}
       \end{figure*}

 \begin{figure*}[th]
              \begin{center}
                     \setlength{\tabcolsep}{0.0pt}  
                     \scalebox{1}{\begin{tabular}{cccc}
                                   \includegraphics[width=0.25\linewidth]{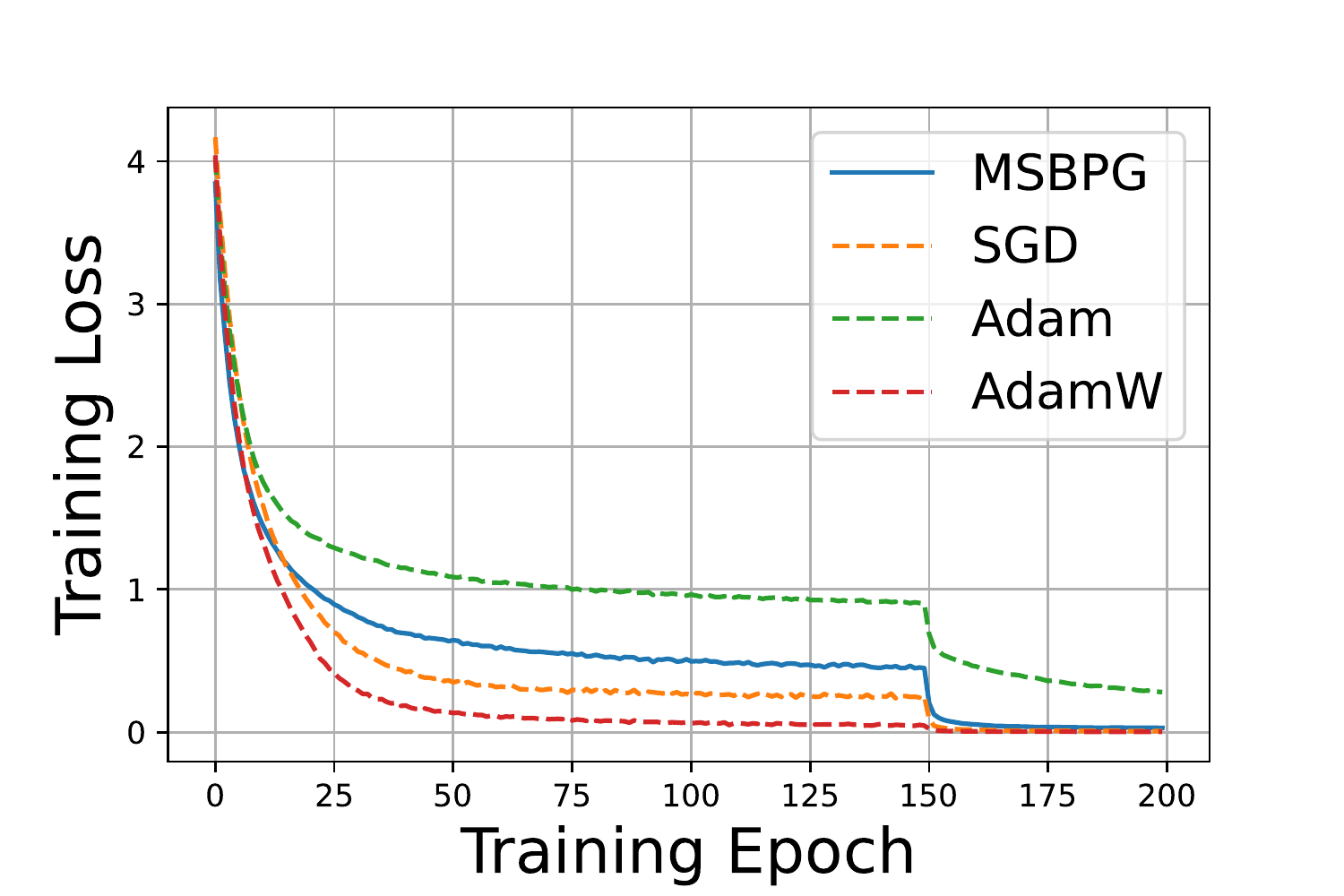}&
                                   \includegraphics[width=0.25\linewidth]{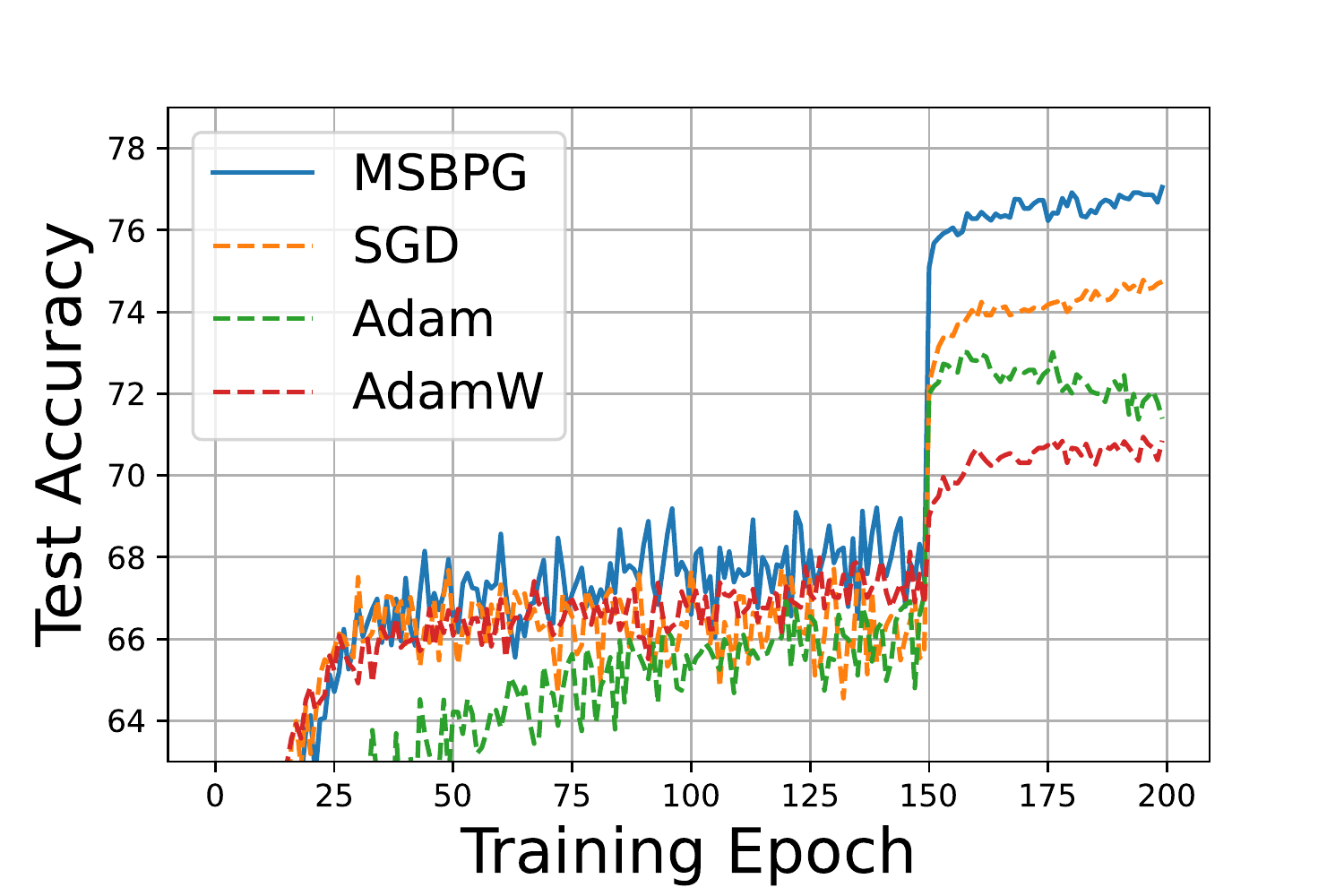}&
                                   \includegraphics[width=0.25\linewidth]{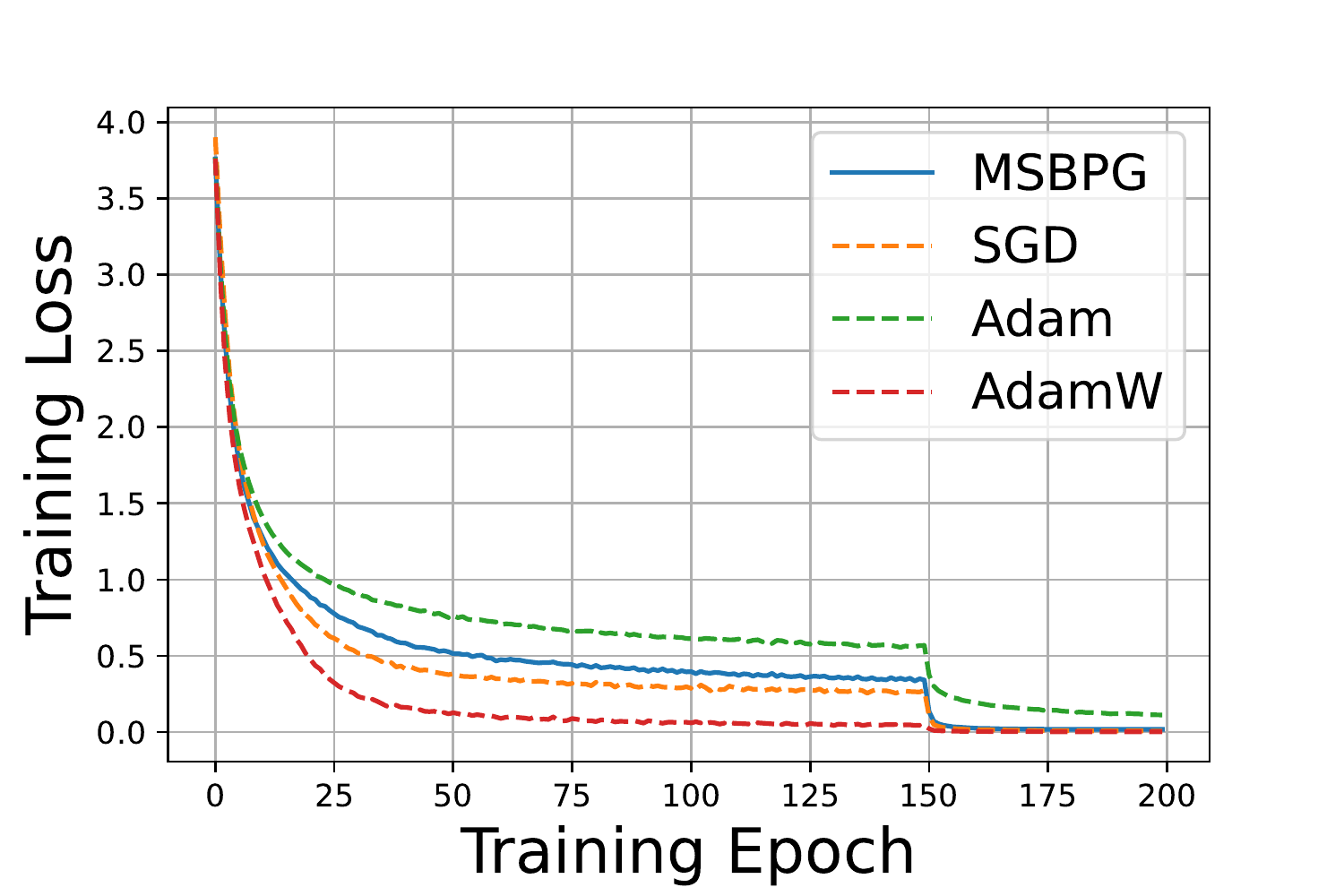} &
                                   \includegraphics[width=0.25\linewidth]{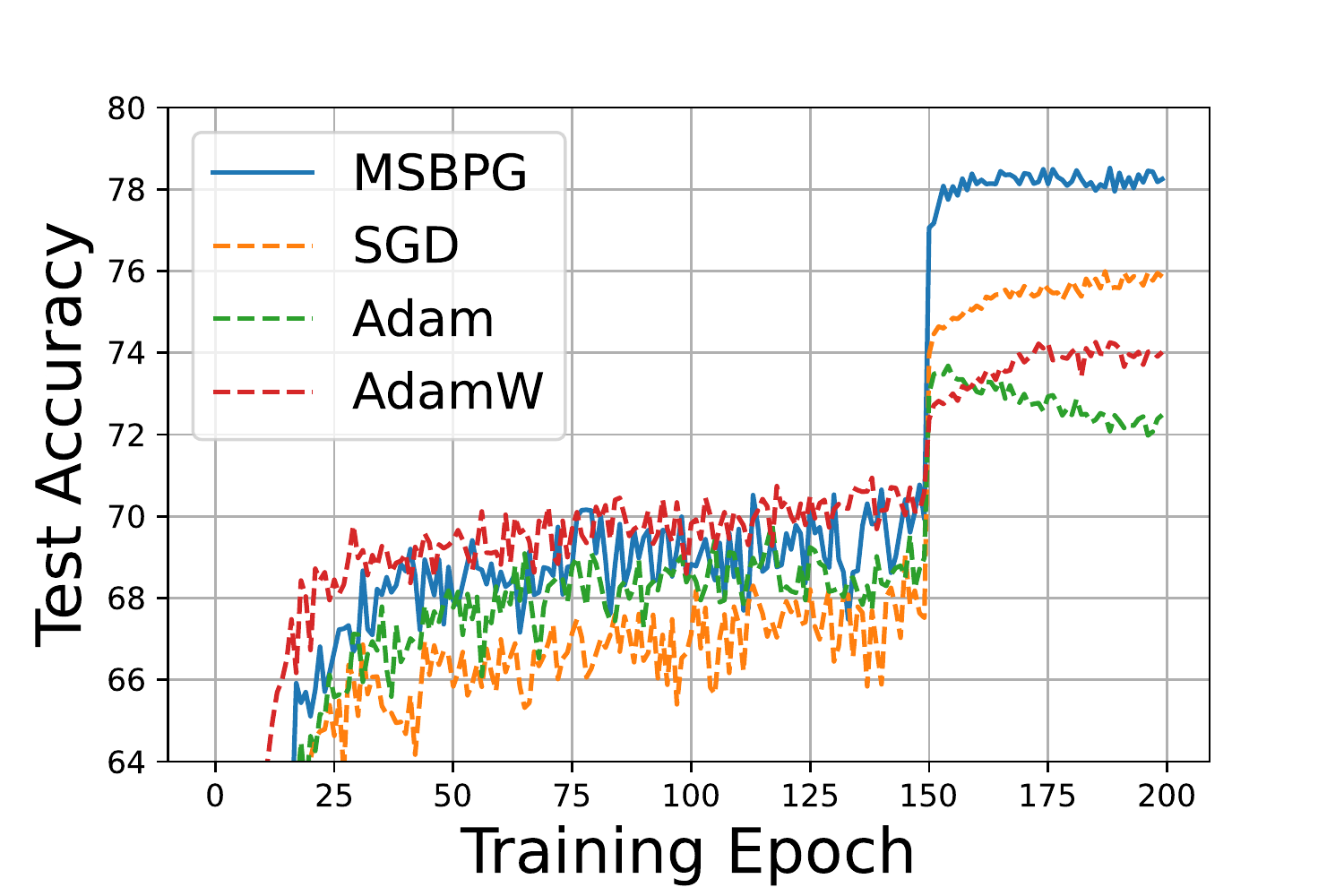}\\
                                          \multicolumn{2}{c}{\footnotesize{(a) ResNet34 on CIFAR100}} &  \multicolumn{2}{c}{\footnotesize{(b) DenseNet121 on CIFAR100}}\\
                     \end{tabular}}
              \end{center}
       \vspace{-0.8em}
              \caption{Training loss and test accuracy (\%) of CNNs on CIFAR100 dataset with learning rate reduced to 0.1 times of the original value at the 150th epoch.} \label{fig:cnn3}
       \end{figure*}

 \begin{figure*}[th]
              \begin{center}
                     \setlength{\tabcolsep}{0.0pt}  
                     \scalebox{1}{\begin{tabular}{cccc}
                                   \includegraphics[width=0.25\linewidth]{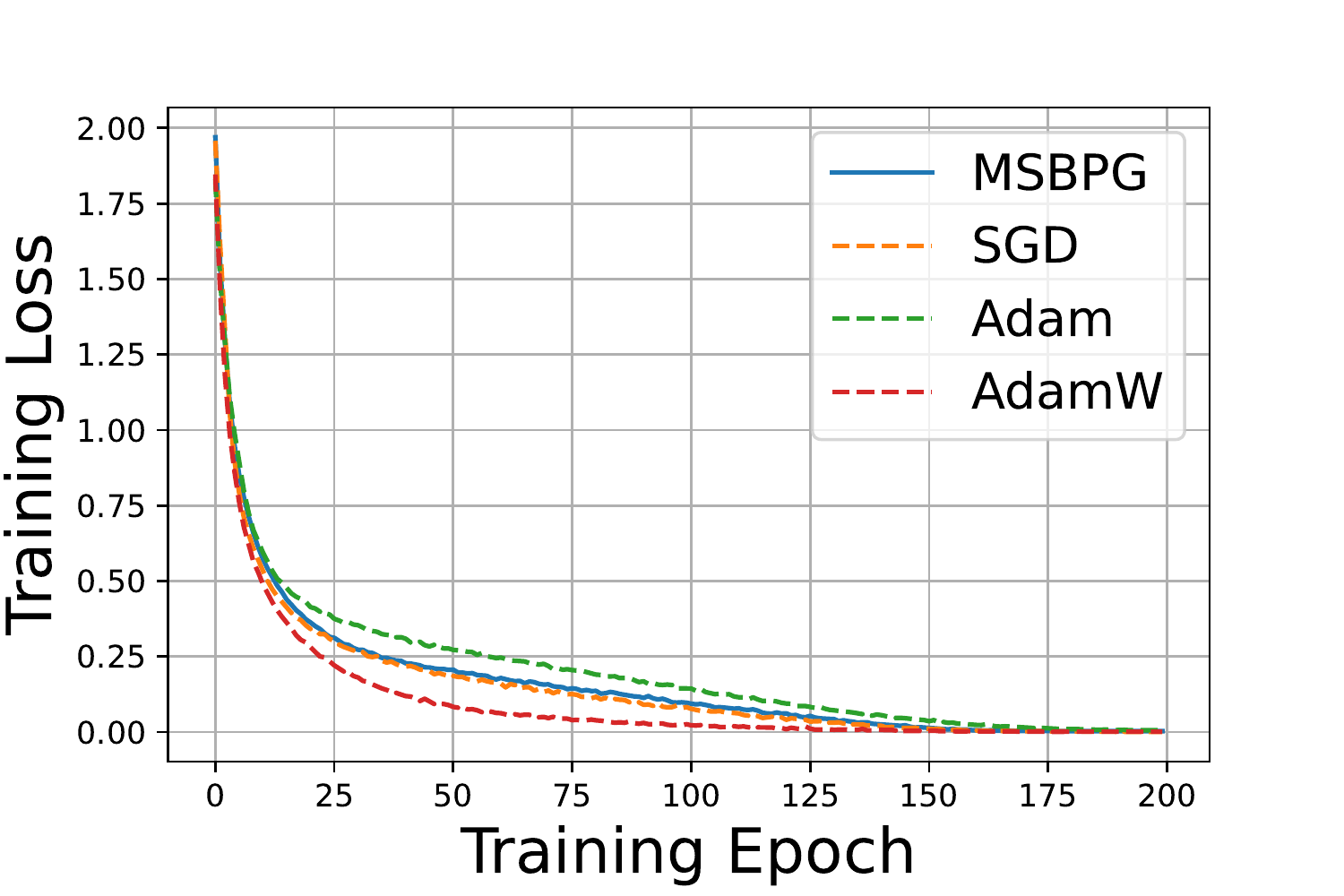}&
                                   \includegraphics[width=0.25\linewidth]{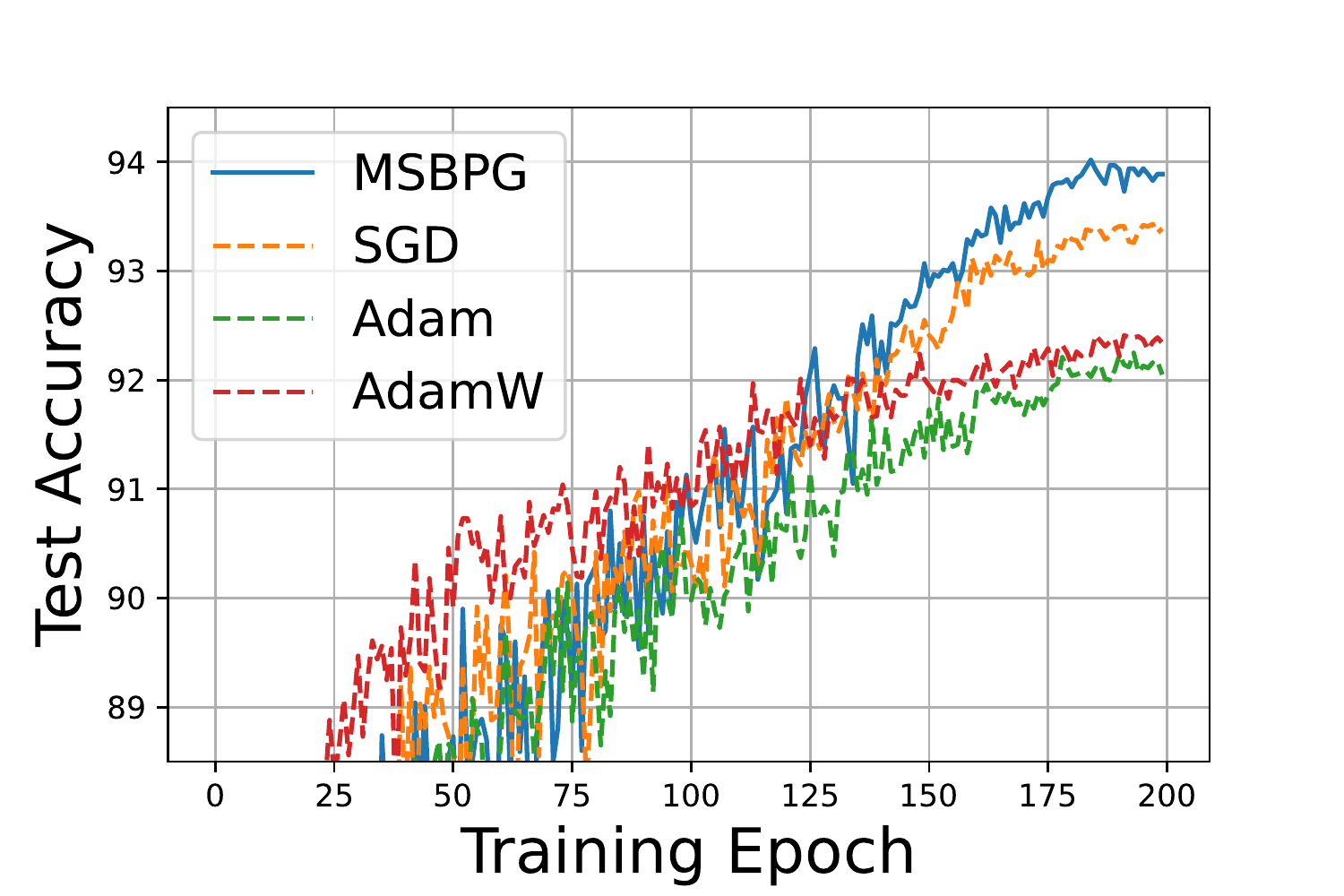}&
                                   \includegraphics[width=0.25\linewidth]{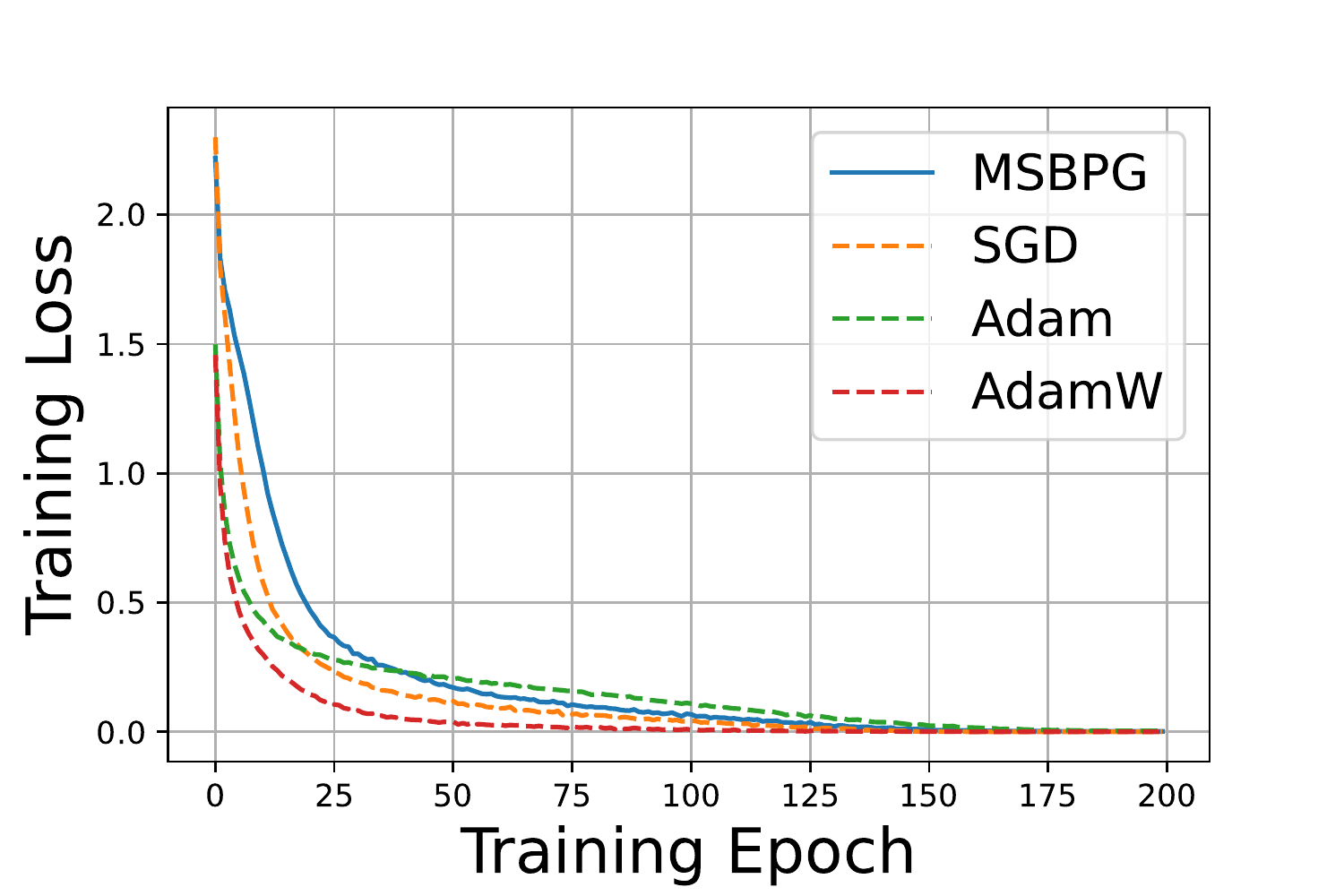} &
                                   \includegraphics[width=0.25\linewidth]{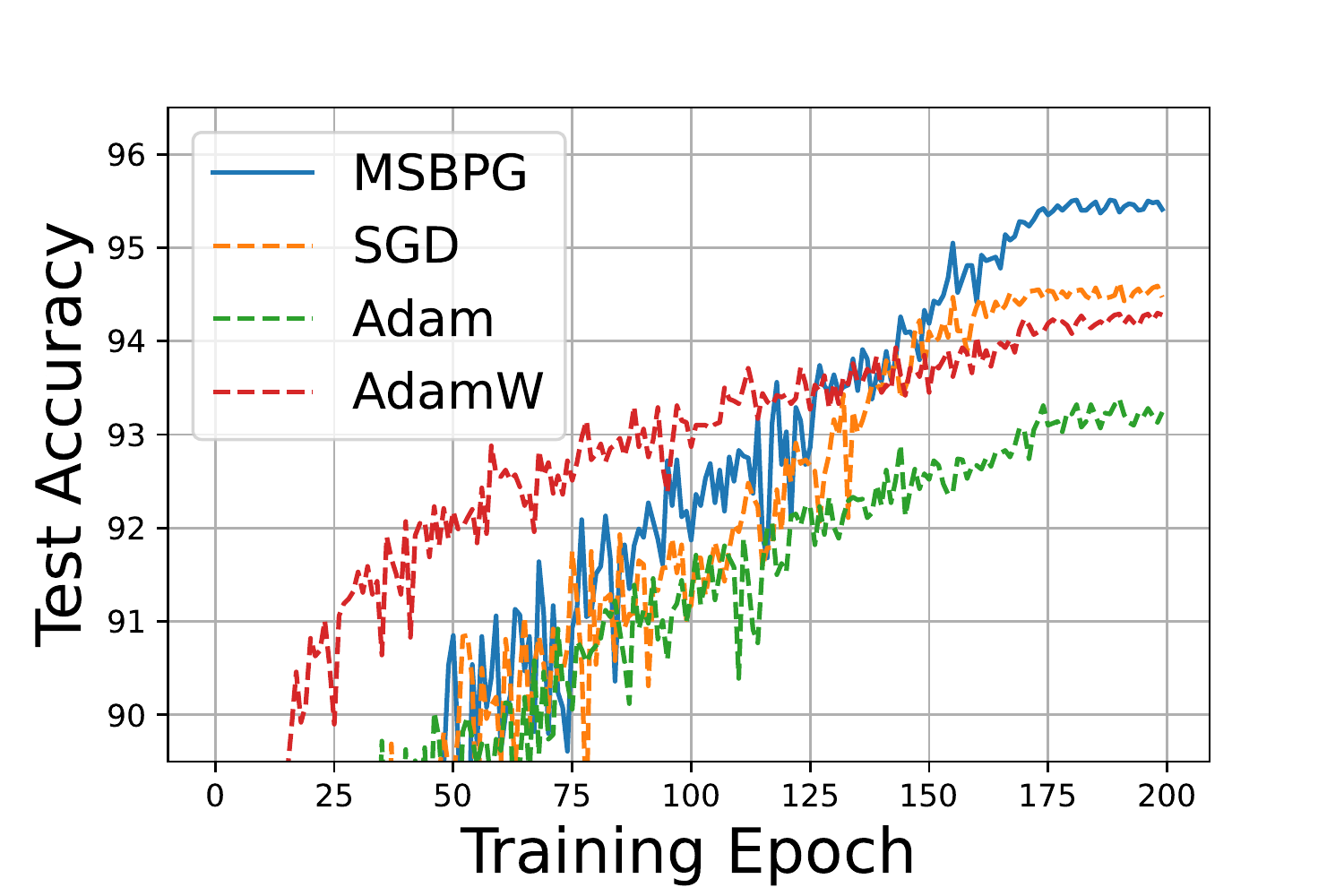}\\
                                          \multicolumn{2}{c}{\footnotesize{(a) VGG16 on CIFAR10}} &  \multicolumn{2}{c}{\footnotesize{(b) ResNet34 on CIFAR10}}\\
                     \end{tabular}}
              \end{center}
       \vspace{-0.8em}
              \caption{Training loss and test accuracy (\%) of CNNs on CIFAR10 dataset with learning rate using the cosine annealing schedule.} \label{fig:cnn2}
       \end{figure*}

 \begin{figure*}[th]
              \begin{center}
                     \setlength{\tabcolsep}{0.0pt}  
                     \scalebox{1}{\begin{tabular}{cccc}
                                   \includegraphics[width=0.25\linewidth]{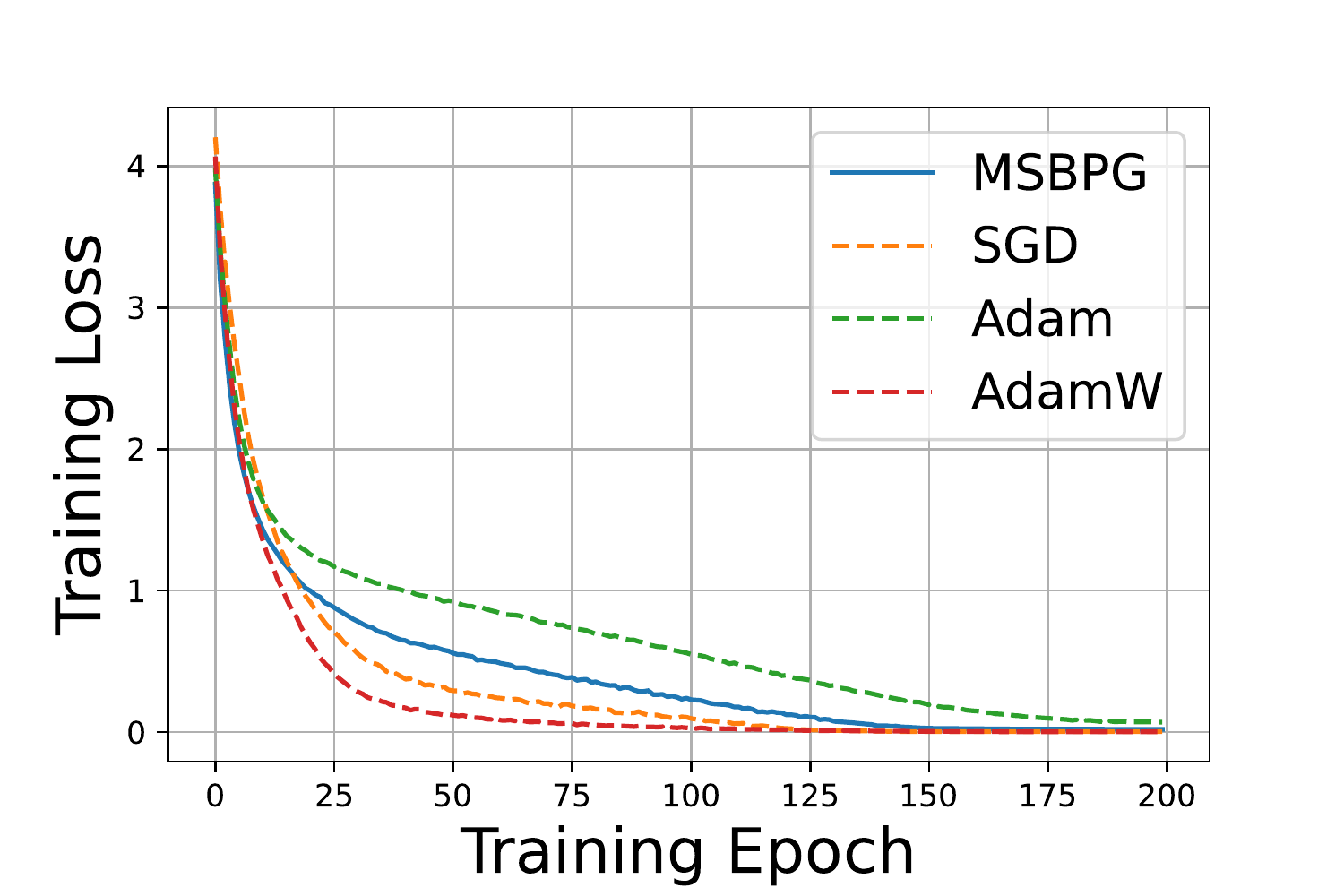}&
                                   \includegraphics[width=0.25\linewidth]{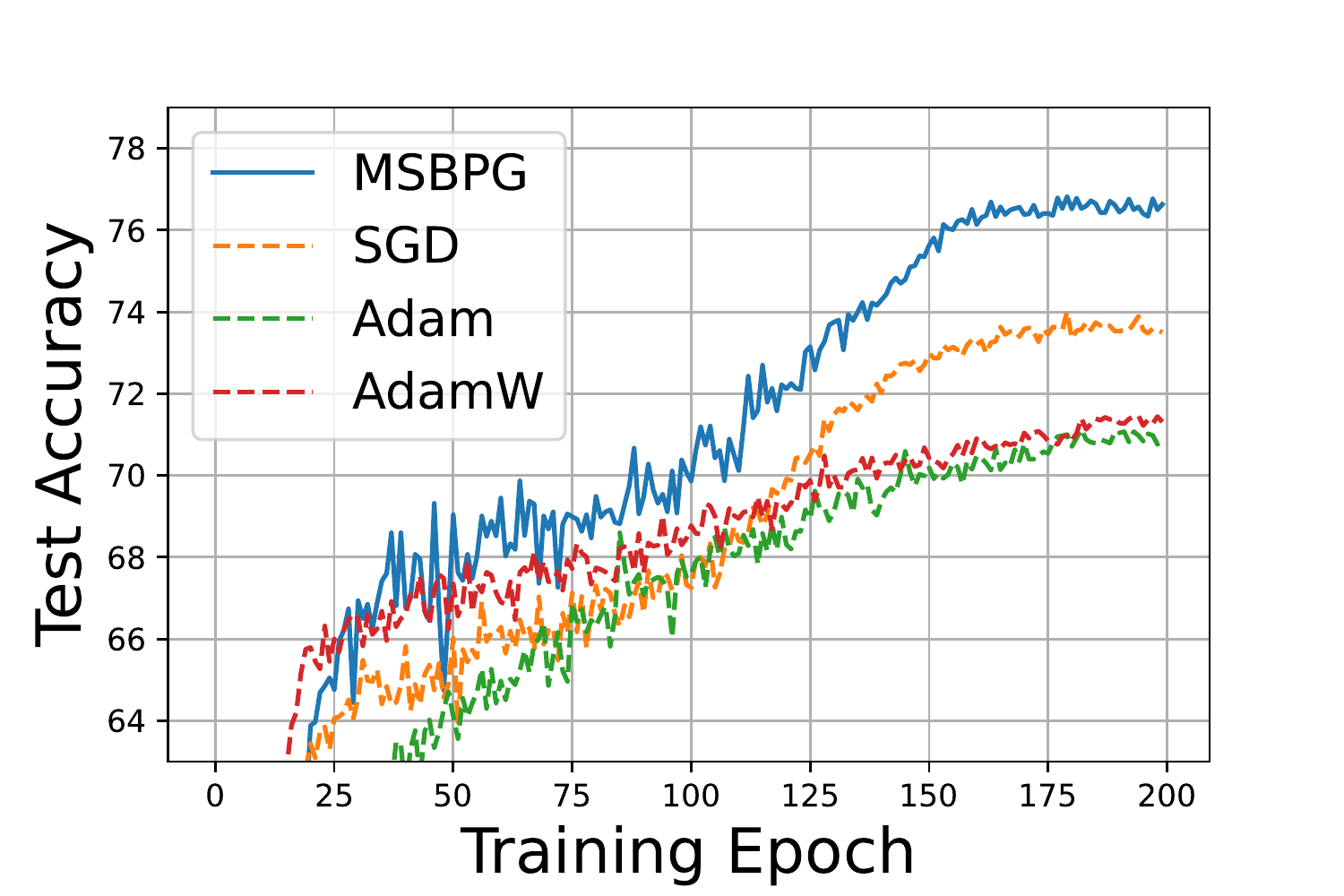}&
                                   \includegraphics[width=0.25\linewidth]{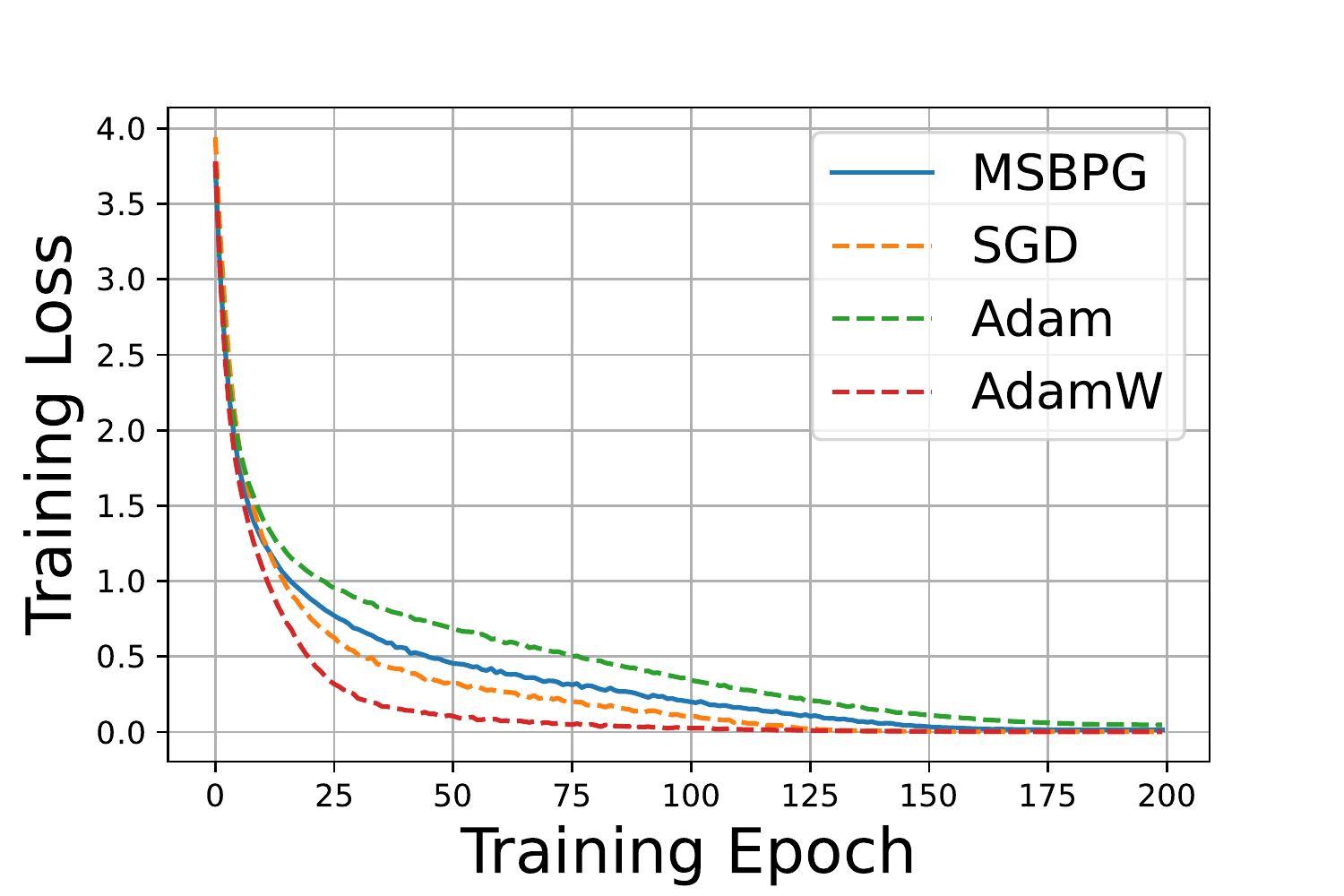} &
                                   \includegraphics[width=0.25\linewidth]{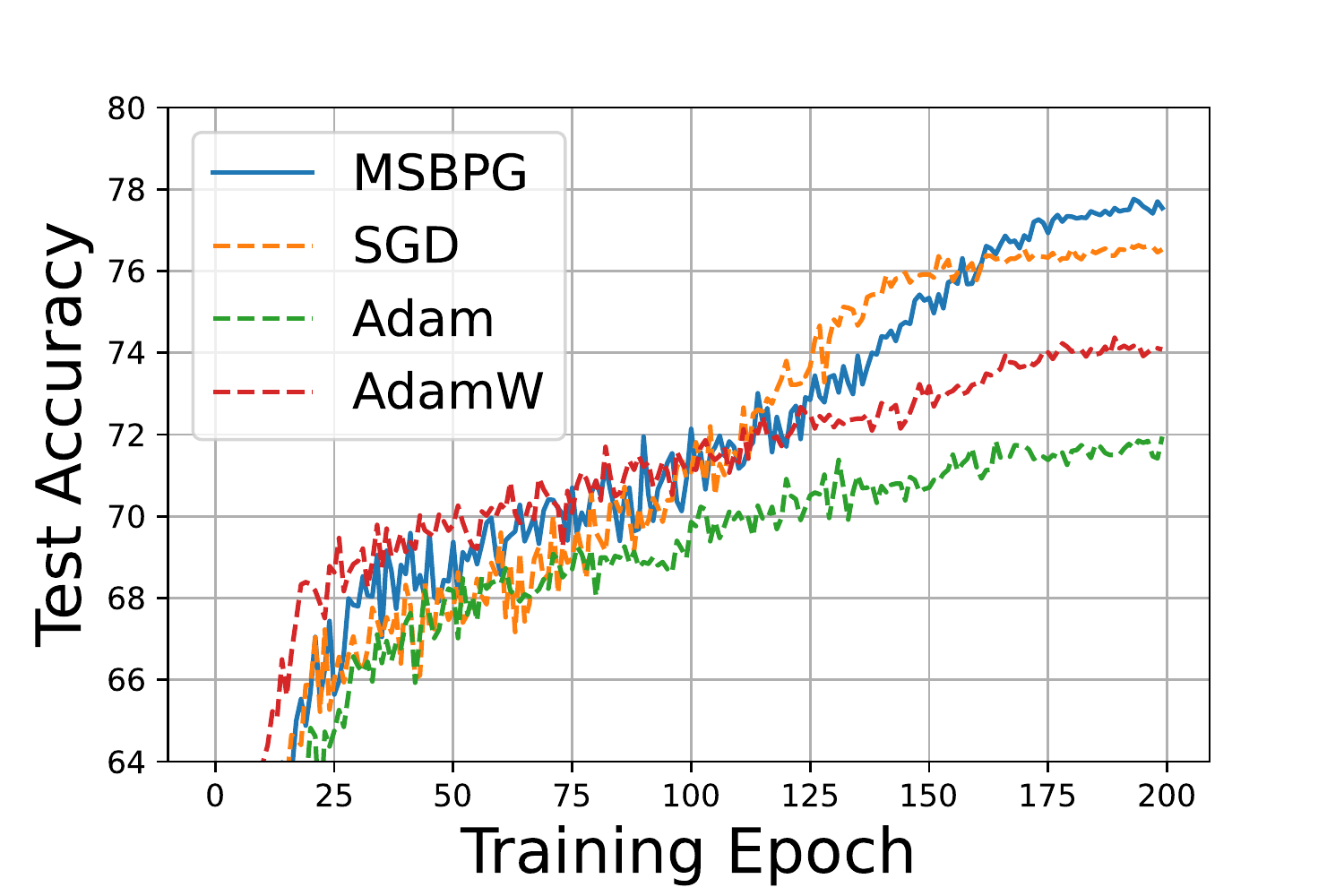}\\
                                          \multicolumn{2}{c}{\footnotesize{(a) ResNet34 on CIFAR100}} &  \multicolumn{2}{c}{\footnotesize{(b) DenseNet121 on CIFAR100}}\\
                     \end{tabular}}
              \end{center}
       \vspace{-0.8em}
              \caption{Training loss and test accuracy (\%) of CNNs on CIFAR100 dataset with learning rate using the cosine annealing schedule.} \label{fig:cnn4}
       \end{figure*}

\paragraph{LSTMs on language modeling}
To further evaluate the performance of MSBPG, we conducted experiments on LSTMs using the Penn Treebank dataset, reporting both training and test perplexity (lower is better). Adam is generally favored over SGD for language modeling tasks due to its better generalization capacity \citep{fu2016using, siami2019performance}, and thus it is the default optimization algorithm for training LSTMs. We followed the standard experimental setup for training LSTMs \citep{zhuang2020adabelief, chen2021closing}, where the learning rate is reduced to 10\%  of its original value twice (at the 75th and 150th epochs). Additionally, we experimented with the cosine annealing learning rate schedule \citep{loshchilov2016sgdr}, which is commonly used in practice. For the training hyperparameters, we used the default settings for SGD, Adam, and AdamW when training 1-, 2-, and 3-layer LSTMs \citep{zhuang2020adabelief, chen2021closing}. For MSBPG, we set the learning rate to 25, 80, and 80 for 1-, 2-, and 3-layer LSTMs, respectively, with a momentum parameter $\beta=0.9$, weight decay coefficient $\lambda_2=2\times 10^{-6}$. For the layerwise kernel function $\phi_i({\mW}_i)=\frac{1}{2}\|{\mW}_i\|^2+\frac{\delta}{r}\|{\mW}_i\|^{r}$, we set $r=4$ and $\delta=1\times 10^{-6}$. From Figure \ref{fig:lstm1} and Figure \ref{fig:lstm2}, we observe that MSBPG converges well on the training dataset for 1-, 2-, and 3-layer LSTMs across both training strategies. In contrast, SGD with the cosine annealing learning rate schedule fails to fully converge on the training dataset, as shown in Figure \ref{fig:lstm2}. Additionally, MSBPG consistently achieves lower test perplexity in all experiments, outperforming other methods by at least 1 unit. This excellent generalization capacity can be attributed to the Bregman proximity model employed by MSBPG. {As an additional evaluation, we further assess the performance of MSBPG on a recently popular transformer model, Transformer-XL~\citep{dai2019transformer}, which is designed for long-sequence tasks. The results are provided in Appendix \ref{appendix:additional exp}.}

 \begin{figure*}[th]
              \begin{center}
                     \setlength{\tabcolsep}{0.0pt}  
                     \scalebox{1}{\begin{tabular}{ccc}
            \includegraphics[width=0.33\linewidth]{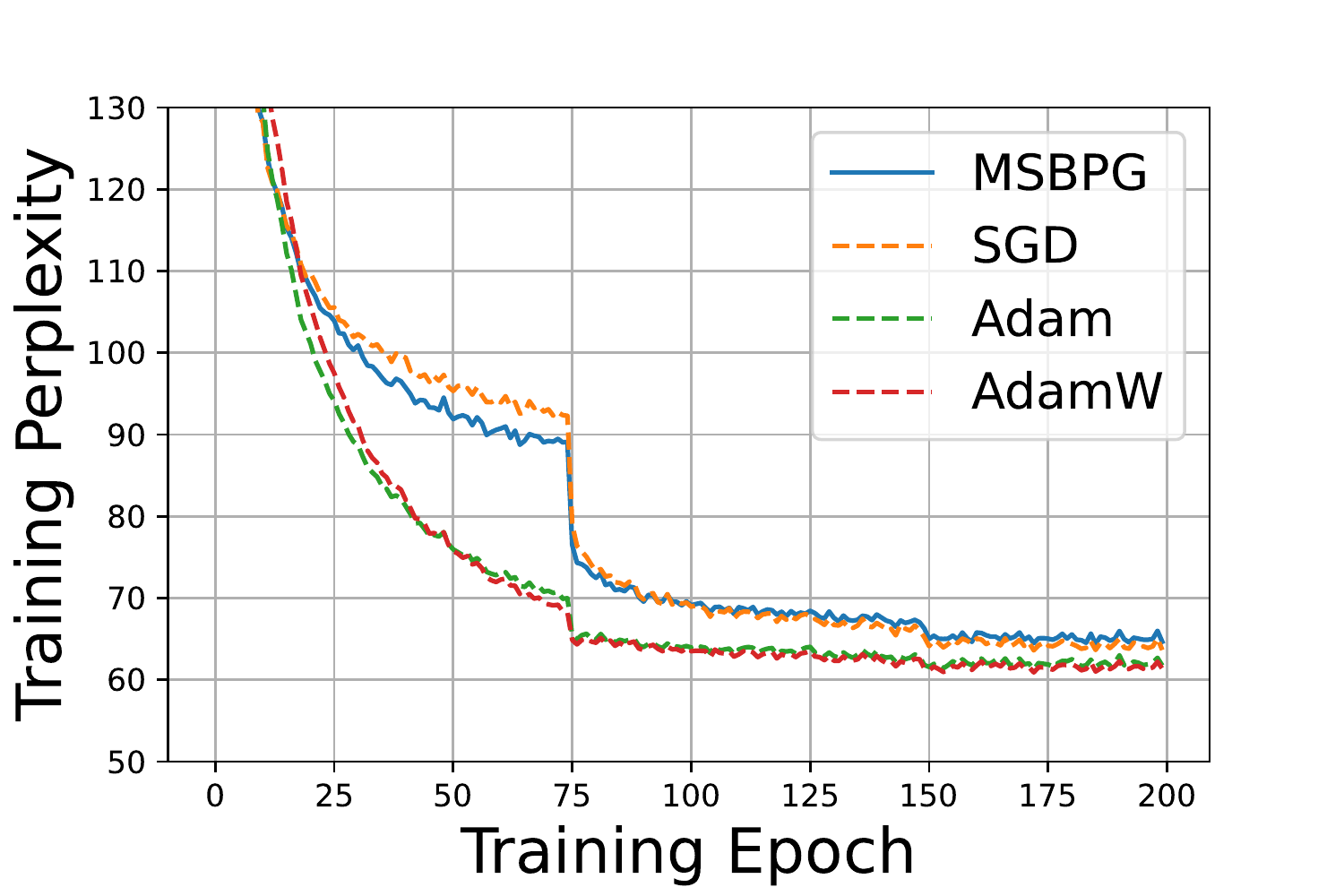}&
            \includegraphics[width=0.33\linewidth]{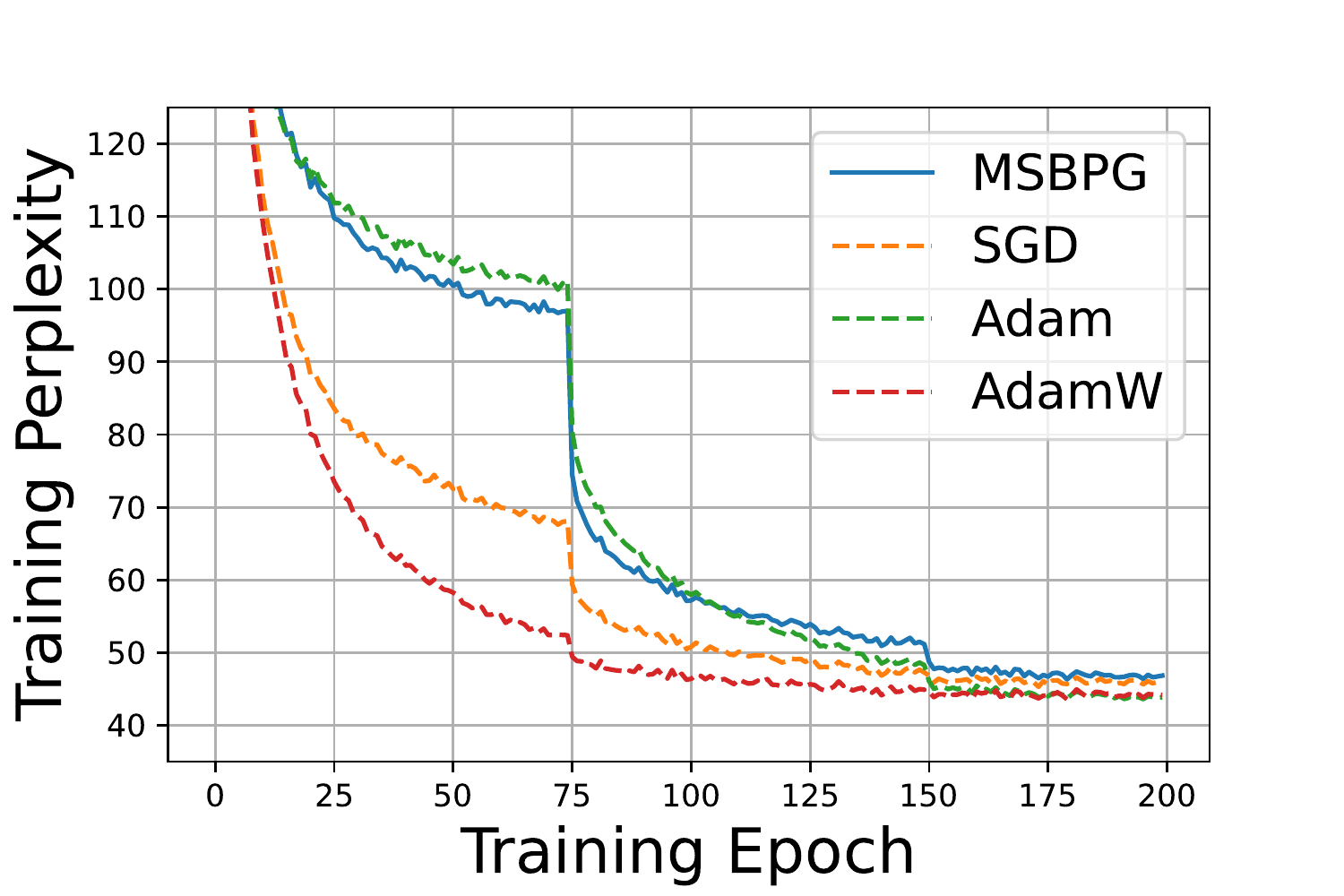}&
            \includegraphics[width=0.33\linewidth]{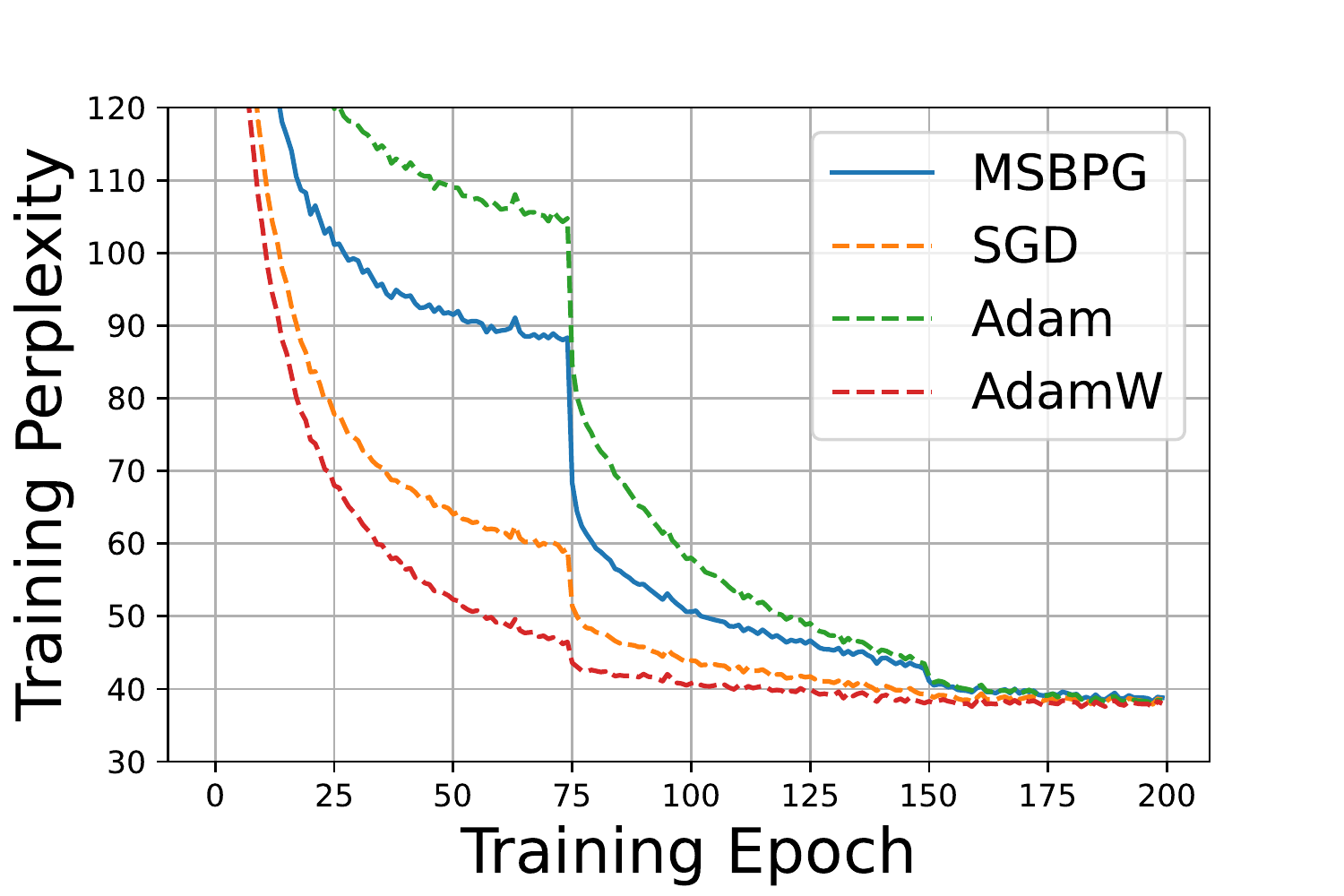}\\

            \includegraphics[width=0.33\linewidth]{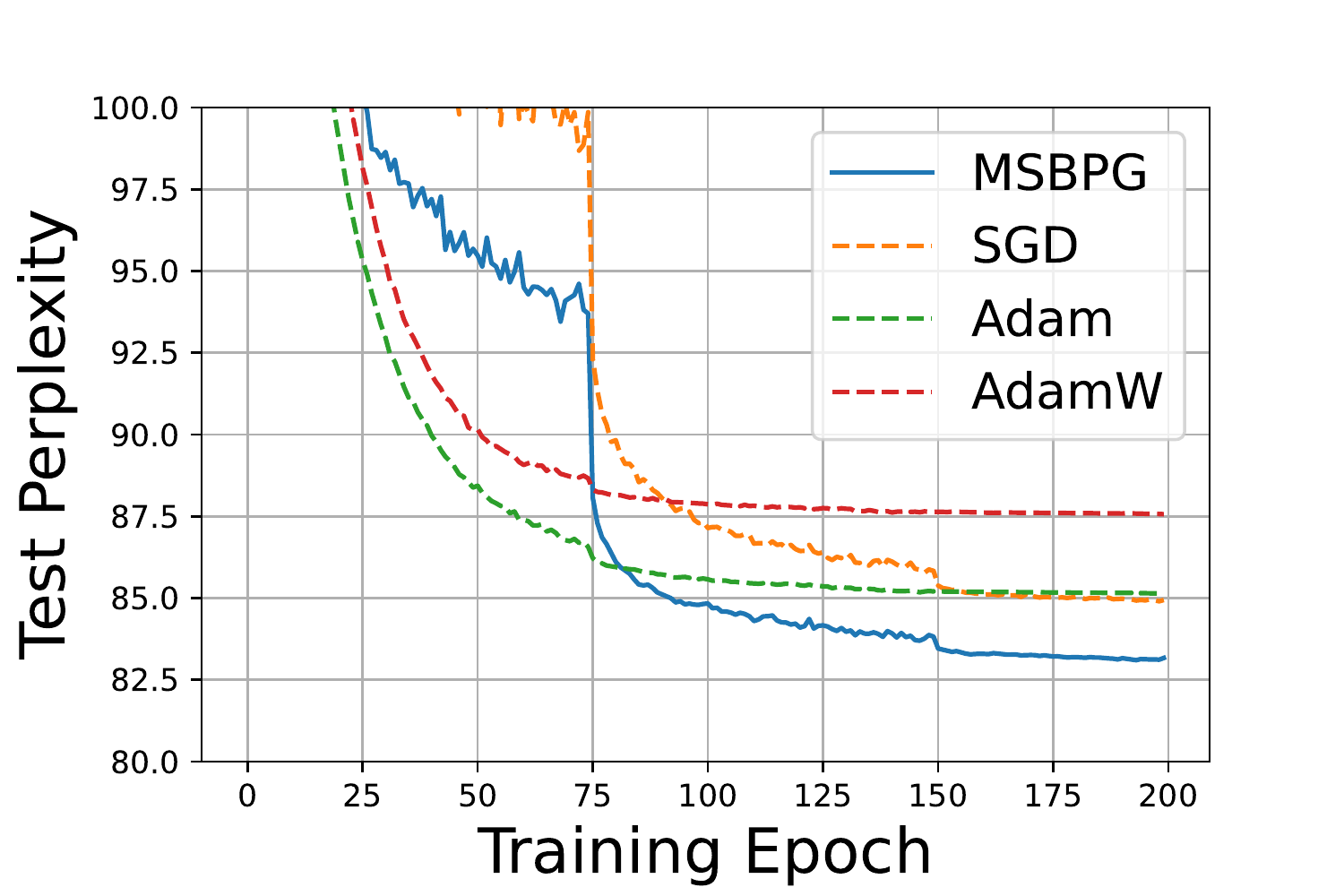}&
            \includegraphics[width=0.33\linewidth]{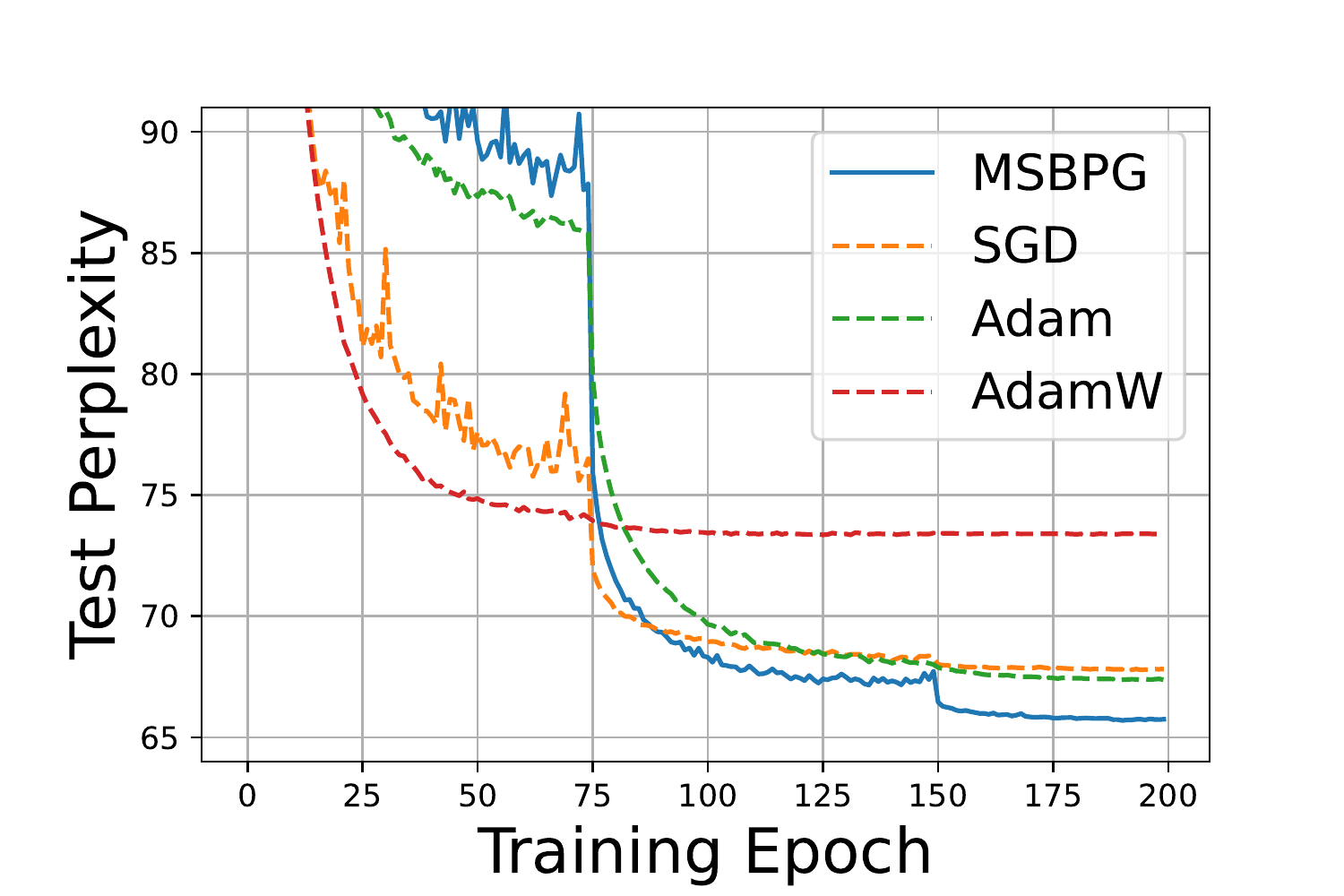}&
            \includegraphics[width=0.33\linewidth]{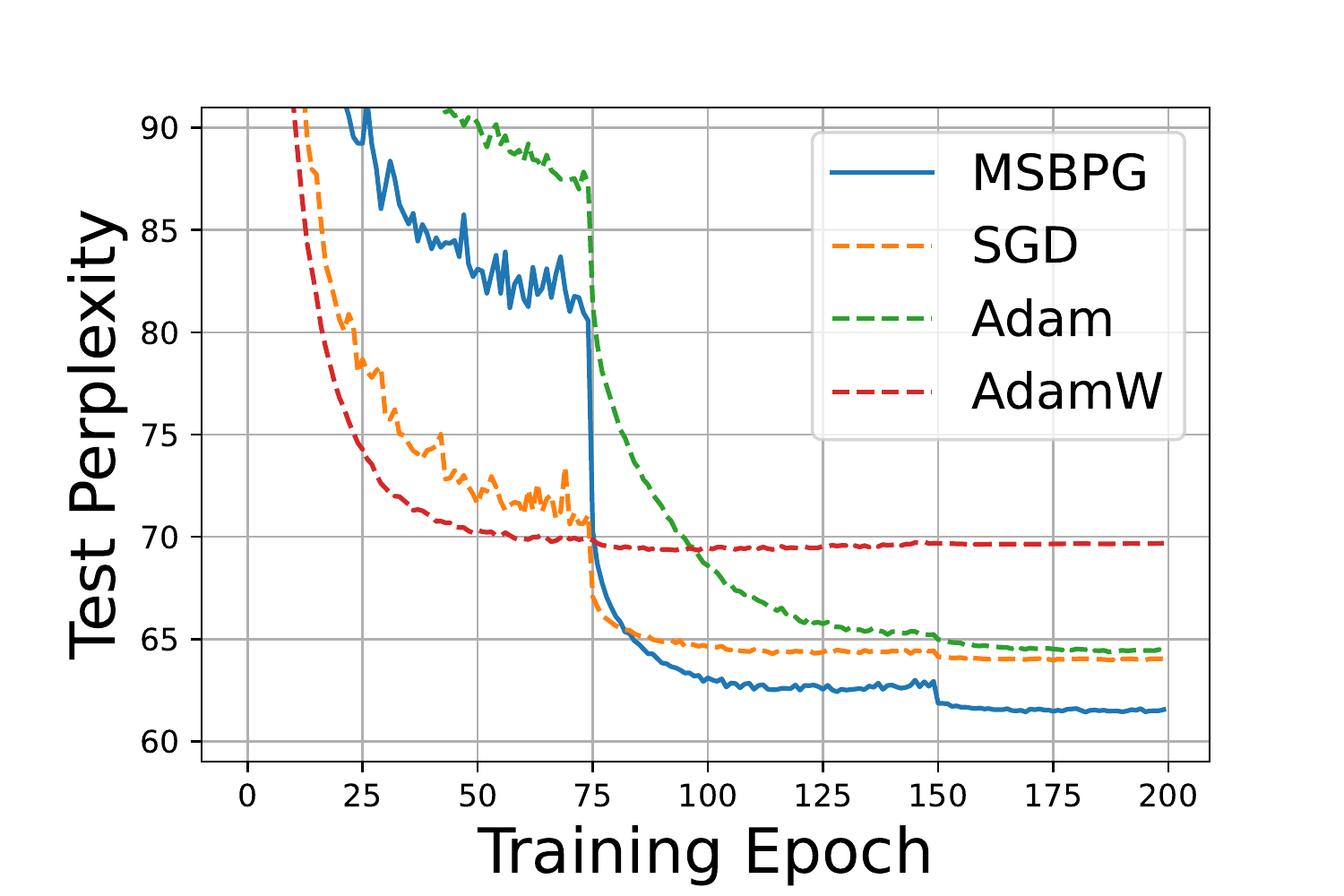}\\
            \footnotesize{(a) 1-layer LSTM} &  \footnotesize{(b) 2-layer LSTM}
            &  \footnotesize{(c) 3-layer LSTM}
            \\
                     \end{tabular}}
              \end{center}
       \vspace{-0.8em}
              \caption{Training and test perplexity (lower is better) of LSTMs on Penn Treebank dataset with learning rate reduced to 0.1 times of the original value at the 75th epoch and 150th epoch.} \label{fig:lstm1}
\end{figure*}

 \begin{figure*}[th]
              \begin{center}
                     \setlength{\tabcolsep}{0.0pt}  
                     \scalebox{1}{\begin{tabular}{ccc}
            \includegraphics[width=0.33\linewidth]{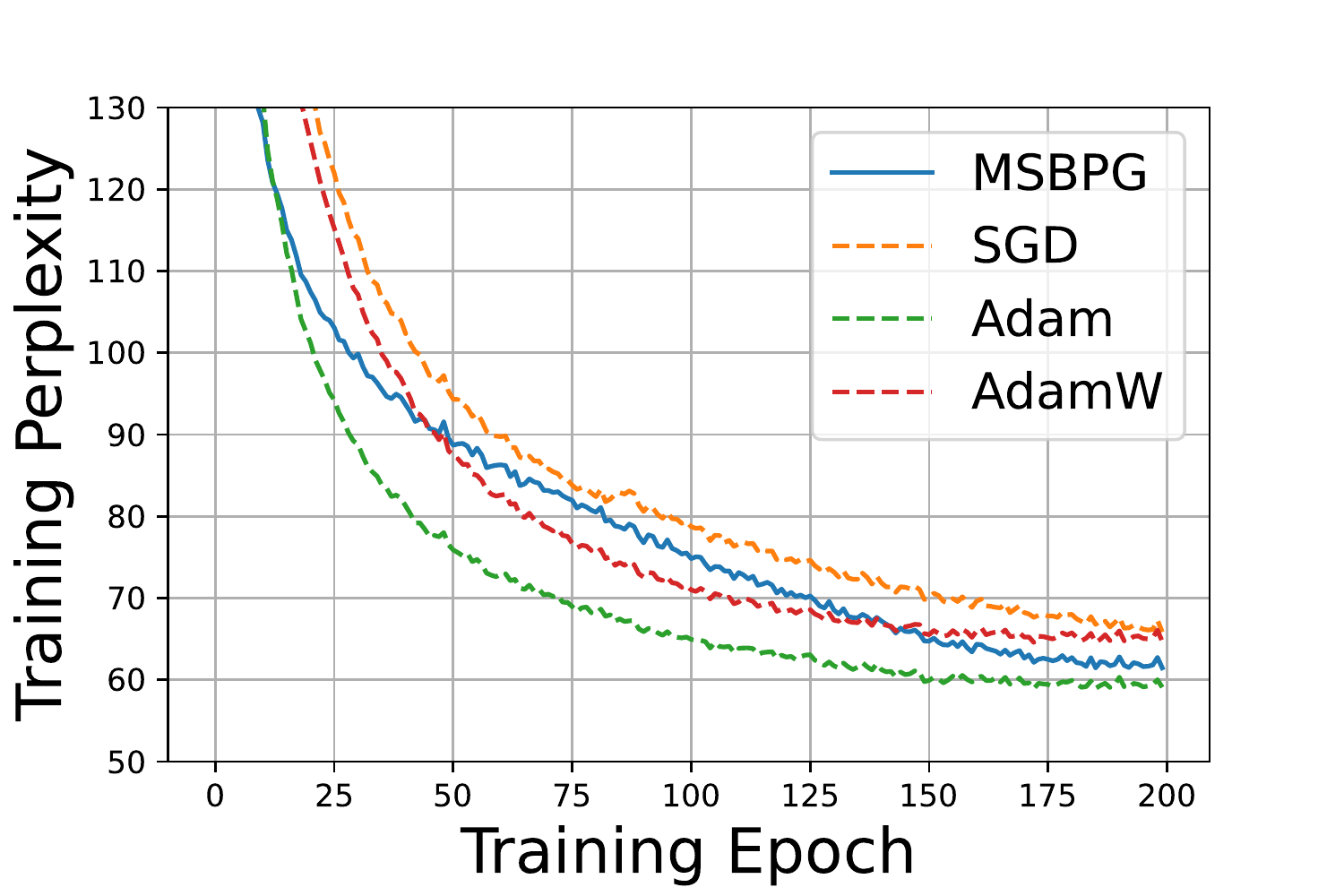}&
            \includegraphics[width=0.33\linewidth]{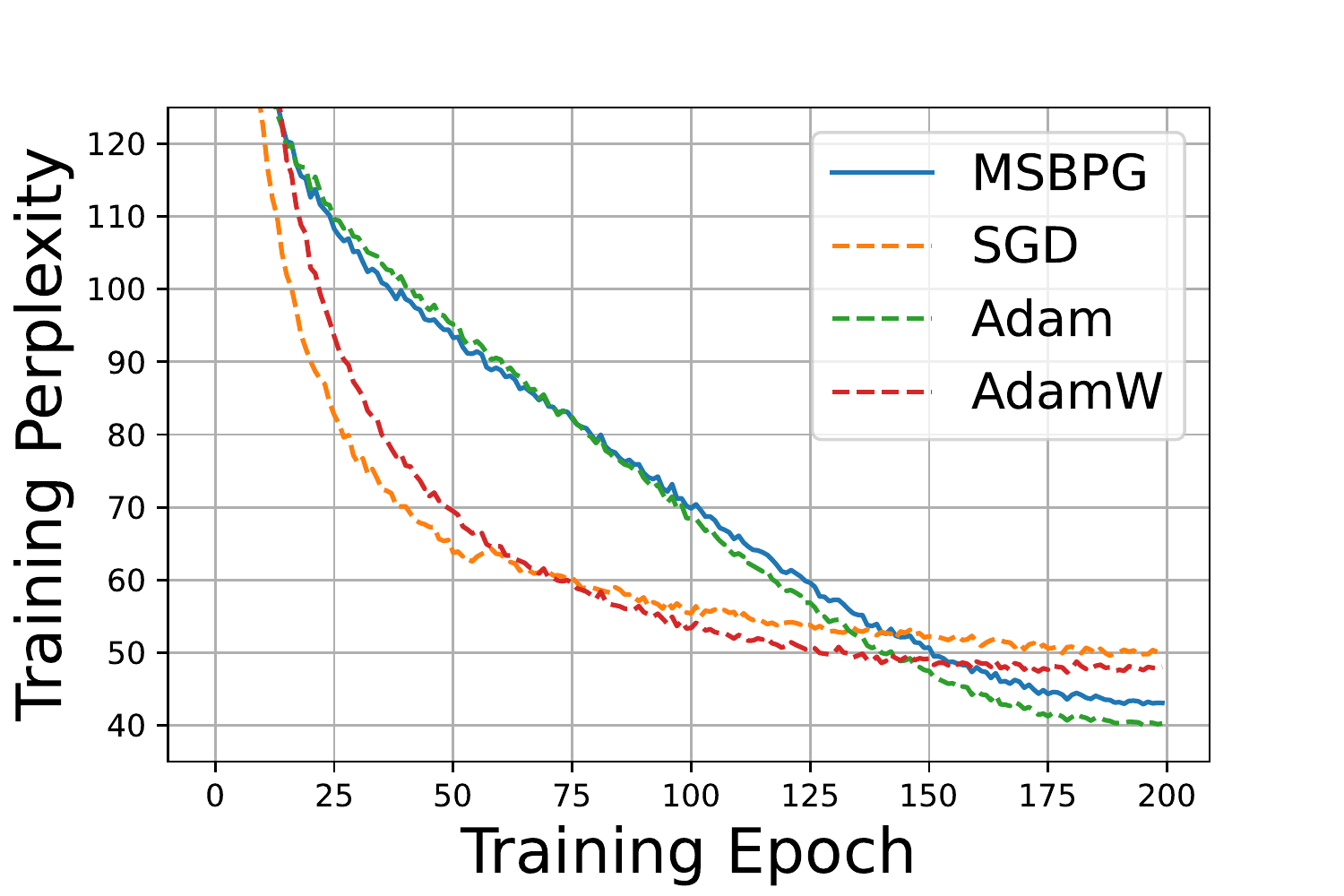}&
            \includegraphics[width=0.33\linewidth]{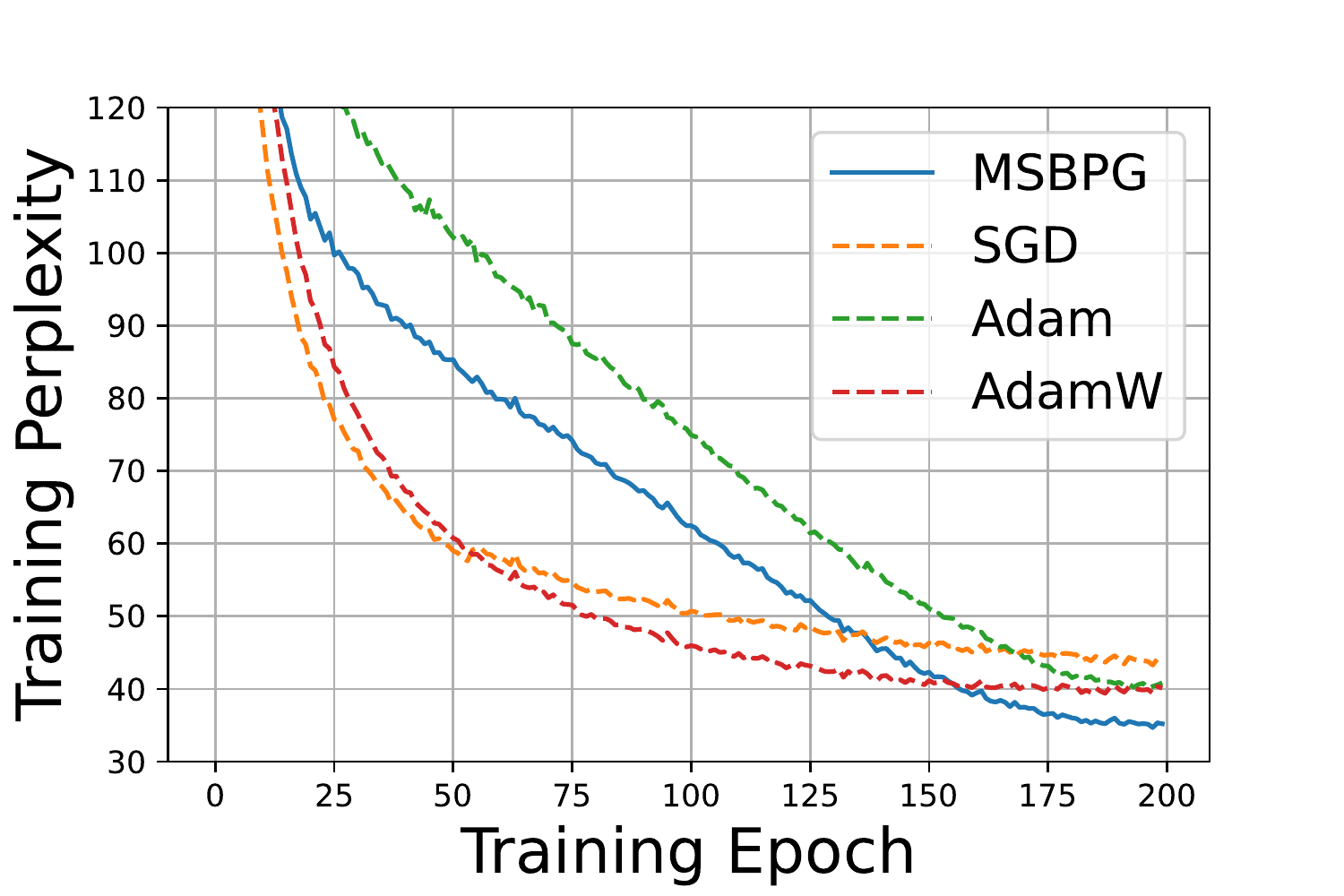}\\

            \includegraphics[width=0.33\linewidth]{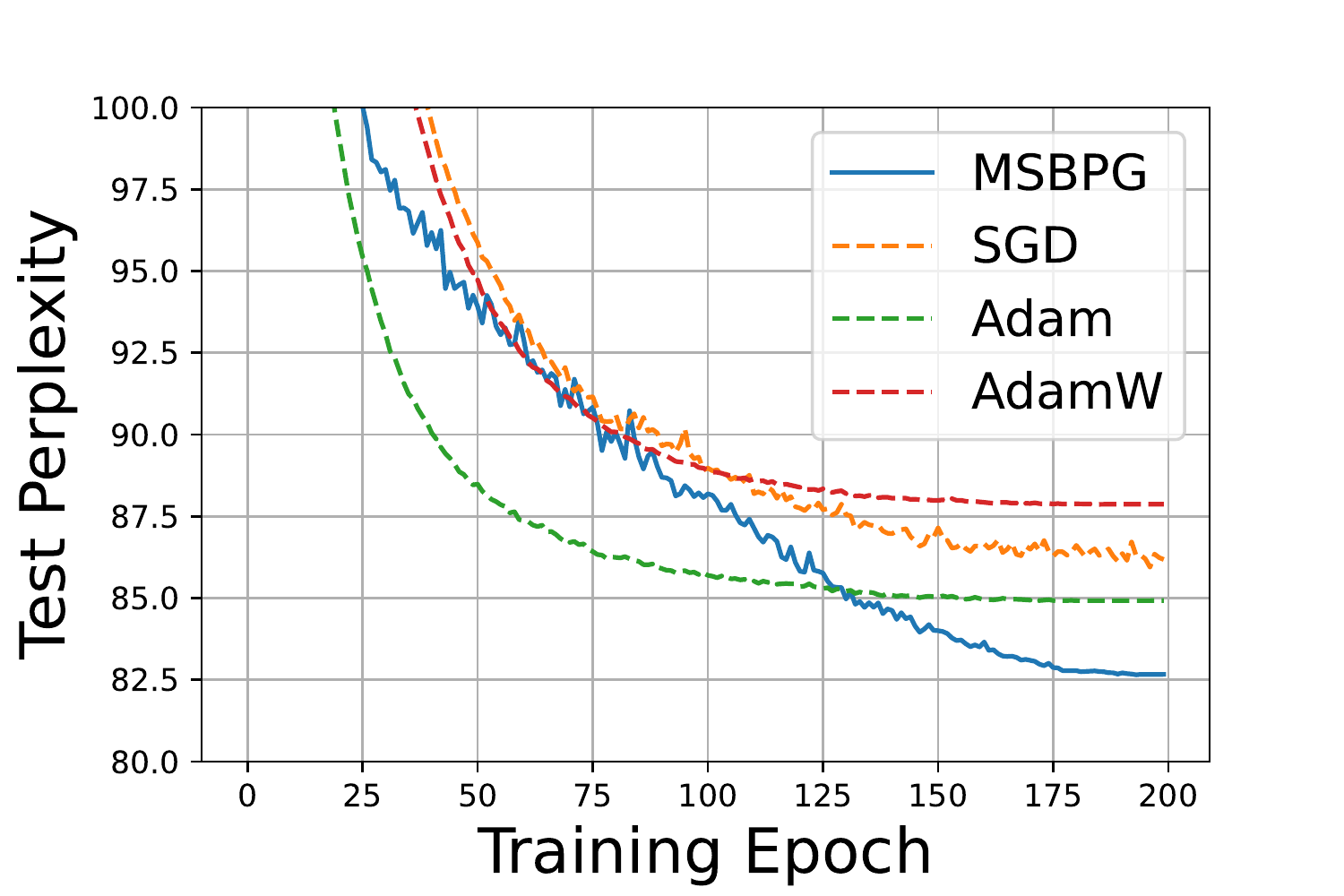}&
            \includegraphics[width=0.33\linewidth]{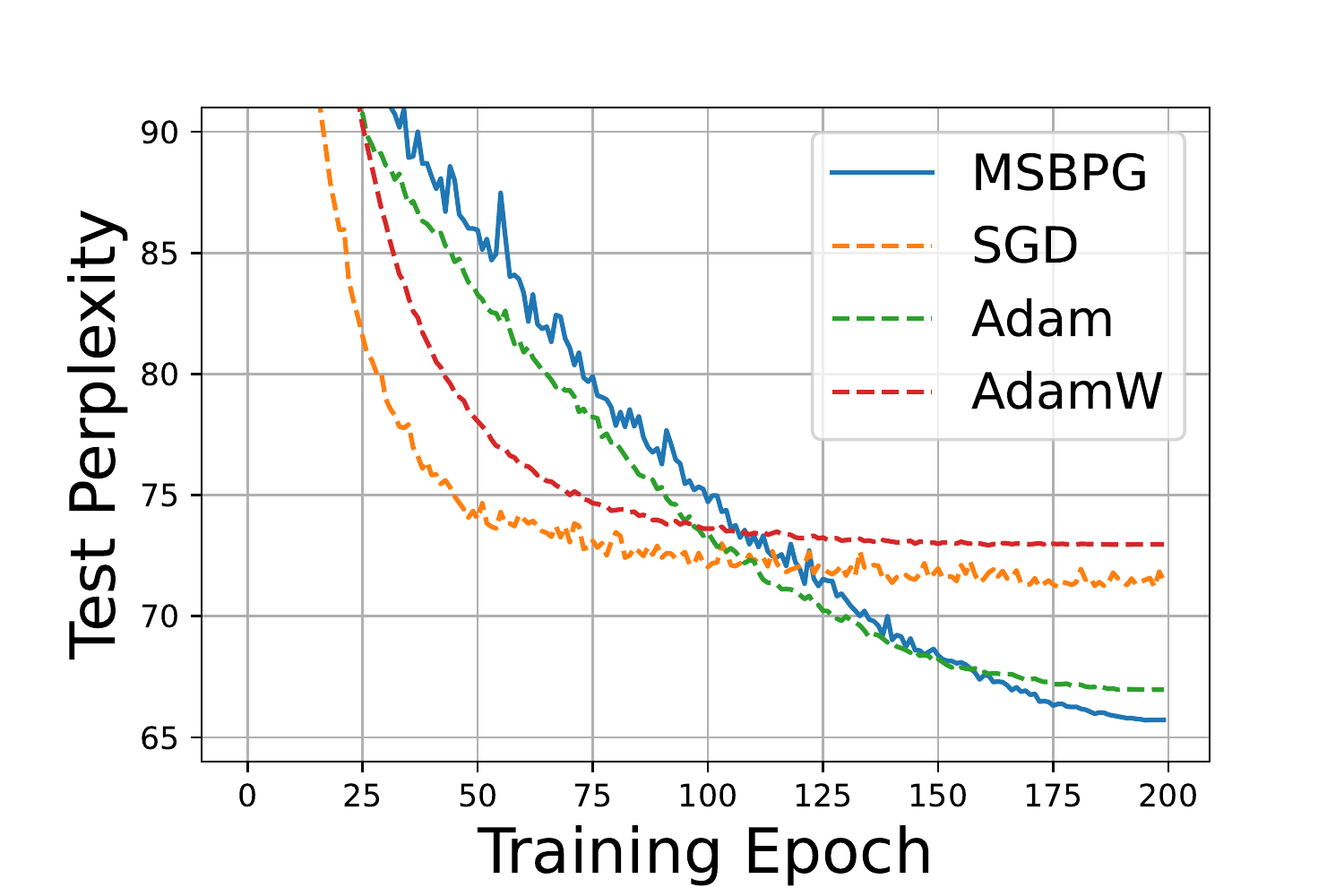}&
            \includegraphics[width=0.33\linewidth]{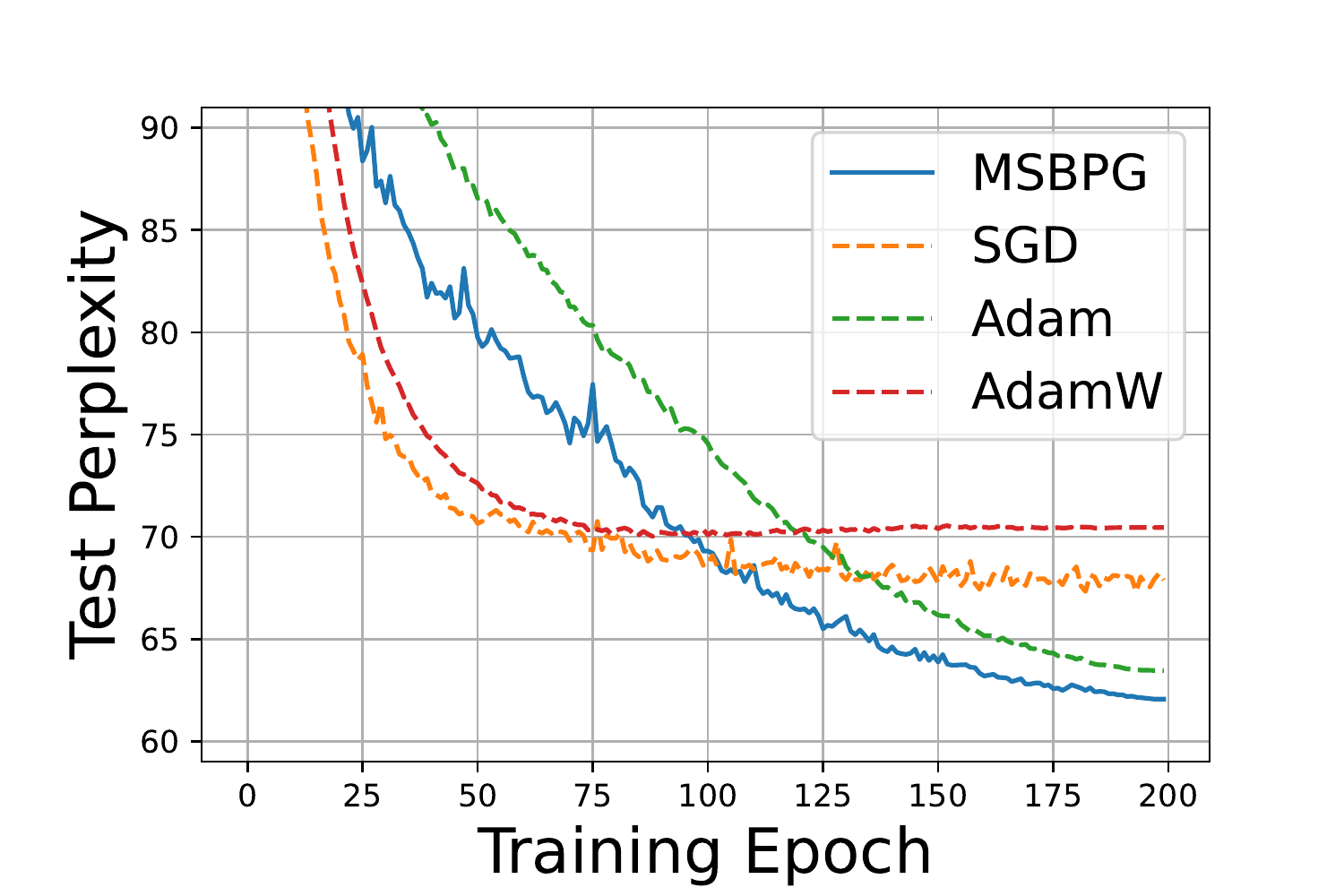}\\
            \footnotesize{(a) 1-layer LSTM} &  \footnotesize{(b) 2-layer LSTM}
            &  \footnotesize{(c) 3-layer LSTM}
            \\
                     \end{tabular}}
              \end{center}
       \vspace{-0.8em}
              \caption{Training and test perplexity (lower is better) of LSTMs on Penn Treebank dataset with learning rate using the cosine annealing schedule.} \label{fig:lstm2}
\end{figure*}

\paragraph{Robustness to initial point scale and stepsize}
As demonstrated in Section \ref{DNN implementation}, MSBPG can mitigate the issue of gradient explosion, which typically occurs when using large stepsizes or large initial point scales. To verify MSBPG's robustness in training neural networks, we conducted experiments using VGG16 on the CIFAR-10 dataset. Specifically, we compared the performance of MSBPG and SGD under different initial point scales and stepsizes, as both algorithms share the same default learning rate of ($1\times 10^{-1}$). Since adaptive gradient algorithms, such as Adam, use a different default learning rate scale ($1\times 10^{-3}$), they were not included in this comparison. For various initial point scales and stepsizes, we ran each optimization algorithm for 50 iterations and reported the best test accuracy. As shown in Figure \ref{fig:robust}, MSBPG is more robust than SGD to large initial points and stepsizes. Training deep neural networks, which have millions or billions of parameters, is highly sensitive to both the initial point scale and stepsize. From Figure \ref{fig:robust}, we can see that SGD fails to converge when the initial point scale is increased to 4.6 or when the stepsize is increased from 0.1 to 0.6. In contrast, MSBPG converges with an initial point scale as large as 20 and a stepsize as large as 5. This robustness of MSBPG  can ease the tuning of hyperparameters for training neural networks, and can also make the training process more robust to noises and errors.

\begin{figure*}[th]
              \begin{center}
                     \setlength{\tabcolsep}{0.0pt}  
                     \scalebox{1}{\begin{tabular}{cc}
            \includegraphics[width=0.5\linewidth]{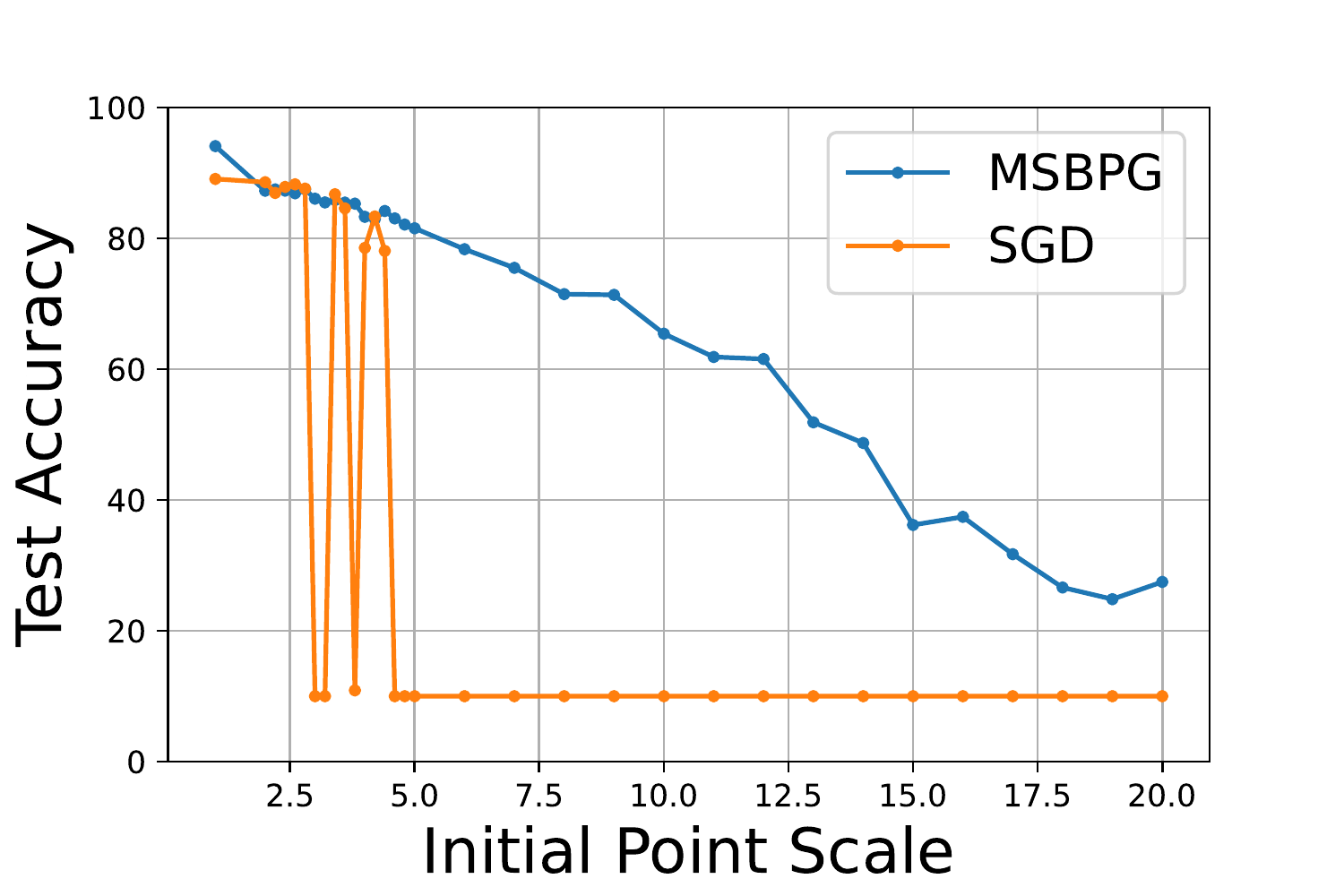}&
            \includegraphics[width=0.5\linewidth]{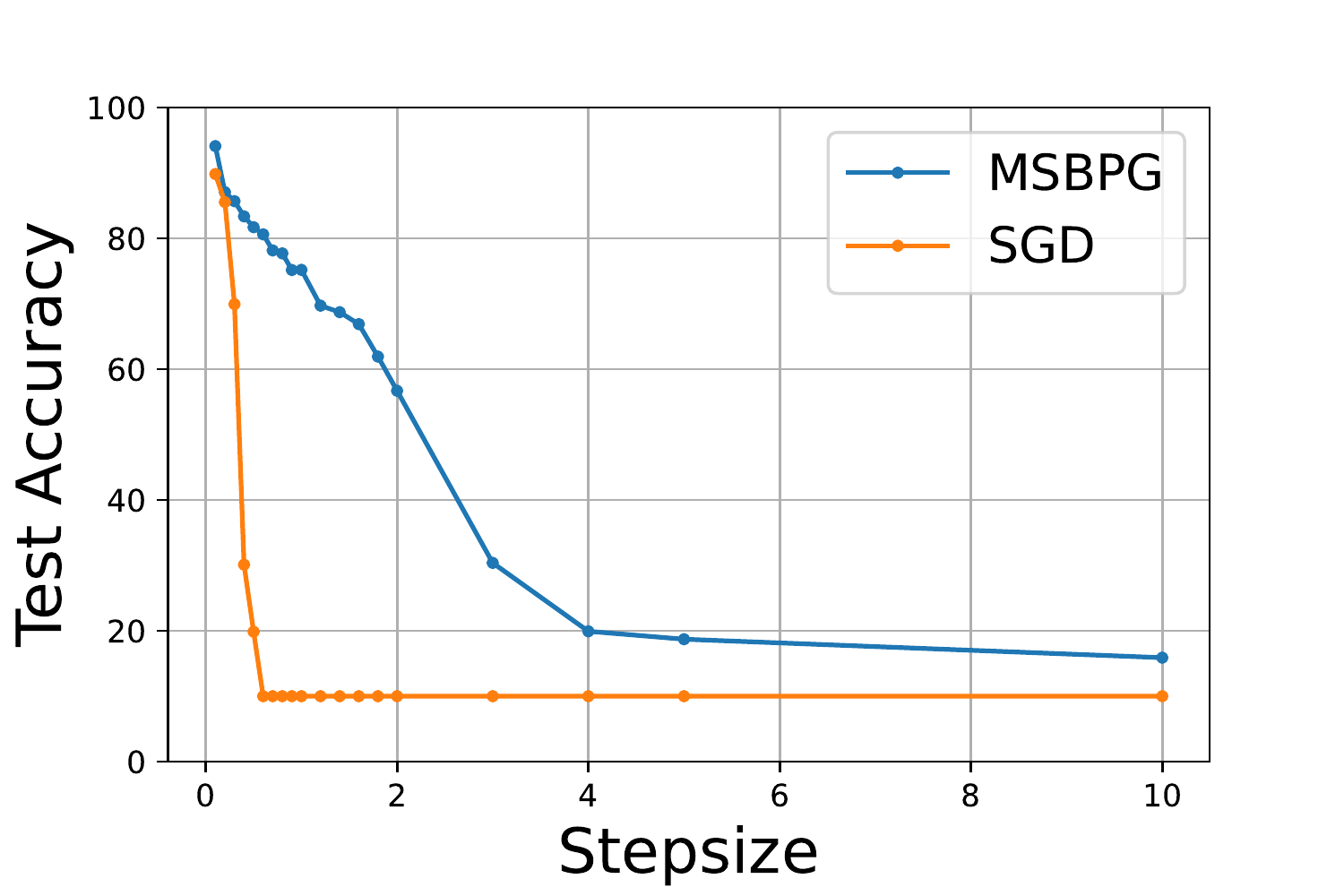}\\
            \footnotesize{(a) Initial point robustness}
            &  \footnotesize{(b) Stepsize robustness}
            \\
                     \end{tabular}}
              \end{center}
       \vspace{-0.8em}
              \caption{Test accuracy (\%) of VGG16 on CIFAR10 dataset with different initial point scale and stepsize choice.} \label{fig:robust}
\end{figure*}

\section{Conclusion}\label{section-conclusion}
In this paper, we introduce a family of nonconvex stochastic Bregman proximal gradient (SBPG) methods to solve optimization problems without the Lipschitz smoothness assumption. By leveraging Bregman proximity measures, SBPG offers a more flexible and robust framework than classical stochastic gradient methods. We establish convergence results for the vanilla SBPG method in the nonconvex setting and propose a momentum-based variant, MSBPG, which improves convergence property by relaxing the mini-batch size requirement. Both methods achieve optimal sample complexity $\mathcal{O}(\epsilon^{-4})$, making them well-suited for large-scale problems. MSBPG is applied to training deep neural networks, where it mitigates gradient explosion and enhances generalization performance. Numerical experiments on sparse quadratic inverse problems and deep neural networks demonstrate MSBPG's robustness and superior performance compared to commonly used optimizers such as SGD, Adam, and AdamW. In conclusion, MSBPG provides an effective and efficient optimization approach for large-scale nonconvex problems, combining theoretical robustness with practical advantages. Future work can explore further refinements and applications in more complex machine learning tasks.

\appendix

\section{Proofs in Preliminaries}\label{appendix-preliminary}
\noindent{\bf Proof of Lemma \ref{well-def-Breg-prox}}
First, we prove the uniqueness of the solution. Problem \eqref{Breg-prox-def} is equivalent to the following problem:
\[
\arg\min_{\vu\in\overline{C}}\;\Psi(\vu):=\alpha R(\vu)+\phi(\vu)-\inprod{\nabla\phi(\vx)}{\vu}.
\]
We have that
\[
\Psi(\vu)\geq \alpha R(\vu)
+\phi(\vu)-\norm{\nabla\phi(\vx)}\norm{\vu}
\geq\|\vu\|\Big(\frac{\alpha R(\vu)+\phi(\vu)}{\|\vu\|}-\|\nabla\phi(\vx)\|\Big).
\]
As $\|\vu\|\rightarrow\infty$, we have $\Psi(\vu)\geq\|\vu\|\Big(\frac{\alpha R(\vu)+\phi(\vu)}{\|\vu\|}-\|\nabla\phi(\vx)\|\Big)=\infty$, where we use the fact that $\phi$ is supercoercive and $R$ is convex. Since $\Psi$ is a proper lower-semicontinuous convex function, by the modern form of Weierstrass theorem \cite[Chapter 1]{rockafellar1997convex}, we know that the solution set of \eqref{Breg-prox-def} is a nonempty compact set. Also note that $\Psi$ is a strictly convex function, which implies the uniqueness of the solution. For any Legendre function $\phi$, from \cite[Chapter 26]{rockafellar1997convex}, we have $dom\,\partial\phi=int\,dom\phi$ with $\partial\phi(\vx)=\{\nabla\phi(\vx)\}$ for all $\vx\in int\,dom\,\phi$. The optimality condition implies that $\partial\phi({\rm Prox}^\phi_{\alpha R}(\vx))$ is nonempty, which automatically forces ${\rm Prox}^\phi_{\alpha P}(\vx)\in\text{int\;dom}\phi$. This completes the proof. $\Box$

 \medskip
 \noindent{\bf Proof of Proposition \ref{prop-Bregman-mapping-subdifferential}}
  Note that $\|\nabla\phi(\vx^+)-\nabla\phi(\vx)\|\leq L_\phi\|\vx^+-\vx\|$ and $\|\nabla F(\vx^+)-\nabla F(\vx)\|\leq L_FL_\phi\|\vx^+-\vx\|$.
By the definition of $\vx^+$, we have 
\[
\nabla F(\vx^+)-\nabla F(\vx)+\frac{\nabla\phi(\vx)-\nabla\phi(\vx^+)}{\alpha}\in\nabla F(\vx^+)+\partial R(\vx^+).
\]
Thus, we obtain 
\[
{\rm dist}\left(0,\partial\Phi(\vx^+)\right)\leq\bigg\|\nabla F(\vx^+)-\nabla F(\vx)+\frac{\nabla\phi(\vx)-\nabla\phi(\vx^+)}{\alpha}\bigg\|\leq\left(L_FL_\phi+\frac{L_\phi}{\alpha}\right)\|\vx^+-\vx\|.
\]
Note that $\|\vx^+-\vx\|=\alpha\|\cG_\alpha(\vx)\|$, which completes the proof. $\Box$

\medskip
\noindent{\bf Proof of Proposition \ref{nonexpansive}}
By the definition of ${\rm Prox}_R^\phi(\cdot)$,  $x_i\in\partial R(\vx_i^+)+\nabla\phi(\vx^+_i)$, $i=1,2$. Since $\partial R(\cdot)$ is monotone, then 
 $\inprod{\vx_1-\vx_2-(\nabla\phi(\vx^+_1)-\nabla\phi(\vx^+_2))}{\vx^+_1-\vx^+_2}\geq0$. From the $\mu$-strong convexity of $\phi$, it follows that $\inprod{\vx_1-\vx_2}{\vx^+_1-\vx^+_2}\geq\inprod{\nabla\phi(\vx^+_1)-\nabla\phi(\vx^+_2)}{\vx^+_1-\vx^+_2}\geq\mu\|\vx^+_1-\vx^+_2\|^2$. Therefore, $\|\vx^+_1-\vx^+_2\|\leq\frac{1}{\mu}\|\vx_1-\vx_2\|$. $\Box$

\section{Algorithmic stability analysis}
{Consider a dataset $\bm\Xi=(\bm\xi_1,...,\bm\xi_n)$ and an algorithm $\cA$. Let $\Phi_{\bm\Xi}(\vx)=F(\vx,\bm\Xi)+R(\vx)$, and let $\cA(\bm\Xi)$ be the output of the algorithm $\cA$ based on the dataset $\bm\Xi$. Since the underlying distribution of $\bm\xi$ is unknown, the population risk can be decomposed into two parts:
 \[
 \E_{\bm\Xi,\cA}[\Phi(\cA(\bm\Xi))-\Phi(\vx^*)]=\underbrace{\E_{\bm\Xi,\cA}\left[\Phi_{\bm\Xi}(\cA(\bm\Xi))-\Phi(\cA(\bm\Xi))\right]}_{\varepsilon_{\rm gen}}+\underbrace{\E_{\bm\Xi,\cA}\left[\Phi_{\bm\Xi}(\cA(\bm\Xi))-\Phi(\vx^*)\right]}_{\varepsilon_{\rm opt}},
 \]
 where $\vx^*$ is independent of $\bm\Xi$ and $\cA$. The first term represents the expected generalization error, and the second term represents the optimization error. Our goal is to assess the expected generalization error $\varepsilon_{\rm gen}$. For clarity, we focus on the simple case where $R(x)\equiv0$, in which case $\Phi(\vx)=F(\vx)=\E_{\bm\xi}[f(\vx,\bm\xi)]$. Consider any two data sets $\bm\Xi,\bm\Xi'$ that differ by at most one example. As shown by \citep{bousquet2002stability,shalev2010learnability,hardt2016train}, the absolute expected generalization error $|\varepsilon_{\rm gen}|$ can be bounded by $\varepsilon$ if the algorithm is $\varepsilon$-uniform stable that is defined as follows:
\begin{definition}
A randomized algorithm $\cA$ is $\varepsilon$-uniformly stable, if for any two datasets $\bm\Xi$ and $\bm\Xi'$ which differ by at most one example, it holds that
\[
\sup_{\vz}\E_{\cA}[f(\cA(\bm\Xi),\vz)-f(\cA(\bm\Xi'),\vz)]\leq\varepsilon.
\]
\end{definition} 
\begin{lemma}
\label{le:general-stability}
Let $\cA$ be $\varepsilon$-uniformly stable, then $|\varepsilon_{\rm gen}|\leq\varepsilon$.
\end{lemma}

In this section, we demonstrate that the polynomial kernel employed in our method (as discussed in Section 4) enhances algorithmic stability, particularly in high-dimensional scenarios such as training deep neural networks. Inspired by the ODE approach for Bregman gradient-type methods \citep{alvarez2004hessian, ding2024stochastic}, the ODE corresponding to SBPG is given by:
\[
\frac{d\nabla\phi(\vx(t))}{d t}=-\nabla F(\vx(t)), 
\]
which is equivalent to the following ODE:
\[
\dot\vx(t)=-[\nabla^2\phi(\vx(t))]^{-1}\nabla F(\vx(t)).
\]
The discrete version of this ODE then leads to the Hessian preconditioned gradient method:
\[
\vx^{k+1}=\vx^k-\alpha_k [\nabla^2\phi(\vx^k)]^{-1}\nabla F(\vx^k).
\]
This connection between the Bregman gradient method and the Hessian preconditioned gradient method inspires us to derive the following estimation on the expected generalization gap:
\begin{theorem}
Given a dataset $\bm\Xi$ containing $n$ samples. Let $\cA$ be the method \eqref{vanilla-SBPG} where the last iterate is the output, and let $\cE$ be the event where the iterates generated by $\cA$ is contained within a compact set $\cB$. Suppose $f$ is differentiable and the kernel function $\phi$ is twice differentiable. Under the event $\cE$, for any sufficiently small $\{\alpha_k\}$, we have
\begin{equation}
|\E_{\bm\Xi,\cA}[\Phi_{\bm\Xi}(\cA(\bm\Xi))-\Phi(\cA(\bm\Xi))]|\leq\sum_{t=1}^k\exp\left(\left(1-\frac{1}{n}\right)L_{\cB}\sum_{i=t+1}^k\alpha_i\right)\frac{3\ell_{f,\cB}^2\ell_{\cB}\alpha_t}{n}=:\varepsilon_{\rm gen}^{\phi},
\label{eq:stability-bound}
\end{equation}
where $L_{\cB}=\sup_{\vx\in\cB,\vy\in\cB,\vz}\frac{\norm{[\nabla^2\phi(\vx)]^{-1}\nabla f(\vx,\vz)-[\nabla^2\phi(\vy)]^{-1}\nabla f(\vy,\vz)}}{\norm{\vx-\vy}}$, $\ell_{f,\cB}=\sup_{\vx\in\cB,\vy\in\cB,\vz}\frac{\norm{f(\vx,\vz)-f(\vy,\vz)}}{\norm{\vx-\vy}}$, and $\ell_{\cB}=\sup_{\vx\in\cB}\norm{[\nabla^2\phi(\vx)]^{-1}}$.
\label{thm:alg-stability}
\end{theorem}
\begin{proof}
Given the dataset $\bm\Xi$, recall the update scheme of SBPG:
\[
\vx_{k+1}=\nabla\phi^*(\nabla \phi(\vx_k)-\alpha_k\nabla f(\vx_k,\vz_k)).
\]
Noting that $\nabla\phi^*(\nabla\phi(\vx))=\vx$ and $\nabla^2\phi^*(\nabla\phi(\vx))=[\nabla^2\phi(\vx)]^{-1}$, we have
\[
\begin{aligned}
\vx_{k+1}=&\nabla\phi^*(\nabla \phi(\vx_k)-\alpha_k\nabla f(\vx_k,\vz_k))\\
=&\nabla\phi^*(\nabla\phi(\vx_k))-\alpha_k\nabla^2\phi^*(\nabla\phi(\vx_k))\nabla f(\vx_k,\vz_k)+\mathcal{O}(\alpha_k^2)\\
=&\vx_k-\alpha_k[\nabla^2\phi(\vx_k)]^{-1}\nabla f(\vx_k,\vz_k)+\mathcal{O}(\alpha_k^2),
\end{aligned}
\]
where the second equality comes from Taylor's expansion. Let $\bm\Xi'$ be a dataset that differs from $\bm\Xi$ by one example. Define $\delta_{k+1}:=\E\left[\norm{\vx_{k+1}-\vx'_{k+1}}\right]$, where $\vx'_{k+1}$ corresponds to the dataset $\bm\Xi'$. From the inequality above, we have 
\[
\begin{aligned}
&\norm{\vx_{k+1}-\vx'_{k+1}}\\
=&\norm{\left(\vx_k-\alpha_k[\nabla^2\phi(\vx_k)]^{-1}\nabla f(\vx_k,\vz_k)\right)-\left(\vx'_k-\alpha_k[\nabla^2\phi(\vx'_k)]^{-1}\nabla f(\vx'_k,\vz'_k)\right)}+\mathcal{O}(\alpha_k^2)\\
\leq&\norm{\vx_k-\vx'_k}+\alpha_k\norm{[\nabla^2\phi(\vx_k)]^{-1}\nabla f(\vx_k,\vz_k)-[\nabla^2\phi(\vx'_k)]^{-1}\nabla f(\vx'_k,\vz'_k)}+\mathcal{O}(\alpha_k^2).
\end{aligned}
\]
Since $\vx_0=\vx'_0$, then we have
\begin{equation}
\begin{aligned}
\delta_{k+1}=&\left(1-\frac{1}{n}\right)\left(1+\alpha_kL_{\cB}\right)\delta_k+\frac{1}{n}\left(\delta_k+2\ell_{f,\cB}\ell_{\cB}\alpha_k\right)+C\alpha_k^2\\
\leq&\left(1+\alpha_k\left(1-\frac{1}{n}\right)L_{\cB}\right)\delta_k+\frac{3\ell_{f,\cB}\ell_{\cB}\alpha_k}{n}\\
\leq&\exp\left(\alpha_k\left(1-\frac{1}{n}\right)L_{\cB}\right)\delta_k+\frac{3\ell_{f,\cB}\ell_{\cB}\alpha_k}{n}.
\end{aligned}
\label{eq:stability-recursion}
\end{equation}
The first inequality comes from the fact that $\alpha_k$ is sufficiently small. Since $\delta_0=0$, by recursion \eqref{eq:stability-recursion}, we have 
\[
\delta_k\leq\sum_{t=1}^k\exp\left(\left(1-\frac{1}{n}\right)L_{\cB}\sum_{i=t+1}^k\alpha_i\right)\frac{3\ell_{f,\cB}\ell_{\cB}\alpha_t}{n}. 
\]
Finally, using the bound $\sup_{\vz}\E_{\cA}[f(\vx_k,\vz)-f(\vx'_k,\vz)]\leq\ell_{f,\cB}\delta_k$ and Lemma \ref{le:general-stability}, we complete the proof.
\end{proof}
Now, we make some remarks on Theorem \ref{thm:alg-stability}. 
\begin{remark}
\begin{enumerate}
    \item When $\phi(\vx)=\frac{1}{2}\norm{\vx}^2$, we use the notation $\bar L_{\cB}=\sup_{\vx\in\cB,\vy\in\cB,\vz}\frac{\norm{\nabla f(\vx,\vz)-\nabla f(\vy,\vz)}}{\norm{\vx-\vy}}$, $\bar\ell_\cB=1$.
    If $\alpha_k\leq\frac{c}{k}$ for some constant $c$, the right-hand side of \eqref{eq:stability-bound} recovers the bound in \cite[Theorem 3.12]{hardt2016train}. 
    \item If we choose the polynomial kernel from Proposition \ref{DNN-subpro-solution-prop} with $\delta=1$ and $r=4$, i.e. $\phi(\vx)=\frac{1}{2}\norm{\vx}^2+\frac{1}{4}\norm{\vx}^4$, then $\nabla^2\phi(\vx)^{-1}=\frac{1}{1+\norm{\vx}^2}I-\frac{2\vx\vx^T}{(1+3\norm{\vx}^2)(1+\norm{\vx}^2)}$. Thus, we have $\ell_{\cB}=\sup_{\vx\in\cB}\norm{\frac{1}{1+\norm{\vx}^2}I-\frac{2\vx\vx^T}{(1+3\norm{\vx}^2)(1+\norm{\vx}^2)}}\leq1$. When $\cB$ is a convex compact set, by the intermediate theorem, we have
    \[
    L_{\cB}=\ell_{f,\cB}\sup_{\vx\in\cB}\left\{\norm{D_{\nabla^2\phi(\vx)^{-1}}}\right\}+\ell_{\cB}\bar L_{\cB},
    \]
    where $D_{\nabla^2\phi(\vx)^{-1}}$ is the first order differential operator of $\nabla^2\phi(\vx)^{-1}$. Furthermore, if ${\rm dist}(0,\cB)\geq M>0$, then $\ell_{\cB}\leq\frac{1}{1+M^2}<1$. Moreover by some basic algebraic calculations, we have $D_{\cB}:=\sup_{\vx\in\cB}\left\{\norm{D_{\nabla^2\phi(\vx)^{-1}}}\right\}\leq\left\{\begin{array}{cc}
        6\norm{\vx}(1+4\norm{\vx}^2), & \norm{\vx}\leq1, \\
        \frac{12}{1+3\norm{\vx}^2}, & \norm{\vx}>1.
    \end{array}\right.$ When $M$ is sufficiently large, which usually occurs in high-dimensional scenarios, $D_{\cB}$ becomes very small. Thus, $L_{\cB}<\bar L_{\cB}$. In summary, when a high-order polynomial kernel is employed, \eqref{vanilla-SBPG} tends to have a better stability bound compared to standard SGD.
     
\end{enumerate}
\end{remark}
}

\section{Additional Experimental Results}
\label{appendix:additional exp}
{In this section, we provide further experimental results and assess the performance of MSBPG on more recent neural network architectures, including ConvNext~\citep{liu2022convnet} and the Vision Transformer (ViT)~\citep{dosovitskiy2020image}. The networks are trained on the CIFAR-100 dataset for 200 epochs, utilizing a cosine annealing learning rate schedule~\citep{loshchilov2016sgdr}, with a batch size of 128. The test accuracies for different optimizers, including MSBPG, SGD, Adam, and AdamW, are reported in Table \ref{tab:test_acc}. Notably, MSBPG achieved slightly higher test accuracy compared to the other methods. This advantage can be attributed to the Bregman proximity model used in our approach.
\begin{table}[ht]
\centering
\begin{tabular}{l|c|c|c|c}
\hline
Method & MSBPG & SGD & Adam & AdamW \\ \hline 
ViT Tiny & 62.37 & 60.65 & 56.88 & 58.77 \\ \hline
ViT Small & 63.59 & 61.90 & 57.66 & 59.39 \\ \hline
ConvNext Atto & 75.98 & 74.46 & 73.76 & 75.29 \\ \hline
ConvNext Femto & 76.47 & 75.96 & 74.35 & 75.65 \\ \hline
\end{tabular}
\caption{Test accuracy (\%) of different optimizers on CIFAR-100 dataset.}
\label{tab:test_acc}
\end{table}

To further evaluate the performance of MSBPG, we applied it to Transformer-XL~\citep{dai2019transformer}, a model designed for long-sequence tasks. We followed the official configuration to train the Transformer-XL-based model on the WikiText-103 dataset~\citep{merity2016pointer}, a large-scale word-level language modeling benchmark that involves long-term dependencies. The performance of MSBPG, along with that of SGD, Adam, and AdamW, was measured by the test perplexity after 50,000 training steps. The results are presented in Table \ref{tab:transformer_ppl}. 

\begin{table}[ht]
\centering
\begin{tabular}{l|c|c|c|c}
\hline
 & MSBPG & SGD & Adam & AdamW \\ \hline 
Transformer-XL & 31.07 & 33.81 & 33.53 & 32.17 \\ \hline
\end{tabular}
\caption{Test perplexity (lower is better) for Transformer-XL-based model on the WikiText-103 dataset.}
\label{tab:transformer_ppl}
\end{table}

We also plot the gradient norm of the objective function, which indicates the stationarity of the minimization problem when the nonsmooth term is absent, as shown in Figure \ref{gradient-norm-figure}. From the figure, we observe that our method, MSBPG, successfully reaches a stationary point, similar to other optimization methods. It is important to highlight that our method theoretically converges to the stationary point without requiring Lipschitz smoothness, while the convergence of traditional methods relies on this property. With careful tuning of the learning rate, methods such as SGD, Adam, and AdamW can also achieve stationary points, but they are more sensitive to learning rate choices compared to MSBPG. This sensitivity is further demonstrated experimentally in Figure \ref{fig:robust}, and is partially attributed to the absence of Lipschitz smoothness, which increases the sensitivity to the learning rate. 

\begin{figure*}[th]
              \begin{center}
                     \setlength{\tabcolsep}{0.0pt}  
                     \scalebox{1}{\begin{tabular}{cc}
            \includegraphics[width=0.5\linewidth]{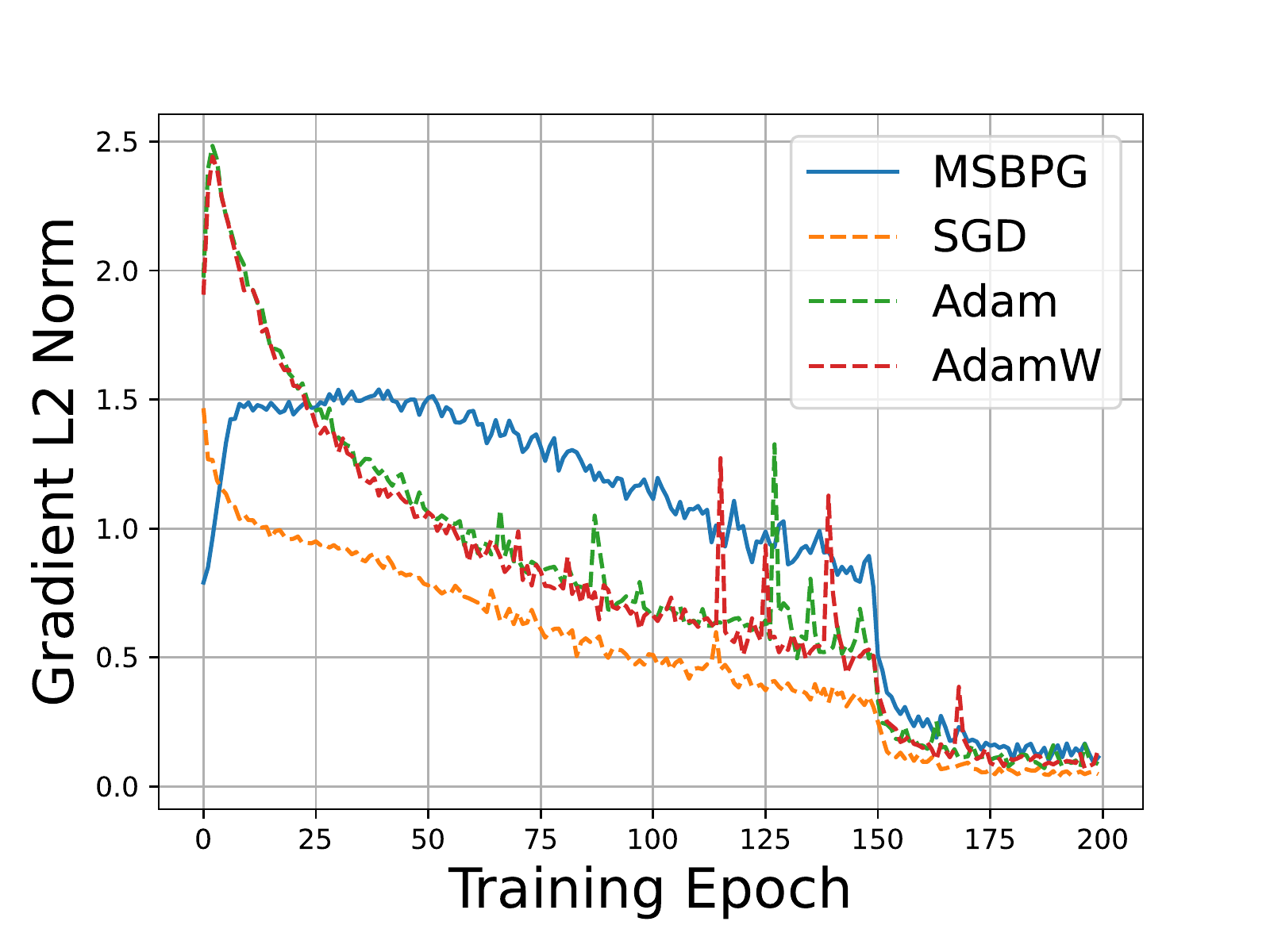}&
            \includegraphics[width=0.5\linewidth]{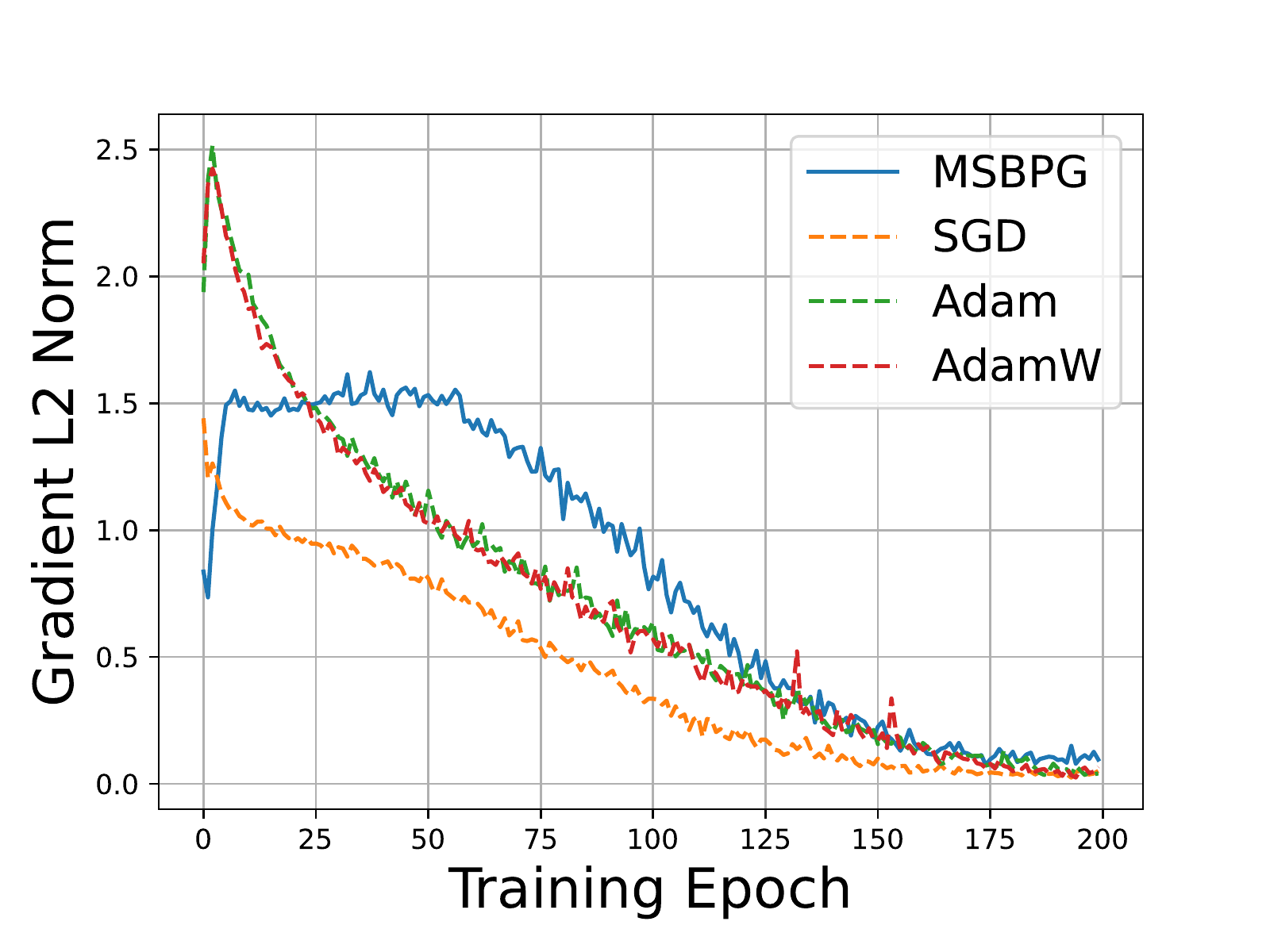}\\
            \footnotesize{(a) Step decay}
            &  \footnotesize{(b) Cosine annealing}
            \\
                     \end{tabular}}
              \end{center}
       \vspace{-0.8em}
              \caption{Gradient norm of the objective function. When the nonsmooth term is absent, the gradient norm can imply stationarity of the minimization problem.}
\label{gradient-norm-figure}
\end{figure*}
}

\section{Additional Figures}\label{app-fig}
\begin{figure}[th]
\centering
\subfigure[Illustration of SBPG and SGD updates]{
\includegraphics[width=5.5cm,height=4.5cm]{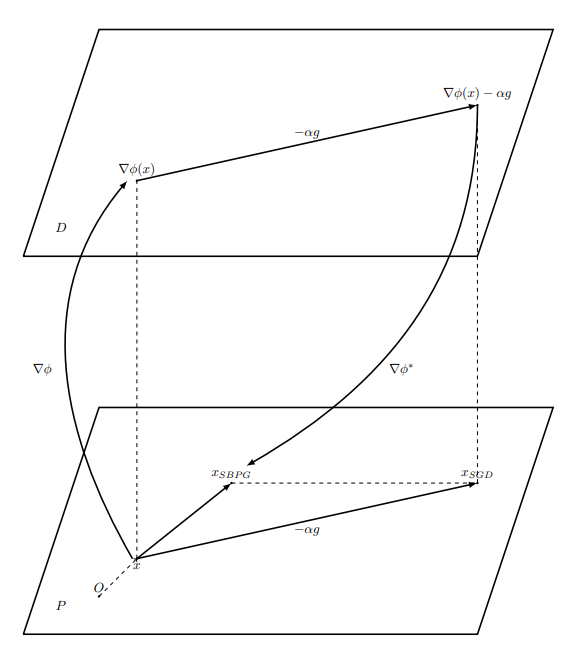}
}
\subfigure[Different initial points for SBPG]{
\includegraphics[width=5.5cm]{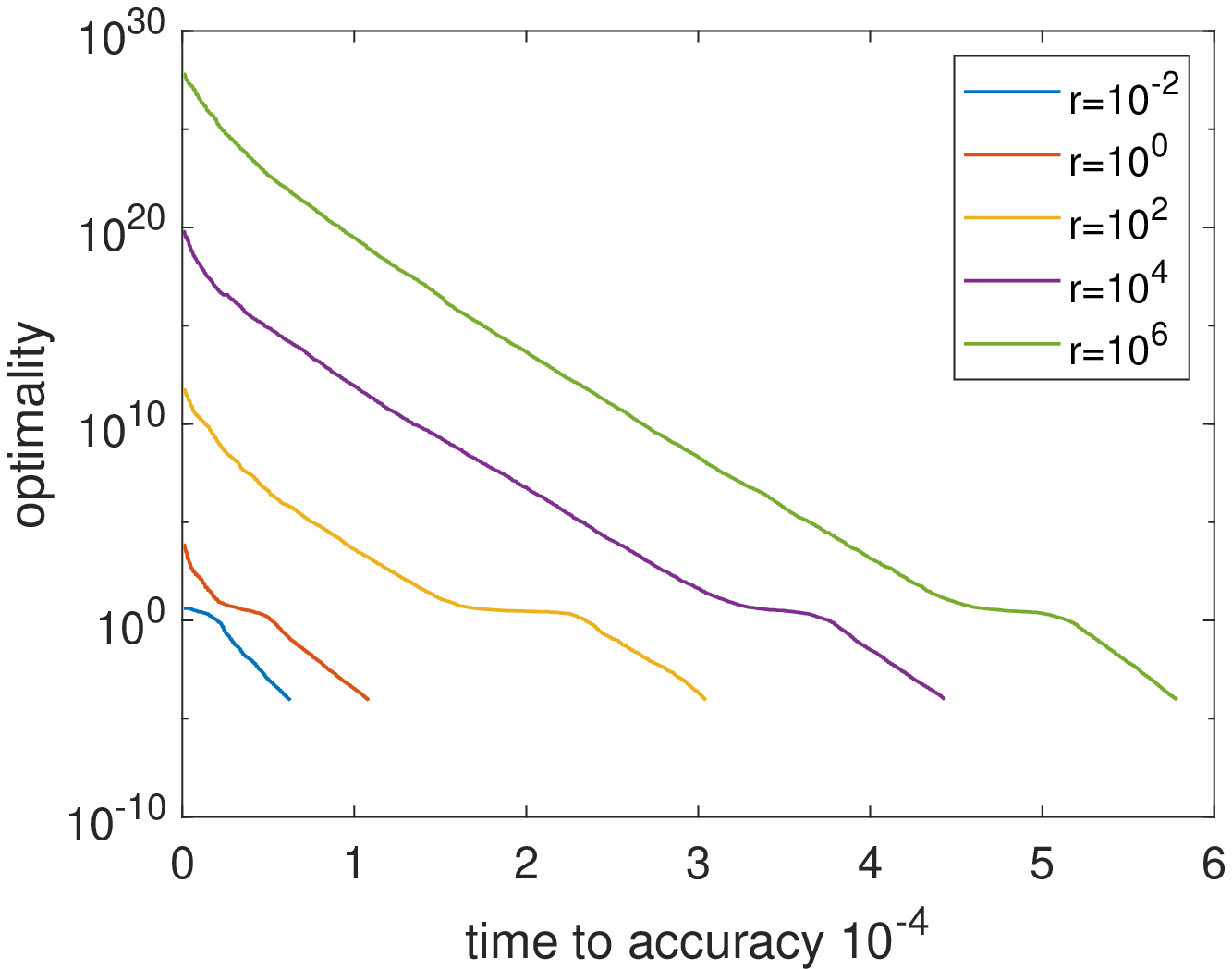}
}
\caption{Figure (a) depicts SGD and SBPG updates. SBPG includes a "pull back" mechanism that prevents the point from moving excessively in any given direction. $"P"$ and $"D"$ refer to the primal and dual spaces, respectively, and these terms are commonly used in the mirror descent method literature (see, e.g., \cite{bubeck2015convex,nemirovskij1983problem}). Figure (b) illustrates the effect of {choosing the initial point from a ball of radius $r$
on the SBPG when $r$ changes} for 
the QIP example with $d=100$ and $n=5000$. All initial step sizes are set to $1\times10^{-3}$. As shown in Figure (b), even for an initial point that is far from the optimal point, SBPG can pull back the iterates to the optimal point. }
\label{additional-figure}
\end{figure}

\vskip 0.2in
\bibliography{SBPG}

\end{document}